\documentclass{amsart}
\usepackage[utf8]{inputenc}
\usepackage[english]{babel}
\usepackage[T1]{fontenc}

\usepackage{amsmath}
\usepackage{amsaddr}
\usepackage{amsfonts}
\usepackage{amssymb}
\usepackage{amsthm} 

\usepackage{hyperref} 
\usepackage{stmaryrd} 
\usepackage{yhmath} 
\usepackage{tikz}
\usetikzlibrary{patterns} 
\usepackage{subcaption}
\usepackage{accents} 
\usepackage{xcolor} 
\usepackage{dsfont} 

\usepackage{enumitem} 

\usepackage[left=2cm,right=2cm,top=2cm,bottom=2cm]{geometry}
\usepackage{verbatim}

\DeclareMathOperator{\N}{\mathbb{N}}

\DeclareMathOperator{\R}{\mathbb{R}}

\DeclareMathOperator{\1}{\mathds{1}}

\DeclareMathOperator{\sign}{\mathrm{sgn}} 
\DeclareMathOperator{\arcosh}{\mathrm{arcosh}}

\DeclareFontFamily{U}{mathx}{}
\DeclareFontShape{U}{mathx}{m}{n}{<-> mathx10}{}
\DeclareSymbolFont{mathx}{U}{mathx}{m}{n}
\DeclareMathAccent{\widehat}{0}{mathx}{"70}
\DeclareMathAccent{\widecheck}{0}{mathx}{"71}

\usepackage[backend=biber,style=numeric]{biblatex}
\usepackage{csquotes}
\numberwithin{equation}{section}

\makeatletter
\def\blfootnote{\xdef\@thefnmark{}\@footnotetext}
\makeatother

\theoremstyle{plain}
\newtheorem{theorem}{Theorem}[section]
\newtheorem{remark}[theorem]{Remark}
\newtheorem{lemma}[theorem]{Lemma}
\newtheorem{proposition}[theorem]{Proposition}
\newtheorem{corollary}[theorem]{Corollary}

\title[Vlasov-Maxwell point charge system: spherically symmetric case]{The Vlasov-Maxwell point charge system : study of the spherically symmetric case}

\author{Emile Breton}

\address{Univ Rennes, CNRS, IRMAR - UMR 6625, F-35000 Rennes, France \\ Email address : emile.breton@univ-rennes.fr}

\date{\today}

\begin{document}



\begin{abstract}
    We consider the Vlasov-Maxwell system with a repulsive point charge. Under a spherical symmetry condition, this reduces to the relativistic Vlasov-Poisson system with a point charge. For this system, the linear characteristics solve the equations of motion for the relativistic Kepler problem. We define action-angle variables for the linearized equation, and then study the nonlinear system in these coordinates. Within this framework, we prove global existence of classical solutions for small initial data and establish their modified scattering behavior. The proof makes use of the ideas developed in \cite{Pausader_Widmayer_2021}.
\end{abstract}

\blfootnote{This work was conducted within the the France 2030 program, Centre Henri Lebesgue ANR-11-LABX-0020-01}

\maketitle

\tableofcontents
\section{Introduction}
\subsection{Physical context}

The Vlasov-Maxwell system describes the evolution of a plasma, which is a collection of charged particles, interacting through an electromagnetic field. The addition of a point charge introduces a singularity in the system, as the point charge generates its own electromagnetic field $(E_{pc},B_{pc})$ that interacts with the continuous density. Mathematically, we consider a continuous density $f$ of particles of mass $m$ and charge $e=1$, an electromagnetic field $(E,B)$, and a point charge with position and velocity $(\xi(t),\eta(t))$, mass $m_q$, and charge $q>0$. We also assume that the point charge does not self-interact (see Remark~\ref{remark_singular_fields} below). 
Finally, we consider a formal solution $(\widetilde{f},E+E_{pc},B+B_{pc})$ of the Vlasov-Maxwell system whose density is given by  $\widetilde{f}=f+q\delta(x-\xi(t))\otimes\delta(v-\eta(t))$. Under these assumptions, $(f,\xi,\eta,E,B)$ satisfies the following system in $\R\times\R^3_x\times\R^3_v$, called the Vlasov-Maxwell point charge system:

\begin{equation}
        \tag{VMpc}
        \label{VMpc}
        \begin{array}{c}
        \displaystyle \partial_t f+\widehat{v}\cdot\nabla_x f +\Big(E+\frac{1}{c}\widehat{v}\times B\Big)\cdot\nabla_v f+q\Big(E_{pc}+\frac{1}{c}\widehat{\eta}(t)\times B_{pc}\Big)\cdot\nabla_v f=0, \vspace{4pt}\\
        \begin{array}{ll}
             \displaystyle\partial_tE=c\nabla\times B -4\pi j,&\quad\nabla\cdot E=4\pi\rho,\vspace{2pt}\\
            \displaystyle\partial_t B=-c\nabla\times E,& \quad \nabla\cdot B=0.
        \end{array}\vspace{5pt}\\
        \left\{\begin{array}{l}
            \displaystyle\dot\xi(t)=\widehat{\eta}(t),  \\
             \displaystyle \dot\eta(t)=q\Big(E(t,\xi(t))+\frac{1}{c}\widehat{\eta}(t)\times B(t,\xi(t))\Big),\\
              \displaystyle(\xi(0),\eta(0))=(\xi_0,\eta_0),
        \end{array}\right.
    	\end{array}
    \end{equation}
    \begin{equation*}
    \end{equation*}
    with initial data $(f_0,E_0,B_0)$ satisfying the constraints 
    \begin{equation*}
        \nabla\cdot E_0=4\pi\int_{\R^3_v}f_0(x,v)\mathrm{d}v,\qquad\nabla\cdot B_0=0.
    \end{equation*}
Here, $\widehat v:= c^2\frac{v}{v^0}$ denotes the relativistic speed, and $v^0:=c^2\sqrt{m^2+\frac{|v|^2}{c^2}}$ represents the kinetic energy. Similarly, we define, with a slight abuse of notation, $\widehat{\eta}:=\frac{\eta}{\sqrt{m_q^2+\frac{|\eta|^2}{c^2}}}$. Finally, $\rho$ and $j$ denote, respectively, the charge and current densities, and are defined by
\begin{equation*}
    \rho(t,x):=\int_{\R^3_v} f(t,x,v)\mathrm{d}v,\qquad j(t,x):=\int_{\R^3_v} \widehat{v}f(t,x,v)\mathrm{d}v.
\end{equation*}

The main difference with the Vlasov-Maxwell system is the coupling between the point charge and the continuous density. More precisely, the Lorentz force generated by the fields $(E,B)$ acts on the particle $(\xi(t),\eta(t))$ that, in return, generates an electromagnetic field $(E_{pc},B_{pc})$. These fields are obtained by solving the Maxwell equations, with source terms
\begin{align*}
    \rho_{pc}(t,x)&=\int_{\R^3_v}\delta(x-\xi(t))\otimes\delta(v-\eta(t))\mathrm{d}v=\delta(x-\xi(t)),\\
    j_{pc}(t,x)&=\int_{\R^3_v} \frac{v}{\sqrt{m_q^2+\frac{|v|^2}{c^2}}}\delta(x-\xi(t))\otimes\delta(v-\eta(t))\mathrm{d}v=\widehat\eta(t)\delta(x-\xi(t)).
\end{align*}
In that case, the fields derive from Liénard-Wichert potentials (see \cite{Griffiths_intro_EM} on this matter) and are sometimes called Liénard-Wichert fields. Their explicit expression is given by
\begin{equation}
    \label{equation_def_field_pc}
    \begin{array}{l}
         \displaystyle E_{pc}(t,x)=\frac{|r|}{\left(|r|-\frac{1}{c}r\cdot\widehat\eta(t_r)\right)^3}\left[\left(1-\frac{|\widehat \eta(t_r)|^2}{c^2}\right)\left(\frac{r}{|r|}-\frac{\widehat  \eta(t_r)}{c}\right)+\frac{1}{c^2}r\times\left(\left(\frac{r}{|r|}-\frac{\widehat \eta(t_r)}{c}\right)\times (\widehat\eta)'(t_r)\right)\right],  \vspace{5pt}\\
         \displaystyle B_{pc}(t,x)=-\frac{1}{c}\frac{r}{|r|}\times E_{pc}(t,x),
    \end{array}
\end{equation}
where $r=x-\xi(t_r)$ and $t_r$ is called the retarded time. Physically, we consider $t_r$ to take into account the causality and the finite speed of wave propagation. Indeed, electromagnetic waves travel at the speed of light $c$ and hence when the wave reaches a point in spacetime, it was produced earlier, at said retarded time. Mathematically, for a fixed $(t,x)$, $t_r(t,x)$ is implicitly defined as the unique solution to 
\begin{equation*}
    t_r=t-\frac{1}{c}|x-\xi(t_r)|.
\end{equation*}
\begin{remark}
    For $t_r(t,x)$ to be well-defined, we need to ensure that the increasing function $g:\tau\mapsto \tau-t+\frac{1}{c}|x-\xi(\tau)|$ is nonpositive at at least one point. For instance, this will hold when the point charge does not move. However, in the general case, it is possible that even for $\tau\rightarrow -\infty$, $g(\tau)$ remains positive. Using the results of \cite{pausaderStabilityPointCharge2024}, where the analogous problem is studied for the Vlasov-Poisson system, we expect that 
    \begin{equation*}
        \eta(t)\xrightarrow[t\rightarrow\pm\infty]{}\eta_{\pm\infty}.
    \end{equation*}
    In that context, $\eta$ is bounded and hence $t_r$ is well-defined.
\end{remark}
\begin{remark}
    In \eqref{equation_def_field_pc}, the expression for $E_{pc}$ involves two terms. The first behaves like $|r|^{-2}$, does not depend on the acceleration, and corresponds to the Coulomb interaction. The second behaves like $|r|^{-1}$ and becomes dominant for large distances. Physically, it is related to electromagnetic radiation caused by the acceleration of the charged particle.
\end{remark}
\begin{remark}
    \label{remark_singular_fields}
    One can notice that $t_r(t,\xi(t))=t$, which makes the fields singular when evaluated at $(t,\xi(t))$. Physically, this means that the particle at time $t$ is affected only by the field it generates at that exact moment. Consequently, it is physically reasonable to assume that the point charge does not experience its own field. In this setting, one may view the point charge $\delta(x-\xi(t))\otimes\delta(v-\eta(t))$ as a formal solution of the Vlasov-Maxwell system.
\end{remark}
\begin{remark}
    In the non-relativistic limit $c\rightarrow + \infty$, the Vlasov-Maxwell system formally reduces to the Vlasov-Poisson point charge system
    \begin{equation}
        \tag{VPpc}
        \label{equation_vlasov_poisson_point_charge}
        \begin{array}{c}
            \displaystyle \partial_t f +v\cdot\nabla_xf +E\cdot\nabla_vf+ q \frac{x-\xi(t)}{|x-\xi(t)|^3}\cdot\nabla_v f=0,\vspace{5pt}\\
            \displaystyle E(t,x)=\int_{\R^3_y}\frac{x-y}{|x-y|^3}\rho(t,y)\mathrm{d}y,\qquad \rho(t,x)=\int_{\R^3_v}f(t,x,v)\mathrm{d}v,\vspace{5pt}\\
            \displaystyle \dot{\xi}(t)=\eta,\qquad \dot{\eta}(t)=q E(t,\xi(t)).
        \end{array}
    \end{equation}
    We refer to \cite{Pausader_Widmayer_2021,pausaderStabilityPointCharge2024,chaturvediLinearNonlinearPhase2026,kepka_widmayer_2025} for the study of this system. In what follows, we normalize the speed of light to $c=1$.
\end{remark}
We now consider the spherically symmetric case, that is,
\begin{equation*}
    f_0(x,v)=f_0(r,u,\ell),\qquad E_0(x)=E_0(r),\qquad B_0(x)=0, \qquad (\xi_0,\eta_0)=(0,0),
\end{equation*}
where $r$ is the spatial radius, $u$ is the radial momentum, and $\ell$ is the square of the angular momentum. They are defined by
\begin{equation*}
    r=|x|,\qquad u=\frac{v\cdot x}{|x|},\qquad \ell=|x\times v|^2.
\end{equation*}

Under the assumption of spherical symmetry, the magnetic field $B$ vanishes, and the electric field $E$ becomes
\begin{equation*}
    E(t,r):=-\partial_r\Psi(t,r)=\frac{M(t,r)}{r^2},
\end{equation*}
where $M$ is the enclosed mass and $\Psi$ is the potential, defined by
\begin{equation*}
    M(t,r):=4\pi^2\int_0^r\int_0^\infty\int_{\R} f(t,s,u,\ell)\mathrm{d}u\mathrm{d}\ell\mathrm{d}s,\qquad \Psi(t,r):=-4\pi^2\int_0^\infty\int_0^\infty\int_{\R} \frac{f(t,s,u,\ell)}{\max(r,s)}\mathrm{d}u\mathrm{d}\ell\mathrm{d}s.
\end{equation*}
Hence, due to the uniqueness of the solution to the Cauchy problem, the point charge does not move, and we have $(\xi(t),\eta(t))=0$. Moreover, under this symmetry assumption, the Vlasov-Maxwell system becomes the relativistic Vlasov-Poisson system \cite{Horst_1990}. Finally, for simplicity, we assume that the point charge has mass $m_q=1$ and charge $q=1$. In this setting, the system \eqref{VMpc} becomes the relativistic Vlasov-Poisson point charge system with spherical symmetry. It is given by
\begin{equation}
\tag{RVPpc}
    \label{equation_VM_pc_radial_case}
    \partial_tf+\frac{u}{\sqrt{u^2+m^2+\frac{\ell}{r^2}}}\partial_rf+\left(\frac{\ell}{r^3\sqrt{u^2+m^2+\frac{\ell}{r^2}}}+\frac{M(t,r)+1}{r^2}\right)\partial_u f=0.
\end{equation}

\begin{remark}
    \label{remark_t_positive}
    For $t_r$ and $(E_{pc},B_{pc})$ to be well-defined, we normally require the time $t$ to be considered on $\R$. As stated above, this is due to the finite speed of propagation of electromagnetic waves. However, in this case, since the point charge does not move, the field $E_{pc}$ is constant. This, in particular, holds for $t\rightarrow -\infty$, meaning that the field behaves as if the propagation speed is infinite. In that case, we can consider equation \eqref{equation_VM_pc_radial_case} for nonnegative times only. 
\end{remark}

\subsection{Previous results}

To the best of our knowledge, this is the first time the Vlasov-Maxwell point charge system is introduced and studied.\\
In the absence of a continuous density $f$, when there are multiple point charges, the system becomes the Klimontovich equation \cite{klimontovich_1967} (also see \cite{Griffiths_intro_EM} for a derivation of the microscopic fields). This system is normally ill-posed due to the singular nature of $(E_{pc}(t,\xi(t)),B_{pc}(t,\xi(t)))$ \cite{kiessling_2012} (see, for instance, \cite{elskens_Kiessling_2020} for discussions on the subject). As said in Remark \ref{remark_singular_fields}, we avoid this singularity by assuming that there is no self-interaction for the point charge. \\

In the absence of a point charge, we recover the Vlasov--Maxwell system. While this system has been extensively studied, the global existence of classical solutions for large initial data remains an open problem, though various continuation criteria have been proved (see, e.g., \cite{Luk_Strain_14}). In 3D, the analysis of small data solutions was initiated by Glassey and Strauss \cite{Glassey_Strauss_1987}. For compactly supported data, they proved the existence of a global solution and decay estimates on the fields. The compactness assumption on the initial data was recently relaxed using vector field methods \cite{Bigorgne_sharp_2020}, as well as by combining these techniques with the Fourier method \cite{Wang_2022}. The smallness hypothesis on the initial fields was also removed in \cite{wei_yang_2021}, where they proved the global existence of solutions for arbitrarily large initial fields. More recent results have focused on establishing the scattering behavior of solutions. In particular, for small data solutions, a modified scattering dynamic was exhibited for compactly supported data \cite{pankavich_ben-artzi_2025,breton_modified_2026}, and without the compactness assumption \cite{bigorgne_modified_2025}. It was also proved in \cite{Breton_2025_absence_linear} that linear scattering is a non-generic phenomenon.  Finally, a wave operator and a scattering map were obtained \cite{bigorgne_ScatteringMap_2023}.

When the magnetic field vanishes, the Vlasov-Maxwell system becomes the relativistic Vlasov-Poisson equation. This holds for solutions with spherically symmetric data \cite{Horst_1990}. In this context, the asymptotic behavior of small data solutions was recently derived \cite{pankavich_2021}. \\

The Vlasov-Maxwell point charge system shares similarities with the Vlasov-Poisson point charge system, or plasma-charge model, first established and studied by Caprino and Marchioro in \cite{caprino_marchioro_2010}. Note that, formally, the Vlasov-Poisson point charge system is the non-relativistic limit of the Vlasov-Maxwell point charge system.  This system arises when the solution to the Vlasov-Poisson system is composed of a continuous density and point charges, modeled by Dirac masses. In the absence of a continuous density, due to the singularity of the Coulomb law, we still need to assume that there is no self-interaction for the point charges. In that case, they satisfy a Newtonian N-body problem. In the absence of point charges, we recover the Vlasov-Poisson system. For this equation, results similar to those for the Vlasov-Maxwell system hold. For small data, modified scattering \cite{choiModifiedScatteringVlasov2016, ionescuAsymptoticBehaviorSolutions2022, pankavichAsymptoticDynamicsDispersive2022} and absence of linear scattering \cite{choiAsymptoticBehaviorNonlinear2011} have been proved. Furthermore, a scattering map was also derived in \cite{flynnScatteringMapVlasov2023}. Finally, \cite{bigorgneHomeomorphicModifiedWave2026} showed that the wave operator is a homeomorphism, and that the scattering states and initial data lie in the same Banach space.\\
For the full Vlasov-Poisson point charge system, following the works of Caprino and Marchioro, the existence of global solutions was then obtained for the attractive case in 2D \cite{caprinoAttractivePlasmaChargeSystem2012} and for the 3D repulsive case \cite{marchioroCauchyProblem3D2011}. Both results, however, require the initial plasma density to be supported away from the point charge. In 3D, the existence of global weak solutions was also proved using Lions and Perthame theory \cite{desvillettesPolynomialPropagationMoments2015,wuPolynomialPropagationMoments2021}.  
Note that recent results also exist in convex domains, with a study of boundary effects \cite{wuPlasmachargeModelConvex2024,wuPlasmaChargeModelBoundary2026}.\\

Recently, with a single species and one point charge \eqref{equation_vlasov_poisson_point_charge}, a new method was introduced. It exploits the Hamiltonian structure and integrability of the linearized equation. In this case, the characteristics of the linearized equation solve a Kepler ODE. The key idea is to introduce action-angle variables via a symplectic diffeomorphism that, in particular, straightens the linear dynamics. More precisely, by composing a solution to the linearized equation with this diffeomorphism, one recovers a solution to the free transport equation. This leads to a new nonlinear problem expressed in action-angle variables, which provides a more appropriate framework for the analysis. This proof strategy was first introduced by Pausader and Widmayer in \cite{Pausader_Widmayer_2021}, where they studied a restrictive radial case. In that context, they proved the existence of a global solution that exhibits a modified scattering dynamic. Pausader, Widmayer and Yang later extended this result to the general case \cite{pausaderStabilityPointCharge2024}. Action-angle variables were also used to study the attractive case. When a static point charge is located at the origin, which occurs in the spherically symmetric setting, its effect is modeled by an external potential. In this setting, there are trapped trajectories that stay in a bounded subset of the phase space and remain away from the origin, preventing any dispersion result. In \cite{chaturvediLinearNonlinearPhase2026}, Chatuverdi and Luk considered small initial data supported on trapped trajectories and proved a long time nonlinear stability and phase mixing result in spherical symmetry. In another setting, small data solutions were shown to be global by Kepka and Widmayer \cite{kepka_widmayer_2025}. More precisely, they studied solutions with initial data supported on hyperbolic trajectories, and proved that they exhibit a modified scattering dynamic.\\

In this paper, we adapt these tools to the relativistic setting. 

\subsection{Main result}
We begin by stating our main theorem. The result we prove is in fact more precise and better stated in action-angle variables (see Theorem \ref{main_theorem_action_angle} below). 

\begin{theorem}
    \label{main_theorem}
    Let $\mathcal{W}$ be a set of spherically symmetric functions defined by 
    \begin{equation}
        \mathcal{W}:=\{f\in L^1(\R^*_+\times\R\times \R_+)\,|\, \partial_rf\in L^1,\,\partial_u f\in L^1\}.
    \end{equation}
    There exists $\varepsilon^* > 0$ such that for any $0 < \varepsilon_0 < \varepsilon^*$ and any $f_0 \in \mathcal{W}$ satisfying
    \begin{equation}
        \|(r^{-65}+r^{35}+|u|^{35}+\ell^{45})(f_0+|\partial_r f_0|+|\partial_uf_0|)\|_{L^1}\leq \varepsilon_0,
    \end{equation}
    there exists a global solution $f \in C^1_t L^1 \cap C^0_t \mathcal{W}$ to \eqref{equation_VM_pc_radial_case}, arising from the initial data $f(t=0,\cdot)=f_0$. Moreover, this solution exhibits a modified scattering dynamic. Namely, there exist $\mathcal{R}, \mathcal{U}$ and $f_\infty \in L^1$ such that
    \begin{equation*}
    f(t,\mathcal{R}(t,r,u,\ell),\mathcal{U}(t,r,u,\ell),\ell) \xrightarrow[t \to +\infty]{L^1(\R^*_+\times\R\times \R_+)} f_\infty(r,u,\ell).
    \end{equation*}
\end{theorem}

\begin{remark}
    Here, the modified characteristics $(\mathcal{R},\mathcal{U})$ asymptotically satisfy
    \begin{align*}
        \mathcal{R}(t,r,u,\ell)=&~t\frac{\sqrt{\left(\sqrt{m^2+u^2+\frac{\ell}{r^2}}+\frac{1}{r}\right)^2-m^2}}{\sqrt{m^2+u^2+\frac{\ell}{r^2}}+\frac{1}{r}}+\frac{m^2}{\left(\left(\sqrt{m^2+u^2+\frac{\ell}{r^2}}+\frac{1}{r}\right)^2-m^2\right)\left(\sqrt{m^2+u^2+\frac{\ell}{r^2}}+\frac{1}{r}\right)}\log(t)\\
        &-\frac{m^2\mathcal{E}_\infty\left(\sqrt{\left(\sqrt{m^2+u^2+\frac{\ell}{r^2}}+\frac{1}{r}\right)^2-m^2}\right)}{\left(\left(\sqrt{m^2+u^2+\frac{\ell}{r^2}}+\frac{1}{r}\right)^2-m^2\right)\left(\sqrt{m^2+u^2+\frac{\ell}{r^2}}+\frac{1}{r}\right)}\log(t)+s(t,r,u,\ell),\\
        |\mathcal{U}(t,r,u,\ell)|=&~\sqrt{\left(\sqrt{m^2+u^2+\frac{\ell}{r^2}}+\frac{1}{r}\right)^2-m^2}-\frac{1}{t}\frac{\left(\sqrt{m^2+u^2+\frac{\ell}{r^2}}+\frac{1}{r}\right)^2}{\left(\sqrt{m^2+u^2+\frac{\ell}{r^2}}+\frac{1}{r}\right)^2-m^2}+d(t,r,u,\ell),
    \end{align*}
    where $\mathcal{E}_\infty$ is an asymptotic field, defined in Proposition \ref{proposition_asymp_E_psi} below. Moreover, for any fixed $(r,u,\ell)$, $s$ and $d$ satisfy
    \begin{equation*}
        s(\cdot,r,u,\ell)=O_{t\rightarrow+\infty}(1),\qquad d(\cdot,r,u,\ell)=O_{t\rightarrow +\infty}\left(t^{-2}\log(t)\right).
    \end{equation*}
    Although these asymptotic expansions remain slightly convoluted, they give the general asymptotic behavior of the characteristics. More precisely, in the expression for $\mathcal{R}$, the first two terms correspond to the linear characteristics, whereas the term containing $\mathcal{E}_\infty$ arises purely from the nonlinearity. The exact and simpler expression for the correction is given in action-angle variables in \eqref{equation_exact_modified_scattering_f}.
\end{remark}



\subsection{Ideas of the proof}
\label{section_idea_proof}

We first consider the linearized system 
\begin{equation}
    \label{equation_linearized_system_idea}
    \partial_tf+\frac{u}{\sqrt{u^2+m^2+\frac{\ell}{r^2}}}\partial_rf+\left(\frac{\ell}{r^3\sqrt{u^2+m^2+\frac{\ell}{r^2}}}+\frac{1}{r^2}\right)\partial_u f=0.
\end{equation}
Here, the characteristics are solutions to the equations of motion appearing in the relativistic Kepler problem, also called classical relativistic two body equation (see, for instance, \cite{onemSolutionsClassicalRelativistic1998}). Moreover, $\ell$ is a constant of the motion. Hence, if we view $\ell$ as a parameter, \eqref{equation_linearized_system_idea} is a $1+1$ Hamiltonian system and can be rewritten
\begin{equation*}
    \partial_t f=\{\mathcal{H}_0,f\},
\end{equation*}
where $\mathcal{H}_0(r,u,\ell):=\sqrt{u^2+m^2+\frac{\ell}{r^2}}+\frac{1}{r}$ is the Hamiltonian of the system, and the Poisson bracket is defined by 
\begin{equation*}
    \{f,g\}:=\partial_r f\partial_ug-\partial_uf\partial_r g.
\end{equation*}
Since the study of \eqref{equation_linearized_system_idea} is more difficult than the free relativistic transport equation, we exploit its structure to introduce action-angle variables. More precisely, we define a new set of coordinates $(\theta,a)\in \R\times\R^*_+$, called action-angle variables. They are defined via a symplectomorphism $(\Theta,\mathcal{A})$  that maps $(r,u)$ onto $(\theta,a)$, and with inverse $(R,U)$. For $g(t,\theta,a,\ell):=f(t,R(\theta,a),U(\theta,a),\ell)$, and $\mathcal{H}_0=a^0=\sqrt{m^2+a^2}$, we have
\begin{equation}
    \partial_t g =\{\mathcal{H}_0,g\}_{\theta,a}=-\widehat{a}\partial_\theta g,
\end{equation}
where $\{f,g\}_{\theta,a}:=\partial_\theta f\partial_ag-\partial_\theta g\partial_a f$. This equation can be solved directly by composing with the linear flow $\theta+t\widehat{a}$. In fact, if $f$ is a solution to the linearized system \eqref{equation_linearized_system_idea}, and
\begin{equation}
    \label{equation_def_gamma_idea}
    \gamma(t,\theta,a,\ell):=f(t,R(\theta+t\widehat{a},a),U(\theta+t\widehat{a},a),\ell),
\end{equation}
then $\partial_t\gamma=0$. Since the study of the linearized system is better adapted to action-angle variables, we also rewrite the nonlinear system in these coordinates. Let $f$ be a solution to the nonlinear system \eqref{equation_VM_pc_radial_case} and $\gamma$ be defined as in \eqref{equation_def_gamma_idea}. Then $\gamma$ solves the nonlinear problem given by 
\begin{equation}
    \label{equation_non_linear_aa_idea}
    \partial_t\gamma=\{\widetilde{\Psi},\gamma\}_{\theta,a},
\end{equation}
where $\Psi$ is the potential, $\widetilde{R}(\theta,a):=R(\theta+t\widehat{a},a)$, and 
\begin{equation}
    \Psi(t,r)=-4\pi^2\iiint \frac{1}{\max(r,\widetilde{R}(\vartheta,\alpha))}\gamma(t,\vartheta,\alpha,\ell)\mathrm{d\vartheta}\mathrm{d}\alpha\mathrm{d}\ell,\qquad \widetilde{\Psi}(t,\theta,a):=\Psi(t,\widetilde{R}(\theta,a)).
\end{equation}
Note that the integral is written using action-angle variables. One can do the same with $M$ and its derivative $\rho$, and find
\begin{align}
    \label{equation_def_M_idea} M(t,r)&=4\pi^2\iiint \1_{\{\widetilde{R}(\theta,a)\leq r\}} \gamma(t,\theta,a,\ell)\mathrm{d}\theta\mathrm{d}a\mathrm{d}\ell,\\
    \label{equation_def_rho_idea}\rho(t,r)&:=\partial_r M(t,r)=4\pi^2\iiint \delta\big(\widetilde{R}(\theta,a)-r\big) \gamma(t,\theta,a,\ell)\mathrm{d}\theta\mathrm{d}a\mathrm{d}\ell.
\end{align}
Also note that the derivatives of $\widetilde{\Psi}$ satisfy, for any $\alpha,\beta\in\{\theta,a\}$,
\begin{equation}
    \label{equation_def_derivatives_psi_idea}
    \partial_\alpha\widetilde{\Psi}=-\partial_\alpha\widetilde{R}\frac{M(t,\widetilde{R})}{\widetilde{R}^2},\qquad \partial_\alpha\partial_\beta\widetilde{\Psi}=-\frac{M(t,\widetilde{R})}{\widetilde{R}^2}\left(\partial_\alpha\partial_\beta \widetilde{R}-2\frac{\partial_\alpha\widetilde{R}\partial_\beta\widetilde{R}}{\widetilde{R}}\right)-\rho(t,\widetilde{R})\frac{\partial_\alpha\widetilde{R}\partial_\beta\widetilde{R}}{\widetilde{R}^2}.
\end{equation}

Now that we have defined these quantities in action-angle variables, let us study the nonlinear problem \eqref{equation_non_linear_aa_idea}. One of the key arguments on which the proof is based is the introduction of the \textbf{bulk} $\mathcal{B}_t$. This subset of $\R^*_+\times\R\times\R_+$ is defined by
\begin{equation}
    \label{equation_def_bulk_idea}
    \mathcal{B}_t:=\left\{(\theta,a,\ell)\in\R\times\R^*_+\times\R_+\,|\, \widehat{a}\geq t^{-\frac{1}{4}},\quad a\leq t^\frac{1}{3},\quad |\theta|\leq \frac{t\widehat{a}}{2},\quad \ell \leq t^\frac{1}{2}\right\}.
\end{equation}
Note that for large times, $(\theta,a,\ell)\in\mathcal{B}_t$. Moreover, in $\mathcal{B}_t^c$, 
\begin{equation}
    \label{equation_complement_bulk_idea}
    \1_{\mathcal{B}_t^c}\lesssim t^{-k}(a^{3k} +a^{-4k}+|\theta|^k(1+a^{-k}) +\ell^{2k}).
\end{equation}
Hence, for sufficiently localized solutions, we can recover any time decay in $\mathcal{B}_t^c$. In the bulk, we can also derive $\widetilde{R}(\theta,a)\sim t\widehat{a}$, giving us a rough estimate of the asymptotic behavior of $\widetilde{R}$. In fact, in a subset of the bulk, we show that 
\begin{equation*}
    |\widetilde{R}(\theta,a)-t\widehat{a}|=O_{t\rightarrow +\infty}(t^\frac{1}{2}\log(t)).
\end{equation*}
In particular, this measures the rate at which the particles travel to infinity. With the introduction of $\mathcal{B}_t$, we can readily estimate the enclosed mass, defined in \eqref{equation_def_M_idea}. To do so, we divide the integral between $\mathcal{B}_t$ and $\mathcal{B}_t^c$. In the bulk, we use $\widetilde{R}(\theta,a)\sim t\widehat{a}$ to recover the decay in $t$, and in $\mathcal{B}_t^c$ we use \eqref{equation_complement_bulk_idea} to recover any decay in $t$. A similar argument can be used for the charge density $\rho$. We obtain, for any $k>0$,
\begin{equation*}
    \frac{M(t,r)}{r^k}\lesssim t^{-k} \mathcal{M},\qquad \frac{|\rho(t,r)|}{r^k}\lesssim t^{-1-k}\mathcal{M},
\end{equation*}
where 
\begin{equation*}
    \mathcal{M}:=\|(a^{24}+a^{-24}+|\theta|^{8}+\ell^{16})(\gamma+|\partial_\theta \gamma|+|\partial_a \gamma|)\|_{L^1},
\end{equation*}
is constant in the linearized setting. By \eqref{equation_def_derivatives_psi_idea}, we can then use these estimates to derive 
\begin{equation}
    \label{equation_estimate_derivatives_psi_idea}
    |a^{-1}\partial_\theta\widetilde{\Psi}(\theta,a)|\lesssim t^{-\frac{3}{2}}\mathcal{M},\qquad |\partial_a\widetilde{\Psi}(\theta,a)|\lesssim t^{-1}(a^3+a^{-3}+|\theta|+\ell^2)\mathcal{M},
\end{equation}
and similar estimates for the second order derivatives (see Proposition \ref{proposition_estimates_psi_second_derivatives} below). By commuting \eqref{equation_non_linear_aa_idea} with the derivatives and using the previous estimates, we propagate the moments of $\gamma$ and its derivatives (see Propositions \ref{proposition_bootstrap_gamma}--\ref{proposition_bootstrap_gamma_derivatives}) and obtain that the small data solutions to \eqref{equation_non_linear_aa_idea} are global. \\
We can thus investigate further the asymptotic behavior of such solutions and, in particular, try to prove a scattering statement for $\gamma$. Note that in the definition \eqref{equation_def_gamma_idea} of $\gamma$, we already composed by the linear flow. Hence, to prove a scattering statement, we may directly study the integrability of $\partial_t\gamma$. Recall that
\begin{equation*}
    \partial_t\gamma=\partial_\theta\widetilde{\Psi}\gamma_a-\partial_a\widetilde{\Psi}\gamma_\theta.
\end{equation*}
From \eqref{equation_estimate_derivatives_psi_idea}, we merely obtain $|\partial_t\gamma|\lesssim t^{-1}\mathcal{M}$. Moreover, the lack of decay is concentrated in the electric field $\partial_a\widetilde{\Psi}$. We then derive the leading-order term in the asymptotic expansion of $\partial_a\widetilde{\Psi}$, and find that it is independent of $\theta$ and $\ell$. More precisely
\begin{equation*}
    \partial_a\widetilde{\Psi}=-\frac{1}{t}\frac{m^2}{a^0a^2}\mathcal{E}_\infty(a)+O(t^{-\frac{6}{5}}),\qquad \mathcal{E}_\infty(a):=\lim_{t\rightarrow+\infty} 4\pi^2\iiint\1_{\{\alpha\leq a\}}\gamma(t,\theta,\alpha,\ell)\mathrm{d}\theta\mathrm{d}\alpha\mathrm{d}\ell.
\end{equation*}
This allows us to define our main theorem in action-angle variables.
\begin{theorem}
    \label{main_theorem_action_angle}
    There exists $\varepsilon>0$ such that, for all $0<\varepsilon_0<\varepsilon$, if 
        \begin{equation}
            \label{equation_smallness_condition_idea}
            \|(a^{30}+a^{-30}+|\theta|^{10}+\ell^{20})(\gamma+|\partial_a\gamma|+|\partial_\theta\gamma|)\|_{L^1_{\theta,a,\ell}}\leq \varepsilon_0,
        \end{equation}
        then, there exists a global solution $\gamma\in C^1_tL^1_{\theta,a,\ell}\cap C^0_tW^{1,1}_{\theta,a}$ with initial data $\gamma_0$. Moreover, there exists $\mathcal{E}_\infty\in L^\infty_a$, and $\gamma_\infty\in L^1_{\theta,a,\ell}$ such that
        \begin{equation}
            \label{equation_scattering_gamma_idea}
            \gamma\left(t,\theta-\log(t)\frac{m^2}{a^0a^2}\mathcal{E}_\infty(a),a,\ell\right)\xrightarrow[t\rightarrow +\infty]{L^1_{\theta,a,\ell}} \gamma_\infty(\theta,a,\ell).
        \end{equation}
\end{theorem}

Finally, we can return to the initial nonlinear problem \eqref{equation_VM_pc_radial_case}. We consider $f_0$ such that 
\begin{equation*}
    \gamma_0:=f_0(R(\theta,a),U(\theta,a),\ell),
\end{equation*}
satisfies the smallness condition \eqref{equation_smallness_condition_idea}. Due to Theorem \ref{main_theorem_action_angle}, there exists a global solution $\gamma$ to \eqref{equation_non_linear_aa_idea}. We then define
\begin{equation*}
    f(t,r,u,\ell):=\gamma(t,\Theta(r,u)-t\widehat{\mathcal{A}}(r,u),\mathcal{A}(r,u),\ell).
\end{equation*}
Consequently, $f$ is a global solution to \eqref{equation_VM_pc_radial_case} with initial data $f_0$. Finally, by \eqref{equation_scattering_gamma_idea}, $f$ exhibits a modified scattering dynamic. 

\subsection{Structure of the paper}

We begin by studying the linearized system. In Section \ref{section_generating_function}, we explain how to find the action-angle variables and provide estimates on the generating functions. This lets us define the action-angle variables in Section \ref{section_action_angle} and show primary estimates on the radius $R$. We then define new quantities in action-angle coordinates and introduce the bulk $\mathcal{B}_t$ in Section \ref{section_composing_linear_flow}. In Section \ref{section_preliminary_lemma} we state an essential lemma to estimate $\rho$. In Section \ref{section_potential_derivatives}, we give key upper bounds for $M,\rho$ and the derivatives of $\widetilde{\Psi}$. In Section \ref{section_non_linear_problem}, we tackle the nonlinear problem. We propagate moments for $\gamma$ and its derivatives and prove that the solutions are global. Section \ref{section_asymptotics} investigates the asymptotic behavior of the solutions. Here we determine the exact asymptotic behavior of $\partial_a\widetilde{\Psi}$ and prove the modified scattering statement. This proves Theorem \ref{main_theorem_action_angle}. We finally go back to the initial problem \eqref{equation_VM_pc_radial_case}, and prove Theorem \ref{main_theorem} in Section \ref{section_going_back}.


\section{Study of the linearized system}

Let $f$ be a solution to the linearized system

\begin{equation}
    \label{equation_RVPpc_linearized_2}
    \partial_tf+\frac{u}{\sqrt{u^2+m^2+\frac{\ell}{r^2}}}\partial_rf+\left(\frac{\ell}{r^3\sqrt{u^2+m^2+\frac{\ell}{r^2}}}+\frac{1}{r^2}\right)\partial_u f=0.
\end{equation}

As stated in Section \ref{section_idea_proof}, our goal is to straighten the characteristics. Here, these characteristics satisfy the equations of motion of the relativistic Kepler problem \cite{onemSolutionsClassicalRelativistic1998}. We introduce new coordinates, called action-angle variables, such that, in these new coordinates, the system becomes the free relativistic transport equation 
\begin{equation}
\label{equation_free_relat_transp}
    \partial_tg +\widehat{a}\partial_\theta g=0.
\end{equation}

\subsection{The generating function}
\label{section_generating_function}
\subsubsection{How to find the action-angle variables}
We begin by formally explaining how to obtain the action-angle variables. One of the ways to find the diffeomorphism that leads to such coordinates is to define a generating function $S(a,r)$ such that $\partial_rS(a,r)=u$. Here, this expression needs to be understood in the following sense. First, note that $\ell$ is a constant of the motion. Hence, if we view $\ell$ as a parameter, we know that \eqref{equation_RVPpc_linearized_2} is a $1+1$ Hamiltonian system with $\mathcal{H}_0(r,u):=\sqrt{u^2+m^2+\frac{\ell}{r^2}}+\frac{1}{r^2}$. As such, it rewrites 
\begin{equation*}
    \partial_t f=\{\mathcal{H}_0,f\},
\end{equation*}
where the Poisson bracket is defined by 
\begin{equation*}
    \{f,g\}:=\partial_r f\partial_ug-\partial_uf\partial_r g.
\end{equation*}
One of the key ideas is to preserve the Hamiltonian structure of the equation in the action-angle variables $(\theta,a)$, which are to be defined later. We introduce the new Poisson bracket 
\begin{equation*}
    \{f,g\}_{\theta,a}:=\partial_\theta f\partial_a g-\partial_\theta g\partial_a f.
\end{equation*}
Hence, since we want \eqref{equation_free_relat_transp} to be equivalent to
\begin{equation*}
    \partial_t g= \{\mathcal{H}_0,g\}_{\theta,a},
\end{equation*}
we are inclined to define the action variable $a\in\R^*_+$ as $a^0=\mathcal{H}_0(r,u)$ (recall that $\partial_a (a^0)=\widehat{a})$. By doing so, we can then express $u$ as 
\begin{equation}
    \label{equation_first_definition_u}
    u=\pm\sqrt{\left(a^0-\frac{1}{r}\right)^2-m^2-\frac{\ell}{r^2}}.
\end{equation}
We can thus define $S(a,r)$ by integrating in $r$ the term on the right-hand side of \eqref{equation_first_definition_u}, so that it automatically satisfies $\partial_r S(a,r)=u$. Then, we define the angle variable by $\theta:=\partial_aS(a,r)$. Since, by definition, $0=\mathrm{d}(\mathrm{d}S)=\mathrm{d}(\theta\mathrm{d}a+u\mathrm{d}r)$, this directly implies $\mathrm{d}\theta\wedge\mathrm{d}a=\mathrm{d}r\wedge\mathrm{d}u$ and thus $(r,u)\mapsto(\theta,a)$ preserves the volume form.

\begin{remark}
    If we do not impose any volume-preserving condition on the action-angle variables, it may be easier to find such variables. Nevertheless, it remains a difficult task to find action-angle variables without a clear guideline. However, introducing the generating function provides a robust framework to obtain volume preserving action-angle variables. In fact, this method can be applied in a more general setting, and we refer to \cite[Section 2.1.1]{pausaderStabilityPointCharge2024} for more details.
\end{remark}

We now define the exact expression for the generating function. First, note that the quantity on the right-hand side of \eqref{equation_first_definition_u} is only well-defined for $r\in(-\infty,r_-(a,\ell)]\cup[r_+(a,\ell),\infty)$ with 
\begin{equation*}
    r_-(a,\ell):=\frac{a^0-\sqrt{m^2+\ell a^2}}{a^2},\qquad r_+(a,\ell):=\frac{a^0+\sqrt{m^2+\ell a^2}}{a^2}.
\end{equation*}
Consequently, when defining $S(a,r)$, we have to choose between integrating \eqref{equation_first_definition_u} on $[r,r_-(a,\ell)]$ or $[r_+(a,\ell),r]$. However, $r_-(a,\ell)=\frac{1-\ell}{a^0+\sqrt{m^2+\ell a^2}}\leq 0$ when $\ell\geq 1$ and $r_-(a,\ell)\leq \frac{1-\ell}{2m}$ otherwise. Since we want to consider $r\in(0,\infty)$, we integrate on $[r_+(a,\ell),r]$ and define
\begin{equation*}
    S(a,r):=\int_{r_+(a,\ell)}^r\sqrt{\left(a^0-\frac{1}{s}\right)^2-m^2-\frac{\ell}{s^2}}\mathrm{d}s.
\end{equation*}

As stated above, in order to compute the angle variable $\theta$, we need to differentiate $S$ with respect to the action variable $a$. We derive
\begin{equation*}
    \partial_aS(a,r)=\int_{r_+(a,\ell)}^r \frac{a}{a^0}\frac{a^0-\frac{1}{s}}{\sqrt{\left(a^0-\frac{1}{s}\right)^2-m^2-\frac{\ell}{s^2}}}\mathrm{d}s=pG_a\left(\frac{r}{p}-\kappa\right),
\end{equation*}
where 
\begin{equation}
    \label{equation_def_G_a}
    G_a(x):=\sqrt{x^2-1}+\left(\kappa-\frac{1}{a^0p}\right)\arcosh(x),
\end{equation}
and 
\begin{equation}
    \label{equation_def_p_kappa}
    p(a,\ell):=\frac{\sqrt{m^2+\ell a^2}}{a^2},\qquad \kappa(a,\ell):=\frac{a^0}{a^2p}=\frac{a^0}{\sqrt{m^2+\ell a^2}}.
\end{equation}

\begin{remark}
    Here we make use of the notations introduced by Kepka and Widmayer in \cite{kepka_widmayer_2025}. Note that, unlike in their article, $G_a$ does not depend on $a$ solely through $\kappa$.
\end{remark}

Before giving the exact expressions of the action-angle variables, let us give some properties on $G_a$ and its inverse.

\subsubsection{Properties of $G_a$ and its inverse}
First, note that for any fixed $a>0$,  $G_a: (1,+\infty)\rightarrow (0,+\infty)$ is an increasing bijective function, with inverse $H_a:=G_a^{-1}$. We begin by giving estimates and asymptotic properties of $G_a$ and $H_a$. 
\begin{proposition}
    \label{proposition_asymptotic_expansion_G_H}
    Let $G_a$ be the function defined in \eqref{equation_def_G_a} and $H_a$ its inverse. The following estimates hold 
    \begin{equation}
        \label{equation_estimates_G_a_H_a}
        \begin{array}{r}
             \displaystyle x-1\leq G_a(x)\leq2x,\quad 1\leq x<+\infty,  \\
            \displaystyle \frac{x}{2}\leq H_a(x)\leq x+1,\quad 0\leq x<+\infty.
        \end{array}
    \end{equation}
    The derivatives of $G_a$ and $H_a$ satisfy
    \begin{equation}
        \label{equation_derivatives_G_a_H_a}
        \begin{array}{ll}
            \displaystyle  G_a'(x)=\frac{x+\kappa-\frac{1}{a^0p}}{\sqrt{x^2-1}}, & \qquad\displaystyle\partial_a G_a(x)=\partial_a\left(\kappa-\frac{1}{a^0p}\right)\arcosh(x),\vspace{5pt}\\
            \displaystyle  H_a'(x)=\frac{\sqrt{H_a(x)^2-1}}{H_a(x)+\kappa-\frac{1}{a^0p}}, & \displaystyle\qquad\partial_a H_a(x)=-\partial_a\left(\kappa-\frac{1}{a^0p}\right)\arcosh(H_a(x))H_a'(x).
            \end{array}
    \end{equation}
    Moreover, $H_a$ and $G_a$ satisfy the following asymptotic expansions
    \begin{equation}
    \label{equation_DL_G_H_infinity}
    \begin{array}{l}
        \displaystyle G_a(y)=y+\left(\kappa-\frac{1}{a^0p}\right)\log(2y) +O_{y\rightarrow +\infty}\left(\frac{1}{y}\right) ,\vspace{4pt}\\
        \displaystyle H_a(x)=x-\left(\kappa-\frac{1}{a^0p}\right)\log(2x)+\left(\kappa-\frac{1}{a^0p}\right)^2\frac{\log(2x)}{2x}+O_{x\rightarrow +\infty}\left(\frac{1}{x}\right),
        \end{array}
    \end{equation}
    \begin{equation*}
        G_a(1+h)=\sqrt{2}\left(1+\kappa-\frac{1}{a^0p}\right)\sqrt{h}+O_{h\rightarrow 0}\left(h^\frac{3}{2}\right),\qquad H_a(h)=1+\left(\frac{a^0\sqrt{m^2+\ell a^2}}{m^2+a^0\sqrt{m^2+\ell a^2}}\right)^2\frac{h^2}{2}+O_{h\rightarrow0}(h^4).
    \end{equation*}
    Here, the remainder terms are all uniform in $a,\ell$.
\end{proposition}
\begin{proof}
    Estimates~\eqref{equation_estimates_G_a_H_a} hold since $0\leq\kappa-\frac{1}{a^0p}=\frac{m^2}{a^0\sqrt{m^2+\ell a^2}}\leq 1$ and
    \begin{equation*}
         x-1\leq \sqrt{x^2-1}+\left(\kappa-\frac{1}{a^0p}\right)\arcosh(x)\leq x+\log(2x)\leq 2x.
    \end{equation*}
    The estimate of $H_a$ is obtained directly by considering $x=H_a(y)$. The expression for the derivatives follow from direct computations. For the asymptotic expansion, we first have 
    \begin{align*}
        G_a(y)&=y+\left(\kappa-\frac{1}{a^0p}\right)\log(2y)+\left(\sqrt{y^2-1}-y\right)+\left(\kappa-\frac{1}{a^0p}\right)(\arcosh(2y)-\log(2y))\\
        &= y+\left(\kappa-\frac{1}{a^0p}\right)\log(2y) +O_{y\rightarrow +\infty}\left(\frac{1}{y}\right). 
    \end{align*}
    Then, for $y=H_a(x)$, we derive
    \begin{align*}
        H_a(x)&=x-\left(\kappa-\frac{1}{a^0p}\right)\log(2H_a(x))+O_{x\rightarrow +\infty}\left(\frac{1}{H_a(x)}\right)\\
        &=x-\left(\kappa-\frac{1}{a^0p}\right)\log(2x)-\left(\kappa-\frac{1}{a^0p}\right)\log\left(1-\left(\kappa-\frac{1}{a^0p}\right)\frac{\log(2H_a(x))}{2x}+O_{x\rightarrow +\infty}\left(\frac{1}{x^2}\right)\right)+O_{x\rightarrow +\infty}\left(\frac{1}{x}\right)\\
        &=x-\left(\kappa-\frac{1}{a^0p}\right)\log(2x)+\left(\kappa-\frac{1}{a^0p}\right)^2\frac{\log(2x)}{2x}+O_{x\rightarrow +\infty}\left(\frac{1}{x}\right).
    \end{align*}
    For $h\rightarrow 0$, we compute
    \begin{align*}
        G_a(1+h)&=\sqrt{h^2+2h}+\left(\kappa-\frac{1}{a^0p}\right)\log\left(1+\sqrt{h^2+2h}+h\right)\\
        &=\sqrt{2h}+\left(\kappa-\frac{1}{a^0p}\right)\left(\sqrt{h^2+2h}+h-\frac{(\sqrt{h^2+2h}+h)^2}{2}+O_{h\rightarrow 0}(h^\frac{3}{2})\right)+O_{h\rightarrow 0}(h^\frac{3}{2})\\
        &=\sqrt{2}\left(1+\kappa-\frac{1}{a^0p}\right)\sqrt{h}+O_{h\rightarrow 0}(h^\frac{3}{2}).
    \end{align*}
    Finally, to obtain the asymptotic behavior of $H_a(h)$ as $h$ goes to $0$, we exploit the equation on $G_a(1+h)$ to derive 
    \begin{equation*}
        h=\sqrt{2}\left(1+\kappa-\frac{1}{a^0p}\right)\sqrt{H_a(h)-1}+O_{h\rightarrow 0}((H_a(h)-1)^\frac{3}{2}),
    \end{equation*}
    so that
    \begin{align*}
        H_a(h)&=1+\left(1+\kappa-\frac{1}{a^0p}\right)^{-2}\frac{h^2}{2}+O_{h\rightarrow0}\left((H_a(h)-1)^3\right)+hO_{h\rightarrow0}\left((H_a(h)-1)^{\frac{3}{2}}\right)\\
        &=1+\left(\frac{a^0\sqrt{m^2+\ell a^2}}{m^2+a^0\sqrt{m^2+\ell a^2}}\right)^2\frac{h^2}{2}+O_{h\rightarrow0}(h^4).
    \end{align*}
\end{proof}
\subsection{The action-angle variables}
\label{section_action_angle}
\subsubsection{Defining the action-angle}
We begin by defining the action-angle variables $(\theta,a)\in\R\times\R^*_+$. As explained at the beginning of the section, we can express them in terms of $(r,u)$ using the following functions
\begin{equation}
    \label{equation_definition_action_angle_A_Theta}
    \mathcal{A}(r,u):=\sqrt{\left(\sqrt{u^2+m^2+\frac{\ell}{r^2}}+\frac{1}{r}\right)^2-m^2},\qquad \Theta(r,u):=\frac{u}{|u|}p(\mathcal{A},\ell)G_\mathcal{A}\left(\frac{r}{p(\mathcal{A},\ell)}-\kappa(\mathcal{A},\ell)\right).
\end{equation}
We can then define the inverse functions by 
\begin{equation}
    \label{equation_definition_action_angle_R_U}
    R(\theta,a):=pH_a\left(\frac{|\theta|}{p}\right)+\kappa p,\qquad U(\theta,a):=\frac{\theta}{|\theta|}\sqrt{\left(a^0-\frac{1}{R(\theta,a)}\right)^2-m^2-\frac{\ell}{R(\theta,a)^2}}.
\end{equation}

\begin{remark}
    Here we defined our action-angle variables $(\mathcal{A},\Theta)$ for $c=1$. For $c\neq 1$, a similar formula can be obtained. Moreover, if we consider the non-relativistic limit $c\rightarrow +\infty$, we derive the following expressions
    \begin{equation}
        \begin{split}
            \mathcal{A}(r,u)&=\sqrt{\frac{2m}{r}+v^2},\\
            \Theta (r,u)&=\frac{u}{|u|}\frac{1}{\mathcal{A}^2}\sqrt{\left(r\mathcal{A}^2-m^2\right)^2-\sqrt{m^2+\ell \mathcal{A}^2}}+\frac{u}{|u|}\frac{m}{\mathcal{A}^2}\log\left(\sqrt{\frac{(r\mathcal{A}^2-m)^2}{m^2+\ell \mathcal{A}^2}-1}+\sqrt{\frac{r\mathcal{A}^2-m}{m^2+\ell \mathcal{A}^2}}\right).
        \end{split}
    \end{equation}
    These action-angle variables are consistent with the ones obtained the Vlasov-Poisson point charge system \eqref{equation_vlasov_poisson_point_charge}. Indeed, if we assume $m=1,\ell=0$, and $q=2$ we recover the same expressions as in \cite[eq. (2.4)]{Pausader_Widmayer_2021}
\end{remark}
\begin{proposition}
    The map $(\Theta,\mathcal{A}):\R^*_+\times\R\rightarrow\R\times\R^*_+$ defines a $C^1$ diffeomorphism with inverse $(R,U)$. Moreover, it preserves the volume form, i.e.
    \begin{equation*}
        \mathrm{d}\Theta\wedge\mathrm{d}\mathcal{A}=\mathrm{d}r\wedge\mathrm{d}u.
    \end{equation*}
\end{proposition}
\begin{proof}
    It follows from direct computations that $(R,U)$ is the inverse of $(\Theta,\mathcal{A})$. Similarly, by Proposition \ref{proposition_asymptotic_expansion_G_H}, we derive that the functions are $C^1$. Finally, by construction, they preserve the volume form.
\end{proof}

As expected, these variables allow us to simplify the linear problem \eqref{equation_RVPpc_linearized_2}.
\begin{proposition}
    Let $f$ be a solution to the linearized problem \eqref{equation_RVPpc_linearized_2} and 
    \begin{equation}
        \label{equation_def_solution_g_linear_aa}
        g(t,\theta,a,\ell):=f(t,R(\theta,a),U(\theta,a),\ell).
    \end{equation}
    Then $g$ is a solution to the free relativistic transport equation \eqref{equation_free_relat_transp}. Conversely, if we consider a solution $g$  to \eqref{equation_free_relat_transp} and 
    \begin{equation*}
        f(t,r,u,\ell):=g(t,\Theta(r,u),\mathcal{A}(r,u),\ell),
    \end{equation*}
    then $f$ is a solution to the linearized equation \eqref{equation_RVPpc_linearized_2}.
\end{proposition}

\subsubsection{Estimates on the radius}
Let us now give some estimates on the derivatives of the radius $R$.
\begin{proposition}
    \label{proposition_writing_derivatives_R}
    We have the following expressions for $R$ and its derivatives.
    \begin{align*}
        \partial_\theta R(\theta,a)=&~\sign(\theta)p\frac{\sqrt{\left(\frac{R}{p}-\kappa\right)^2-1}}{R-\frac{1}{a^0}}=\sign(\theta)\frac{a^0}{a}\frac{\sqrt{\left(a^0-\frac{1}{R}\right)^2-m^2-\frac{\ell}{R^2}}}{a^0-\frac{1}{R}},\\
        \partial_aR(\theta,a)=&~\partial_ap\left(\frac{R}{p}-\kappa-\frac{\theta}{p}\partial_\theta R\right)-p\partial_a\left(\kappa-\frac{1}{a^0p}\right)\arcosh\left(\frac{R}{p}-\kappa\right)\sign(\theta)\partial_\theta R+\partial_a(\kappa p),\\
        \partial^2_{\theta}R(\theta,a)=&~\frac{1}{a^2\left(R-\frac{1}{a^0}\right)^3}\left(\ell+m^2\frac{R}{a^0}\right),\\
        \partial_a\partial_\theta R(\theta,a)=&~\partial^2_\theta R\left(-\frac{\theta}{p}\partial_a p-p\partial_a\left(\kappa-\frac{1}{a^0p}\right)\arcosh\left(\frac{R}{p}-\kappa\right)\sign(\theta)\right)-\partial_a\left(\kappa-\frac{1}{a^0p}\right)\frac{p}{R-\frac{1}{a^0}}\sign(\theta)\partial_\theta R,\\
        \partial^2_aR(\theta,a)=&~\partial^2_ap\left(\frac{R}{p}-\kappa-\frac{\theta}{p}\partial_\theta R\right)+\theta(\partial_ap)\partial_a \left(\kappa-\frac{1}{a^0p}\right)\frac{\partial_\theta R}{R-\frac{1}{a^0}}+\partial^2_a(\kappa p)\\
        &-\left(2(\partial_ap)\partial_a\left(\kappa-\frac{1}{a^0p}\right)+p\partial_a^2\left(\kappa-\frac{1}{a^0p}\right)-\frac{p^2}{R-\frac{1}{a^0}}\left(\partial_a\left(\kappa-\frac{1}{a^0p}\right)\right)^2\right)\arcosh\left(\frac{R}{p}-\kappa\right)\frac{\theta}{|\theta|}\partial_\theta R\\
        &-\left(\partial_a p\frac{\theta}{p}+p\partial_a\left(\kappa-\frac{1}{a^0p}\right)\arcosh\left(\frac{R}{p}-\kappa\right)\frac{\theta}{|\theta|}\right)\partial_a\partial_\theta R.
\end{align*}
    \end{proposition}
    \begin{proof}
        First of all, notice that, by definition of $R$, we have
        \begin{equation}
            \label{equation_proof_equality_R_theta_H_a_prime}
            \partial_\theta R(\theta,a)=\sign(\theta)H_a'\left(\frac{|\theta|}{p}\right).
        \end{equation}
        Moreover, by \eqref{equation_derivatives_G_a_H_a}, we derive
        \begin{equation*}
            H_a'\left(\frac{|\theta|}{p}\right)=p\frac{\sqrt{\left(\frac{R}{p}-\kappa\right)^2-1}}{R-\frac{1}{a^0}}=\frac{a^0}{a^0-\frac{1}{R}}\sqrt{\left(1-\frac{\kappa p}{R}\right)^2-\frac{p^2}{R^2}}.
        \end{equation*}
        Since $\kappa p=\frac{a^0}{a^2}$ and $p=\frac{\sqrt{m^2+\ell a^2}}{a^2}$, it follows that
        \begin{equation*}
            H_a'\left(\frac{|\theta|}{p}\right)= \frac{a^0}{a^0-\frac{1}{R}}\sqrt{1-\frac{2a^0}{Ra^2}+\frac{(a^0)^2}{R^2 a^4}-\frac{m^2+\ell a^2}{R^2 a^4}}=\frac{a^0}{a}\frac{\sqrt{\left(a^0-\frac{1}{R}\right)^2-m^2-\frac{\ell}{R^2}}}{a^0-\frac{1}{R}}.
        \end{equation*}
        This grants us the expression of $\partial_\theta R$. Then, using \eqref{equation_derivatives_G_a_H_a} and \eqref{equation_proof_equality_R_theta_H_a_prime}, we derive
        \begin{align*}
    \partial_a R(\theta,a)&=\partial_apH_a\left(\frac{|\theta|}{p}\right)-\frac{\partial_a p}{p}|\theta|H_a'\left(\frac{|\theta|}{p}\right)+p\partial_a H_a\left(\frac{|\theta|}{p}\right)+\partial_a(\kappa p)\\
            &=\partial_ap\left(\frac{R}{p}-\kappa-\frac{\theta}{p}\partial_\theta R\right)-p\partial_a\left(\kappa-\frac{1}{a^0p}\right)\arcosh\left(\frac{R}{p}-\kappa\right)\sign(\theta)\partial_\theta R+\partial_a(\kappa p).
        \end{align*}
        Next, we recall the first expression of $\partial_\theta R$ from Proposition \ref{proposition_writing_derivatives_R} and differentiate it with respect to $\theta$. We obtain
        \begin{align*}
            \partial_\theta^2 R(\theta,a)&=\sign(\theta)\frac{p}{R-\frac{1}{a^0}}\frac{\partial_\theta R}{p}\frac{\frac{R}{p}-\kappa}{\sqrt{\left(\frac{R}{p}-\kappa\right)^2-1}}-\sign(\theta)p\partial_\theta R\frac{\sqrt{\left(\frac{R}{p}-\kappa\right)^2-1}}{\left(R-\frac{1}{a^0}\right)^2}\\
            &=\frac{p|\partial_\theta R| }{\left(R-\frac{1}{a^0}\right)^2\sqrt{\left(\frac{R}{p}-\kappa\right)^2-1}}\left(\left(\frac{R}{p}-\kappa\right)\left(\frac{R}{p}-\frac{1}{a^0 p}\right)-\left(\frac{R}{p}-\kappa\right)^2+1\right)
        \end{align*}
        Recalling the expression of $\partial_\theta R$, we derive
        \begin{align*}
            \partial_\theta^2 R(\theta,a)&=\frac{1}{\left(R-\frac{1}{a^0}\right)^3}\left(\left(\kappa p-\frac{1}{a^0}\right)R+ p^2+(\kappa p)^2+\frac{\kappa p}{a^0}\right).
        \end{align*}
        The result then follows from the expressions of $\kappa$ and $p$ in \eqref{equation_def_p_kappa}. Then, by definition of $R$ and using \eqref{equation_derivatives_G_a_H_a}, we have
        \begin{equation*}
            \partial_\theta\left(\frac{R}{p}-\kappa\right)=\frac{\partial_\theta R}{p},\qquad \partial_a\left(\frac{R}{p}-\kappa\right)=-\frac{\partial_a p}{p^2}\theta\partial_\theta R -\partial_a\left(\kappa-\frac{1}{a^0p}\right)\arcosh\left(\frac{R}{p}-\kappa\right)\sign(\theta)\partial_\theta R.
        \end{equation*}
        The expression of $\partial_a\partial_\theta R$ then follows from differentiating the expression of $\partial_a R$ with respect to $\theta$. Finally, the expression of $\partial_a^2R$ is obtained by differentiating the expression of $\partial_aR$ with respect to $a$, and using the previous equation.
    \end{proof}
    Before proving the estimates on $R$ and its derivatives, let us state a lemma regarding the derivatives of several quantities depending on $\kappa$ and $p$ defined in \eqref{equation_def_p_kappa}.
\begin{lemma}
    \label{lemma_derivatives_of_p_kappa}
    We have 
    \begin{equation}
        \label{equation_derivatives_kappa_p_etc}
        \begin{array}{c}
            \displaystyle \partial_a p =\frac{-\ell a^2-2m^2}{a^3\sqrt{m^2+\ell a^2}}=\frac{\ell}{pa^3}-\frac{2p}{a}, \qquad \partial_a(\kappa p)=-\frac{a^2+2m^2}{a^0a^3}  \vspace{5pt}\\
             \displaystyle \partial_a\left(\kappa p-\frac{1}{a^0}\right)=-\frac{m^2}{(a^0)^3a}-\frac{2m^2}{a^0a^3} ,\qquad \partial_a\left(\kappa -\frac{1}{a^0p}\right)=-\frac{\partial_a p}{p^2}\frac{m^2}{a^0a^2}-\frac{m^2}{p}\left(\frac{1}{a(a^0)^3}+\frac{2}{a^0a^3}\right).
        \end{array}
    \end{equation}
    For the second derivatives, we obtain 
    \begin{equation*}
        \begin{array}{c}
            \displaystyle \partial_a^2p=-\frac{\ell^2}{p^3a^6}-\frac{3\ell}{pa^4}+\frac{6p}{a^2},\quad  \partial_a^2(\kappa p)=\frac{2a^4+9m^2a^2+6m^4}{(a^0)^3a^4},\quad \partial_a^2\left(\kappa p-\frac{1}{a^0}\right)=\frac{12m^2a^4+15m^4a^2+6m^6}{(a^0)^5a^4},\vspace{5pt}\\
            \displaystyle \partial_a^2\left(\kappa -\frac{1}{a^0p}\right)=-\frac{\partial_a^2p}{p^2}\frac{m^2}{a^0a^2}+2\frac{(\partial_a p)^2}{p^3}\frac{m^2}{a^0a^2}+\frac{2m^2\partial_a p}{aa^0p^2}\left(\frac{2}{a^2}+\frac{1}{(a^0)^2}\right)+\frac{m^2}{a^0p}\left(\frac{1}{a^2(a^0)^2}+\frac{3}{(a^0)^4}+\frac{2}{(a^0)^2a^2}+\frac{6}{a^4}\right)
        \end{array}
    \end{equation*}
    Finally, the following estimates hold
    \begin{equation}
        \label{equation_estimates_derivatives_kappa_p}
        \begin{array}{c}
            \displaystyle |\partial_ap|\lesssim \frac{p}{a},\qquad |\partial_a(\kappa p)|\lesssim \frac{a^0}{a^3},\qquad \left|\partial_a\left(\kappa-\frac{1}{a^0p}\right)\right|\lesssim \frac{1}{p}\frac{1}{a^0a^3},\qquad \left|\partial_a\left(\kappa p-\frac{1}{a^0}\right)\right|\lesssim \frac{1}{a^0a^3},\vspace{5pt}\\
            \displaystyle|\partial_a^2p|\lesssim\frac{p}{a^2},\qquad |\partial_a^2(\kappa p)|\lesssim \frac{a^0}{a^4},\qquad\left|\partial_a^2\left(\kappa-\frac{1}{a^0p}\right)\right|\lesssim \frac{1}{p}\frac{1}{a^0a^4},\qquad \left|\partial_a^2\left(\kappa p-\frac{1}{a^0}\right)\right|\lesssim \frac{1}{a^0a^4}.
        \end{array}
    \end{equation}
\end{lemma}
\begin{proof}
    We first recall that
    \begin{equation}
        \label{equation_kappa_times_p}
        \kappa p=\frac{a^0}{a^2},\qquad \kappa -\frac{1}{a^0 p}= \frac{1}{p}\frac{m^2}{a^0 a^2}.
    \end{equation}
    The expressions of the derivatives are then obtained by direct computations. For the estimates, we have
    \begin{equation*}
        |\partial_a p|\lesssim p \frac{\ell a}{(m^2+\ell a^2)}+\frac{p}{a}\lesssim\frac{p}{a}.
    \end{equation*}
    This implies
    \begin{equation*}
        \left|\partial_a\left(\kappa -\frac{1}{a^0p}\right)\right|\lesssim \frac{|\partial_a p|}{p^2}\frac{1}{a^0a^2}+\frac{1}{p}\frac{1}{a^0 a^3}\lesssim \frac{1}{p}\frac{1}{a^0 a^3}.
    \end{equation*}
    Similarly,
    \begin{equation*}
        |\partial_a^2p|\lesssim p\frac{\ell^2a^2}{(m^2+\ell a^2)^2}+p\frac{\ell}{m^2+\ell a^2}+\frac{p}{a^2}\lesssim \frac{p}{a^2}.
    \end{equation*}
    
    The other estimates follow from the estimates of $\partial_a p,\partial_a^2 p$ and similar computations.
\end{proof}
    \begin{proposition}
    \label{proposition_estimates_R}
    $R(\theta,a)$ has a lower bound given by 
    \begin{equation}
        \label{equation_lower_bound_R}
        R(\theta,a)a^2\geq a^0+\sqrt{m^2+\ell a^2},\qquad \frac{R(\theta,a)}{p}-\kappa\geq 1.
    \end{equation}
    Moreover, the following estimates hold for its derivatives. 
    \begin{equation}
        \label{equation_estimate_first_derivatives_R}
        0\leq\sign(\theta)\partial_\theta R(\theta,a)\leq 1,\qquad |\partial_a R(\theta,a)|\lesssim \left(\frac{1}{a}+\frac{\sqrt{\ell}}{a^2}+\frac{1}{a^3}\right)\log \left\langle \frac{R}{p}\right\rangle\lesssim \left(\frac{1}{a}+\frac{p}{a}\right)\log \left\langle \frac{R}{p}\right\rangle,
    \end{equation}
    \begin{equation}
        \label{equation_estimate_second_derivatives_theta_R}
        |\partial_\theta^2R|\lesssim \frac{1}{R-\frac{1}{a^0}}\lesssim \frac{a^0}{R},\qquad |\partial_a\partial_\theta R|\lesssim \frac{a^0}{Ra^3}+\frac{\sqrt{\ell}}{Ra^2}\lesssim \frac{1}{a},\qquad |\partial^2_aR|\lesssim \left(\frac{1}{a^2}+\frac{\sqrt{\ell}}{a^3}+\frac{1}{a^4}\right)\log\left\langle \frac{R}{p}\right\rangle
    \end{equation}

\end{proposition}
\begin{proof}
    By definition of $R$ in \eqref{equation_definition_action_angle_R_U}, since $H_a\geq 1$, we find
    \begin{equation*}
        R(\theta,a)\geq p+\kappa p=\frac{a^0+\sqrt{m^2+\ell a^2}}{a^2}.
    \end{equation*}
    This directly implies \eqref{equation_lower_bound_R}. We now turn to the estimates for the derivatives. For $\partial_\theta R$, by Proposition \ref{proposition_writing_derivatives_R}, we have 
    \begin{align}
        \label{equation_derivative_theta_R_lower_1_proof}
        \sign(\theta)\partial_\theta R=\sqrt{\frac{(a^0)^2(a^0-\frac{1}{R})^2-(a^0)^2\frac{\ell}{R^2}-(a^0)^2m^2}{a^2(a^0-\frac{1}{R})^2}}
        =\sqrt{1-\frac{\frac{\ell (a^0)^2}{R}+m^2\left(2a^0-\frac{1}{R}\right)}{Ra^2\left(a^0-\frac{1}{R}\right)^2}}.
    \end{align}
    Since $a^0-\frac{1}{R}>0$, we obtain the result. Recalling the expression for $\partial_a R$ in Proposition \ref{proposition_writing_derivatives_R} and $|\theta|=pG_a\left(\frac{R}{p}-\kappa\right)$, we obtain
    \begin{equation}
        \label{equation_derivative_a_R_other_expression_proof}
        \begin{array}{r@{}l}
            \displaystyle\partial_aR=&\displaystyle~\partial_ap\left(\frac{R}{p}-\kappa-\left(\sqrt{\left(\frac{R}{p}-\kappa\right)^2-1}+\left(\kappa-\frac{1}{a^0 p}\right)\arcosh\left(\frac{R}{p}-\kappa\right)\right)\sign(\theta)\partial_\theta R\right) \vspace{4pt}\\
        \displaystyle&\displaystyle-p\partial_a\left(\kappa-\frac{1}{a^0p}\right)\arcosh\left(\frac{R}{p}-\kappa\right)\sign(\theta)\partial_\theta R+\partial_a(\kappa p)\vspace{4pt}\\
        \displaystyle=&\displaystyle~\partial_ap\left(\frac{R}{p}-\kappa-\sqrt{\left(\frac{R}{p}-\kappa\right)^2-1}\frac{\theta}{|\theta|}\partial_\theta R\right)-\partial_a\left(\kappa p-\frac{1}{a^0}\right)\arcosh\left(\frac{R}{p}-\kappa\right)\frac{\theta}{|\theta|}\partial_\theta R+\partial_a(\kappa p)
\end{array}
    \end{equation}
    Then, using the expression for $\partial_\theta R$ from Proposition \ref{proposition_writing_derivatives_R}, note that
    \begin{equation}
        \label{equation_R_minus_theta_proof}
        \frac{R}{p}-\sqrt{\left(\frac{R}{p}-\kappa\right)^2-1}\sign(\theta)\partial_\theta R
        =\frac{1}{R-\frac{1}{a^0}}\left(\left(2\kappa-\frac{1}{a^0p}\right)R+p(1-\kappa^2)\right).
    \end{equation}
    Moreover, using \eqref{equation_lower_bound_R}, we obtain
    \begin{equation}
        \label{equation_lower_bound_R_a0_proof}
        R-\frac{1}{a^0}\geq \frac{m^2+a^0\sqrt{m^2+\ell a^2}}{a^0 a^2},\qquad a^0-\frac{1}{R}\geq \frac{m^2+a^0\sqrt{m^2+\ell a^2}}{a^0+\sqrt{m^2+\ell a^2}}\geq \frac{m}{2}.
    \end{equation}
    Hence, by \eqref{equation_estimates_derivatives_kappa_p} and the relation $\kappa p =\frac{a^0}{a^2}$, we derive
    \begin{equation*}
        |\partial_ap|\frac{a^0}{a^0-\frac{1}{R}}\left(2\kappa -\frac{1}{a^0p}\right)=\frac{|\partial_a p|}{p}\frac{(a^2+2m^2)}{a^2\left(a^0-\frac{1}{R}\right)}\lesssim \frac{1}{a}+\frac{1}{a^3}.
    \end{equation*}
    Similarly,
    \begin{equation*}
        |\partial_ap|\frac{p(1-\kappa^2)}{R-\frac{1}{a^0}}=\frac{|\partial_a p|}{p}\frac{|\ell-1|}{a^2}\frac{1}{R-\frac{1}{a^0}}\lesssim \frac{\sqrt{\ell}}{a^2}+\frac{1}{a}.
    \end{equation*}

    Finally, by \eqref{equation_derivatives_kappa_p_etc}, since $\frac{R}{p}-\kappa\geq 1$, we find
    \begin{equation*}
        \left|\partial_a\left(\kappa p-\frac{1}{a^0}\right)\arcosh\left(\frac{R}{p}-\kappa\right)\sign(\theta)\partial_\theta R\right|\lesssim \frac{1}{a^3}\log\left\langle\frac{R}{p}\right\rangle,
    \end{equation*}
    as well as
    \begin{equation*}
        \partial_a(\kappa p)\lesssim \frac{1}{a^2}+\frac{1}{a^3},
    \end{equation*}
    which concludes the estimatation of $\partial_aR$. 

    \textbf{\underline{Estimates on the second order derivatives}}\\
    We begin with $\partial^2_\theta R$, one has 
    \begin{equation*}
        \partial^2_{\theta}R(\theta,a)=\frac{1}{a^2\left(R-\frac{1}{a^0}\right)^3}\left(\ell+m^2\frac{R}{a^0}\right)=\frac{1}{R-\frac{1}{a^0}}\Bigg(\frac{\ell}{a^2\big(R-\frac{1}{a^0}\big)^2}+\frac{m^2}{a^2}\frac{1}{\left(a^0-\frac{1}{R}\right)}\frac{1}{R-\frac{1}{a^0}}\Bigg).
    \end{equation*}
    Thus, using \eqref{equation_lower_bound_R_a0_proof}, we obtain
    \begin{equation*}
        \frac{\ell}{a^2\big(R-\frac{1}{a^0}\big)^2}\lesssim \frac{\ell a^2 (a^0)^2}{\big(m^2+a^0\sqrt{m^2+\ell a^2}\big)^2}\lesssim 1,
    \end{equation*}
    as well as
    \begin{equation*}
        \frac{m^2}{a^2}\frac{1}{\left(a^0-\frac{1}{R}\right)}\frac{1}{R-\frac{1}{a^0}}\lesssim \frac{a^0}{m^2+a^0\sqrt{m^2+\ell a^2}}\lesssim 1,
    \end{equation*}
    which gives the estimate by applying \eqref{equation_lower_bound_R_a0_proof}. Then, for $\partial_a\partial_\theta R$, we have
    \begin{align*}
        \partial_a\partial_\theta R=&-\partial_a\left(\kappa-\frac{1}{a^0p}\right)\frac{p}{R-\frac{1}{a^0}}\sign(\theta)\partial_\theta R\\
        &+\frac{1}{a^2\left(R-\frac{1}{a^0}\right)^3}\left(\ell+m^2\frac{R}{a^0}\right)\left(-\frac{\theta}{p}\partial_a p-p\partial_a\left(\kappa-\frac{1}{a^0p}\right)\arcosh\left(\frac{R}{p}-\kappa\right)\sign(\theta)\right).
    \end{align*}
    By \eqref{equation_estimates_derivatives_kappa_p} and \eqref{equation_lower_bound_R_a0_proof}, since $|\partial_\theta R|\leq 1$, the first term satisfies
    \begin{equation*}
        \left|\partial_a\left(\kappa-\frac{1}{a^0p}\right)\frac{p}{R-\frac{1}{a^0}}\sign(\theta)\partial_\theta R\right|\lesssim \frac{1}{a^0 a^3}\frac{a^0}{R}\frac{1}{a^0-\frac{1}{R}}\lesssim  \frac{1}{Ra^3}.
    \end{equation*}
    For the second term, we study each part separately. First, since $|\theta|\lesssim p\sqrt{\left(\frac{R}{p}-\kappa\right)^2-1}$, we can use the expression for $\partial_\theta R$ from Proposition \ref{proposition_writing_derivatives_R}, \eqref{equation_estimates_derivatives_kappa_p}, and \eqref{equation_lower_bound_R_a0_proof} to find that
    \begin{equation*}
        |\partial_ap|\frac{|\theta|}{p}\frac{\ell}{a^2\left(R-\frac{1}{a^0}\right)^3}\lesssim \frac{1}{Ra^2} \frac{|\partial_a p|}{p}\frac{\ell a^0}{\left(a^0-\frac{1}{R}\right)\left(R-\frac{1}{a^0}\right)}|\partial_\theta R|\lesssim \frac{1}{R a}\frac{\ell (a^0)^2(a^0+\sqrt{m^2+\ell a^2})}{(m^2+a^0\sqrt{m^2+\ell a^2})^2}.
    \end{equation*}
    Hence,
    \begin{equation*}
        |\partial_ap|\frac{|\theta|}{p}\frac{\ell}{a^2\left(R-\frac{1}{a^0}\right)^3}\lesssim \frac{1}{Ra}\frac{\ell a^0}{m^2+\ell a^2}+\frac{1}{Ra}\frac{\ell}{\sqrt{m^2+\ell a^2}}\lesssim\frac{\sqrt{\ell}}{Ra^2}+\frac{a^0}{Ra^3}.
    \end{equation*}
    Similarly, since $a^0-\frac{1}{R}\gtrsim 1$,
    \begin{equation*}
        |\partial_ap|\frac{|\theta|}{p}\frac{R}{a^0}\frac{1}{a^2\left(R-\frac{1}{a^0}\right)^3}\lesssim \frac{1}{Ra^2} \frac{|\partial_a p|}{p}\frac{ a^0}{\left(a^0-\frac{1}{R}\right)^2}|\partial_\theta R|\lesssim \frac{a^0}{Ra^3}.
    \end{equation*}

    Moreover, since
    \begin{equation*}
        p\frac{1}{R-\frac{1}{a^0}}\arcosh\left(\frac{R}{p}-\kappa\right)\leq p\frac{\sqrt{\left(\frac{R}{p}-\kappa\right)^2-1}}{R-\frac{1}{a^0}}=|\partial_\theta R|\leq 1,
    \end{equation*}
    we obtain
    \begin{equation*}
        \frac{\ell+m^2\frac{R}{a^0}}{a^2\left(R-\frac{1}{a^0}\right)^3}p\left|\partial_a\left(\kappa-\frac{1}{a^0p}\right)\right|\arcosh\left(\frac{R}{p}-\kappa\right)\lesssim \frac{\ell+\frac{R}{a^0}}{a^2\left(R-\frac{1}{a^0}\right)^2}\left|\partial_a\left(\kappa-\frac{1}{a^0p}\right)\right|.
    \end{equation*}
    Then, by \eqref{equation_estimates_derivatives_kappa_p} and \eqref{equation_lower_bound_R_a0_proof}, since $pa^2=\sqrt{m^2+\ell a^2}$, we find
    \begin{equation*}
        \frac{\ell}{a^2\left(R-\frac{1}{a^0}\right)^2}\left|\partial_a\left(\kappa-\frac{1}{a^0p}\right)\right|\lesssim \frac{1}{R}\frac{\ell}{pa^5\left(R-\frac{1}{a^0}\right)\left(a^0-\frac{1}{R}\right)}\lesssim \frac{1}{Ra}\frac{\ell a^0}{\sqrt{m^2+\ell a^2}(m^2+a^0\sqrt{m^2+\ell a^2})}\lesssim \frac{1}{R a^3},
    \end{equation*}
    and similarly, we have
    \begin{equation*}
        \frac{R}{a^0a^2\left(R-\frac{1}{a^0}\right)^2}\left|\partial_a\left(\kappa-\frac{1}{a^0p}\right)\right|\lesssim \frac{1}{R}\frac{1}{pa^5\left(a^0-\frac{1}{R}\right)^2}\lesssim \frac{1}{R a^3}.
    \end{equation*}
    Hence,
    \begin{equation*}
        \frac{\ell+m^2\frac{R}{a^0}}{a^2\left(R-\frac{1}{a^0}\right)^3}p\left|\partial_a\left(\kappa-\frac{1}{a^0p}\right)\right|\arcosh\left(\frac{R}{p}-\kappa\right)\lesssim \frac{1}{R a^3}.
    \end{equation*}
    This completes the estimate for $\partial_a\partial_\theta R$. We finally prove the estimate for $\partial_a^2R$. First, let us recall the expression for $\partial_a^2R$ given in Proposition \ref{proposition_writing_derivatives_R} by 
    \begin{align*}
        \partial^2_aR(\theta,a)=\partial^2_a(\kappa p)+I_1+I_2+I_3+I_4,
    \end{align*}
    where 
    \begin{align*}
        I_1&:=\partial^2_ap\left(\frac{R}{p}-\kappa-\frac{\theta}{p}\partial_\theta R\right)=\partial_a^2p\left(\frac{R}{p}-\kappa-\sqrt{\left(\frac{R}{p}-\kappa\right)^2-1}\frac{\theta}{|\theta|}\partial_\theta R-\left(\kappa-\frac{1}{a^0p}\right)\arcosh\left(\frac{R}{p}-\kappa\right)\frac{\theta}{|\theta|}\partial_\theta R\right),\\
        I_2&:= \theta(\partial_ap)\partial_a \left(\kappa-\frac{1}{a^0p}\right)\frac{\partial_\theta R}{R-\frac{1}{a^0}},\\
        I_3&:=-\left(2(\partial_ap)\partial_a\left(\kappa-\frac{1}{a^0p}\right)+p\partial_a^2\left(\kappa-\frac{1}{a^0p}\right)-\frac{p^2}{R-\frac{1}{a^0}}\left(\partial_a\left(\kappa-\frac{1}{a^0p}\right)\right)^2\right)\arcosh\left(\frac{R}{p}-\kappa\right)\frac{\theta}{|\theta|}\partial_\theta R,\\
        I_4&:=-\left(\partial_a p\frac{\theta}{p}+p\partial_a\left(\kappa-\frac{1}{a^0p}\right)\arcosh\left(\frac{R}{p}-\kappa\right)\frac{\theta}{|\theta|}\right)\partial_a\partial_\theta R.
    \end{align*}
    We begin by giving an upper bound for $I_1$. In view of \eqref{equation_R_minus_theta_proof}, we are lead to estimate
    \begin{equation*}
        |\partial_a^2p|\frac{a^0}{a^0-\frac{1}{R}}\left(2\kappa-\frac{1}{a^0p}\right)=\frac{|\partial_a^2 p|}{p}\frac{(a^2+2m^2)}{a^2\left(a^0-\frac{1}{R}\right)}\lesssim \frac{1}{a^2}+\frac{1}{a^4},
    \end{equation*}
    and, by \eqref{equation_estimates_derivatives_kappa_p} and \eqref{equation_lower_bound_R_a0_proof},
    \begin{equation*}
        |\partial_a^2p|\frac{p(1-\kappa^2)}{R-\frac{1}{a^0}}=\frac{|\partial_a^2 p|}{p}\frac{|\ell-1|}{a^2}\frac{1}{R-\frac{1}{a^0}}\lesssim \frac{1+\ell}{a^2}\frac{a^0}{m^2+a^0\sqrt{m^2+\ell a^2}}\lesssim \frac{\sqrt{\ell}}{a^3}+\frac{1}{a^2}.
    \end{equation*}
    Moreover, by \eqref{equation_kappa_times_p} and \eqref{equation_estimates_derivatives_kappa_p}, since $\frac{R}{p}-\kappa\geq 1$, we obtain
    \begin{equation*}
        |\partial_a^2p|\left(\kappa-\frac{1}{a^0 p}\right)\arcosh\left(\frac{R}{p}\right)|\partial_\theta R|\lesssim \frac{1}{a^2}\left(\kappa p-\frac{1}{a^0}\right) \log\left\langle \frac{R}{p}\right\rangle\lesssim \frac{1}{a^4}\log\left\langle \frac{R}{p}\right\rangle.
    \end{equation*}
    Hence, by combining $|\partial_a^2p| \kappa \lesssim \frac{\kappa p}{a^2}\lesssim \frac{1}{a^3}+\frac{1}{a^4}$ with the previous estimates, we obtain 
    \begin{equation*}
         |I_1|\lesssim\left(\frac{1}{a^2}+\frac{\sqrt{\ell}}{a^3}+\frac{1}{a^4}\right)\log\left\langle \frac{R}{p}\right\rangle.
    \end{equation*}
    Then, since $|\theta|=pG_a\left(\frac{R}{p}-\kappa\right)\leq 2p\sqrt{\left(\frac{R}{p}-\kappa\right)^2-1}$, we find
    \begin{equation*}
        |I_2|\leq 2\left|\partial_a\left(\kappa-\frac{1}{a^0p}\right)\right||\partial_a p||\partial_\theta R|^2 \lesssim \frac{|\partial_a p|}{p}\frac{1}{a^3}\lesssim \frac{1}{a^4}.
    \end{equation*}
    For $I_4$, we recall the upper bounds for $\partial_a\partial_\theta R$ given in \eqref{equation_estimate_second_derivatives_theta_R}. Moreover, we have
    \begin{equation*}
        \frac{|\partial_ap|}{p}|\theta|\lesssim \frac{R}{a},
    \end{equation*}
    as well as
    \begin{equation*}
        p\left|\partial_a\left(\kappa-\frac{1}{a^0p}\right)\right|\arcosh\left(\frac{R}{p}-\kappa\right)\lesssim \frac{1}{a^3}\log\left\langle\frac{R}{p}\right\rangle.
    \end{equation*}
    This implies
    \begin{equation*}
        |I_4|\lesssim \frac{a^0}{a^4}+\frac{\sqrt{\ell}}{a^3}+\frac{1}{a^4}\log\left\langle\frac{R}{p}\right\rangle\lesssim \left(\frac{1}{a^2}+\frac{\sqrt{\ell}}{a^3}+\frac{1}{a^4}\right)\log\left\langle \frac{R}{p}\right\rangle.
    \end{equation*}
    Finally, we study $I_3$. Using \eqref{equation_estimates_derivatives_kappa_p} and \eqref{equation_lower_bound_R_a0_proof}, we find
    \begin{align*}
        |\partial_a p|\left|\partial_a\left(\kappa-\frac{1}{a^0p}\right)\right|&\lesssim \frac{1}{a^4},\\
        p\left|\partial_a^2\left(\kappa-\frac{1}{a^0p}\right)\right|&\lesssim \frac{1}{a^4},\\
         \frac{p^2}{R-\frac{1}{a^0}}\left|\partial_a\left(\kappa-\frac{1}{a^0p}\right)\right|^2&\lesssim \frac{1}{a^4},
    \end{align*}
    which grants us, as $|\partial_\theta R|\leq 1$,
    \begin{equation*}
        |I_3|\lesssim \frac{1}{a^4}|\partial_\theta R|\arcosh\left(\frac{R}{p}-\kappa\right)\lesssim \frac{1}{a^4}\log\left\langle\frac{R}{p}\right\rangle.
    \end{equation*}
    This proves the estimate for $\partial_a^2R$.
\end{proof}

\subsection{Composing with the linear flow}
\label{section_composing_linear_flow}

Recall that $f$ solves the linear problem \eqref{equation_RVPpc_linearized_2}. Since the action-angle variables have already been introduced, equation \eqref{equation_RVPpc_linearized_2} can be rewritten as the free relativistic transport equation \eqref{equation_free_relat_transp}. Consequently, we can compose $g$, defined in \eqref{equation_def_solution_g_linear_aa}, with the linear flow $\theta + t\widehat{a}$ and define
\begin{equation}
    \label{equation_def_gamma}
    \gamma(t,\theta,a,\ell):=f(t,R(\theta+t\widehat{a},a), U(\theta+t\widehat{a},a),\ell).
\end{equation}
By doing so, we directly obtain $\partial_t\gamma=0$. In the following, we will write
\begin{equation*}
    \widetilde{R}(\theta,a):=R(\theta+t\widehat{a},a).
\end{equation*}
Here, we omit the dependence on $t$ in our notation. Using these new functions, we can write, for instance,  
\begin{equation}
    M(t,r)=4\pi^2\int_0^r\int_0^\infty\int_{\R} f(t,s,u,\ell)\mathrm{d}u\mathrm{d}\ell\mathrm{d}s=4\pi^2\iiint \1_{\{\widetilde{R}(\theta,a)\leq r\}} \gamma(t,\theta,a,\ell)\mathrm{d}\theta\mathrm{d}a\mathrm{d}\ell,
\end{equation}
\begin{equation}
    \label{equation_definition_rho}
    \rho(t,r):=\partial_r M(t,r)=4\pi^2\iiint \delta\left(\widetilde{R}(\theta,a)-r\right) \gamma(t,\theta,a,\ell)\mathrm{d}\theta\mathrm{d}a\mathrm{d}\ell.
\end{equation}
Similarly, we know that the electric field $E(t,r)=\frac{M(t,r)}{r^2}$ derives from the potential $\Psi$. It rewrites
\begin{equation}
    \Psi(t,r)=-4\pi^2\iiint\frac{f(t,s,u,\ell)}{\max(r,s)}\mathrm{d}u\mathrm{d}\ell\mathrm{d}s=-4\pi^2\iiint \frac{1}{\max(r,\widetilde{R}(\theta,a))}\gamma(t,\theta,a,\ell)\mathrm{d\theta}\mathrm{d}a\mathrm{d}\ell.
\end{equation}
Finally, we also define 
\begin{equation}
    \widetilde{\Psi}(t,\theta,a):=\Psi(t,\widetilde{R}(\theta,a)).
\end{equation}
Consequently, we may focus entirely on the study of $\gamma$, and establish upper bounds for $M,\rho$ and the derivatives of $\widetilde{\Psi}$ in terms of $\gamma$.\\
In the following, since we focus on proving the global existence and large time behavior of solutions, we will consider the case of large times $t\geq T_0$. Note that $T_0$ will only depend on $m$ (see the proof of \eqref{equation_lower_bound_R_tilde_bulk} on this matter).
\subsubsection{The bulk}

In order to prove the estimates in the following sections, we introduce a subset $\mathcal{B}_t$, called the \textbf{bulk}. It is defined as
\begin{equation}
    \label{equation_def_bulk}
    \mathcal{B}_t:=\left\{(\theta,a,\ell)\in\R\times\R^*_+\times\R_+\,|\, \widehat{a}\geq t^{-\frac{1}{4}},\quad a\leq t^\frac{1}{3},\quad |\theta|\leq \frac{t\widehat{a}}{2},\quad \ell \leq t^\frac{1}{2}\right\},
\end{equation}
and represented in Figure \ref{figure_bulk}.

\begin{figure}[!h]
    \centering
    \centering
    \begin{tikzpicture}[thick,scale=1.5, every node/.style={scale=1}]

    
            \draw ({1/sqrt(sqrt(20)-1)},-0.2) node[left] {$\widehat{a}=t^{-\frac{1}{4}}$};
            \draw ({20^(1/3)},-0.2) node[right] {$a=t^\frac{1}{3}$};

            \draw plot[domain=(-2/sqrt(sqrt(20)-1))/sqrt(1+1/(sqrt(20)-1)):2/sqrt(sqrt(20)-1))/sqrt(1+1/(sqrt(20)-1)), very thick] ({1/sqrt(sqrt(20)-1)},\x);

            \draw plot[domain=(-2*20^(1/3)/sqrt(1+20^(2/3)):2*20^(1/3)/sqrt(1+20^(2/3)), very thick] ({20^(1/3)},\x);

            \fill [pattern=vertical lines, pattern color=gray]
            plot[domain=1/sqrt(sqrt(20)-1):20^(1/3)] (\x,{2*\x/(sqrt(1+\x^2)})--plot [domain=20^(1/3):1/sqrt(sqrt(20)-1)] (\x,{-2*\x/(sqrt(1+\x^2)});
            \draw (1.7,0.8) node [thick, scale =1.5] {$\mathcal{B}_t$};

            \draw[domain=1/sqrt(sqrt(20)-1):20^(1/3), very thick]   plot (\x,{-2*\x/(sqrt(1+\x^2)}) ;
            \draw ({20^(1/3)},-2) node[right] {$\theta=-t\widehat a$};

            \draw[domain=1/sqrt(sqrt(20)-1):20^(1/3), very thick]   plot (\x,{2*\x/(sqrt(1+\x^2)}) ;
            \draw ({20^(1/3)},2) node[right] {$\theta=t\widehat a$};

            \draw[->] (0,-2)--(0,2);
            \draw[->] (-0.2,0)--(3.5,0);
            \draw (3.6,0) node[right] {$a$};
            \draw (0,2.1) node[above] {$\theta$};
            
        \end{tikzpicture}
        \caption{Representation of the bulk $\mathcal{B}_t$}
        \label{figure_bulk}
\end{figure}
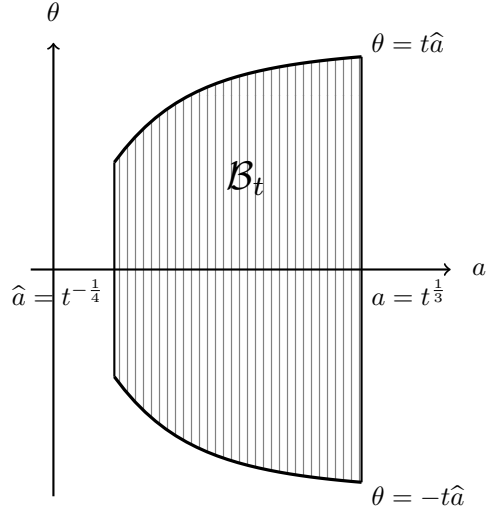
In this subset, we derive more precise estimates for $\widetilde{R}$, the most important being that $\widetilde{R}\sim t\widehat{a}$ in the bulk. Note that for any fixed $(\theta,a,\ell)$,$(\theta,a,\ell)\in\mathcal{B}_t$ for all $t\geq T_0$ sufficiently large. Conceptually, this means that for large times, the quantities will remain in the bulk and the dynamics happen in this subset. One can also remark that in the complement of the bulk,
\begin{equation}
    \1_{\mathcal{B}_t^c}\lesssim t^{-k}(a^{3k} +a^{-4k}+|\theta|^k(1+a^{-k}) +\ell^{2k}),
\end{equation}
so that any sufficiently localized function decays strongly in $\mathcal{B}_t^c$. In fact, in what follows, we will often split the integrals into regions inside and outside the bulk.

\subsubsection{Estimates on the radius composed by the linear flow}

\begin{proposition}
    \label{proposition_derivatives_R_tilde}
    The derivatives of $\widetilde{R}$ are given by
    \begin{equation}
        \label{equation_expression_derivative_R_tilde}
        \begin{array}{l}
                        \begin{array}{ll}
                \displaystyle\partial_\theta \widetilde{R}=(\partial_\theta R)(\theta+t\widehat a,a),&\qquad \partial_a \widetilde{R}=t\partial_a\widehat{a}(\partial_\theta R)(\theta+t\widehat a,a)+(\partial_a R)(\theta+t\widehat a,a),  \vspace{2pt} \\
            \displaystyle \partial^2_\theta\widetilde{R}=(\partial_\theta^2 R)(\theta+t\widehat a,a),&\qquad \partial_a\partial_\theta \widetilde{R}=t\partial_a\widehat{a}(\partial^2_\theta R)(\theta+t\widehat a,a)+(\partial_a\partial_\theta R)(\theta+t\widehat a,a),\vspace{2pt}
            \end{array}\\
            \displaystyle \partial_a^2\widetilde{R}=t\partial_a^2\widehat a(\partial_\theta R)(\theta+t\widehat a,a)+2t\partial_a\widehat{a}(\partial_a\partial_\theta R)(\theta+t\widehat a,a)+t^2(\partial_a \widehat a)^2(\partial^2_\theta R)(\theta+t\widehat a,a)+(\partial_a^2 R)(\theta+t\widehat a,a).
        \end{array}
    \end{equation}
    Moreover, we have the following estimates.
    \begin{equation}
        \label{equations_estimates_R_tilde_derivatives}
        \begin{array}{c}
            \displaystyle  |\partial_\theta\widetilde{R}|\leq 1,\qquad |\partial_a\widetilde{R}|\lesssim t\frac{1}{(a^0)^3}+\left(\frac{1}{a}+\frac{\sqrt{\ell}}{a^2}+\frac{1}{a^3}\right)\log\left\langle \frac{\widetilde{R}}{p}\right\rangle, \vspace{3pt}\\
            \displaystyle |\partial_\theta^2\widetilde{R}|\lesssim \frac{1}{\widetilde{R}-\frac{1}{a^0}}\lesssim \frac{a^0}{\widetilde{R}},\qquad |\partial_a\partial_\theta\widetilde{R}|\lesssim t\frac{1}{\widetilde{R}(a^0)^2}+\frac{a^0}{\widetilde{R}a^3}+\frac{\sqrt{\ell}}{\widetilde{R}a^2},\vspace{5pt}\\
            \displaystyle |\partial_a^2\widetilde{R}|\lesssim t\left[\frac{a}{(a^0)^5}+\frac{1}{(a^0)^3}\left(\frac{a^0}{\widetilde{R}a^3}+\frac{\sqrt{\ell}}{\widetilde{R}a^2}\right)\right]+t^2\frac{1}{\widetilde{R}(a^0)^5}+\left(\frac{1}{a^2}+\frac{\sqrt{\ell}}{a^3}+\frac{1}{a^4}\right)\log\left\langle \frac{\widetilde{R}}{p}\right\rangle.
        \end{array}
    \end{equation}
    Finally, inside the bulk $\mathcal{B}_t$ defined in \eqref{equation_def_bulk}, provided that $t \geq C(m)$ for a constant $C(m)$ depending only on $m$, the following holds
    \begin{equation}
        \label{equation_lower_bound_R_tilde_bulk}
        \frac{t\widehat a}{4}\leq |\widetilde{R}|\leq 2t\widehat a,\qquad \partial_a \widetilde{R}\geq \frac{3}{4}t\frac{m^2}{(a^0)^3}.
    \end{equation}
\end{proposition}
\begin{proof}
    First, note that
    \begin{equation*}
        |\partial_a \widehat{a}|=\frac{m^2}{(a^0)^3},\qquad |\partial_a^2\widehat{a}|=3\frac{m^2 a}{(a^0)^5}.
    \end{equation*}
    Consequently, estimates \eqref{equations_estimates_R_tilde_derivatives} follow directly from \eqref{equation_expression_derivative_R_tilde} and Proposition \ref{proposition_estimates_R}. Then, by \eqref{equation_estimates_G_a_H_a}, for $(\theta,a,\ell)\in\mathcal{B}_t$, 
    \begin{equation*}
        \frac{|\theta+t\widehat a|}{2}\leq \widetilde{R}(\theta,a)\leq |\theta+t\widehat a|+p+\frac{a^0}{a^2}\leq \frac{3t\widehat{a}}{2}+p+\frac{a^0}{a^2}.
    \end{equation*}
    In the bulk $\mathcal{B}_t$, we have $\widehat a\geq t^{-\frac{1}{4}}$, $\sqrt{\ell}\leq t^{\frac{1}{4}}$ and $a\geq m t^{-\frac{1}{4}}$. Therefore, 
    \begin{equation*}
        p+\frac{a^0}{a^2}\leq t\widehat a\left(\frac{\sqrt{\ell}}{t\widehat a a}+\frac{m}{t\widehat a a^2}+\frac{1}{t(\widehat a)^2 a}\right)\leq t\widehat a\frac{3}{t^{\frac{1}{4}}m}.
    \end{equation*}
    For $t\geq \left(\frac{6}{m}\right)^4$, this implies $\frac{a^0}{a^2}+p\leq \frac{t\widehat a}{2}$. Since $|\theta|\leq \frac{1}{2}t\widehat a$, we obtain the first estimate of \eqref{equation_lower_bound_R_tilde_bulk}. We now focus on the lower bound of $\partial_a \widetilde{R}$. By \eqref{equation_expression_derivative_R_tilde}, we have
    \begin{equation*}
        \partial_a \widetilde{R}=t\frac{m^2}{(a^0)^3}\left((\partial_\theta R)(\theta+t\widehat a,a)+\frac{(a^0)^3}{tm^2}(\partial_a R)(\theta+t\widehat a,a)\right).
    \end{equation*}
    However, by Proposition \ref{proposition_estimates_R}, we know that $\partial_\theta R=O(1)$ and, in the bulk $\mathcal{B}_t$, $\frac{(a^0)^3}{t}\partial_aR=O(t^{-1/4}\log\langle Ra^2\rangle)$. Hence, formally, if $t$ is sufficiently large,  $\frac{(a^0)^3}{t}\partial_aR$ is dominated by $\partial_\theta R$ and we will obtain the result. More precisely, recall \eqref{equation_derivative_theta_R_lower_1_proof}, so that 
    \begin{align*}
        (\partial_\theta R)(\theta+t\widehat a,a)=\sqrt{1-\frac{(a^0)^2\frac{\ell}{\widetilde{R}^2}+2m^2\frac{a^0}{\widetilde{R}}-\frac{m^2}{\widetilde{R}^2}}{a^2(a^0-\frac{1}{\widetilde{R}})^2}}\geq 1-\frac{(a^0)^2\ell}{a^2\widetilde{R}^2(a^0-\frac{1}{\widetilde{R}})^2}-\frac{1}{\widetilde{R}}\frac{2m^2 a^0}{a^2(a^0-\frac{1}{\widetilde{R}})^2}.
    \end{align*}
    However, by \eqref{equation_lower_bound_R_a0_proof}
    \begin{equation*}
        \frac{(a^0)^2\ell}{a^2\widetilde{R}^2(a^0-\frac{1}{\widetilde{R}})^2}= \frac{\ell a^0}{\widetilde{R}a^2\left(a^0-\frac{1}{\widetilde{R}}\right)\left(\widetilde{R}-\frac{1}{a^0}\right)}\leq \frac{2}{m\widetilde{R}}\frac{\sqrt{\ell}a^0}{a}.
    \end{equation*}
    Similarly,
    \begin{equation*}
        \frac{1}{\widetilde{R}}\frac{2m^2 a^0}{a^2(a^0-\frac{1}{\widetilde{R}})^2}\leq \frac{8a^0}{\widetilde{R} a^2}.
    \end{equation*}
    Since $\widehat{a}\geq t^{-\frac{1}{4}},\,a\geq mt^{-\frac{1}{4}},\, \sqrt{\ell}\leq t^{\frac{1}{4}},\, \widetilde{R}\geq \frac{t\widehat{a}}{4}$ in the bulk, this implies 
    \begin{equation*}
        (\partial_\theta R)(\theta+t\widehat a,a)\geq 1-\frac{2}{m\widetilde{R}}\frac{\sqrt{\ell}a^0}{a}-\frac{8a^0}{\widetilde{R} a^2}\geq 1-\frac{8}{mt^{1/4}}-\frac{32}{mt^{1/4}}\geq 1-\frac{40}{mt^{1/4}}.
    \end{equation*}
    We now turn to $\partial_aR$. Recalling the expressions of $\partial_a R$ in \eqref{equation_derivative_a_R_other_expression_proof} and $\partial_\theta R$ in Proposition \ref{proposition_writing_derivatives_R}, we have
    \begin{align*}
        (\partial_aR)(\theta+t\widehat a,a)=&~~\partial_ap\left(\frac{\widetilde{R}}{p}-\kappa-\sqrt{\left(\frac{\widetilde{R}}{p}-\kappa\right)^2-1}\sign(\theta+t\widehat a)\partial_\theta R\right)\\
        &-\partial_a\left(\kappa p-\frac{1}{a^0}\right)\arcosh\left(\frac{\widetilde{R}}{p}-\kappa\right)\sign(\theta+t\widehat a)\partial_\theta R+\partial_a(\kappa p)\\
        =&-\kappa\partial_ap+\frac{\partial_a p}{\widetilde{R}-\frac{1}{a^0}}\left(\left(2\kappa-\frac{1}{a^0p}\right)\widetilde{R}+p(1-\kappa^2)\right)\\
        &-\partial_a\left(\kappa p-\frac{1}{a^0}\right)\arcosh\left(\frac{\widetilde{R}}{p}-\kappa\right)\sign(\theta+t\widehat a)\partial_\theta R+\partial_a(\kappa p).
    \end{align*}
    However, by Lemma \ref{lemma_derivatives_of_p_kappa}, we know that $\kappa\partial_ap<0$ and $\partial_a\left(\kappa p-\frac{1}{a^0}\right)<0$. Similarly, $(\partial_ap)p(1-\kappa^2)\geq p\partial_ap$ and $\frac{\partial_ap}{a^0p}< 0$. This implies
    \begin{equation*}
        (\partial_aR)(\theta+t\widehat a,a)\geq \frac{\partial_a p}{\widetilde{R}-\frac{1}{a^0}}\left(2\kappa \widetilde{R}+p\right)+\partial_a(\kappa p).
    \end{equation*}
    Moreover, by \eqref{equation_estimates_derivatives_kappa_p} and \eqref{equation_lower_bound_R_a0_proof}
    \begin{equation*}
        \frac{p\partial_ap}{\widetilde{R}-\frac{1}{a^0}}\geq -\frac{a^0}{a^3}\frac{\ell a^2+2m^2}{m^2+a^0\sqrt{m^2+\ell a^2}}\geq -2\frac{\sqrt{\ell a^2+m^2}}{a^3}\geq -2\frac{\sqrt{\ell}}{a^2}-2\frac{m}{a^3}.
    \end{equation*}
    Since $\frac{\partial_a p}{p}\geq -\frac{2}{a}$ and $\kappa=\frac{a^0}{p a^2}$, using \eqref{equation_lower_bound_R_a0_proof} we also derive  
    \begin{equation*}
        \frac{2\kappa\partial_a p}{1-\frac{1}{a^0\widetilde{R}}}=\frac{2(a^0)^2}{a^2}\frac{\partial_a p}{p}\frac{1}{a^0-\frac{1}{\widetilde{R}}}\geq -\frac{8}{m}\frac{(a^0)^2}{a^3}\geq -\frac{8}{ma}-\frac{8m}{a^3}
    \end{equation*}
    Hence, since $\widehat{a}\geq t^{-\frac{1}{4}},\, t^{\frac{1}{3}}\geq a\geq mt^{-\frac{1}{4}},\, \sqrt{\ell}\leq t^{\frac{1}{4}}$ in the bulk $\mathcal{B}_t$, we find for $t\geq \left(\frac{6}{m}\right)^4$
    \begin{align*}
        \frac{(a^0)^3}{tm^2}(\partial_a R)(\theta+t\widehat a,a)\geq&-2\frac{(a^0)^3\sqrt{\ell}}{tm^2a^2}-2\frac{1}{tm(\widehat a)^3}-\frac{8}{tm^3}\frac{(a^0)^3}{a}-\frac{8}{tm^3(\widehat{a})^3}\\
        \geq& -8\frac{a\sqrt{\ell}}{tm^2}-8\frac{m\sqrt{\ell}}{ta^2}- 2\frac{1}{tm(\widehat a)^3}-8\frac{1}{tm^3(\widehat a)^3}-32\frac{a^2}{tm^3}-32\frac{1}{tma^2}\\
        \geq & -\frac{8}{m^2}t^{-5/12}-\frac{8}{m}t^{-1/4}-\frac{2}{m}t^{-1/4}-\frac{8}{m^3}t^{-1/4}-\frac{32}{m}t^{-1/3}-\frac{32}{m^3} t^{-1/2}\\
        \geq & -C(m)t^{-1/4}.
    \end{align*}
    We proved that 
    \begin{equation*}
        \partial_a \widetilde{R}\geq t\frac{m^2}{(a^0)^3}\left(1-C(m)t^{-\frac{1}{4}}\right),
    \end{equation*}
    this directly implies \eqref{equation_lower_bound_R_tilde_bulk}.
\end{proof}

\subsection{A preliminary lemma}
\label{section_preliminary_lemma}
Recall from \eqref{equation_definition_rho} that
\begin{align*}
    \rho(t,r)=4\pi^2\iiint\delta\left(\widetilde{R}(\theta,a)-r\right)\gamma(t,\theta,a,\ell)\mathrm{d}\theta\mathrm{d}a\mathrm{d}\ell.
\end{align*}
Hence, in order to show estimates on $\rho(t,r)$, we need to study the equation
\begin{equation}
    \label{equation_curve_R_r}
    \widetilde{R}(\theta,a)=r,
\end{equation}
for a fixed $(t,r,\ell)$. To do so, we divide the reduced action-angle phase space $\R\times\R_+^*$ into 3 domains $(\mathcal{R}_j)_{j=1,2,3}$, which can be seen in Figure \ref{figure_separation_domains} below. In each domain, we can exhibit a diffeomorphism with good properties. 

\begin{figure}[!h]
    \centering
    \begin{tikzpicture}[thick,scale=1, every node/.style={scale=1}]
            \draw[-] (1,-2.9)--(1,2.9);
            \fill [pattern=horizontal lines, pattern color=blue]
            (0,-2.9)--(1,-2.9)--(1,2.9)--(0,2.9)--cycle;
            \draw (0.5,1.5) node[ultra thick]{$\mathcal{R}_0$};
            
            \fill [pattern=north west lines, pattern color=red]
            (1,-2.9)--plot [domain=1:5] (\x,{-\x/(sqrt(1+\x^2)})--(5,-2.9);
            \draw (3,-2) node {$\mathcal{R}_1$};

            \fill [pattern=vertical lines, pattern color=green]
            (1,2.9)--plot [domain=1:5] (\x,{-\x/(sqrt(1+\x^2)})--(5,2.9);
            \draw (3,1) node {$\mathcal{R}_2$};

            \draw[domain=0:5, very thick]   plot (\x,{-\x/(sqrt(1+\x^2)}) ;
            \draw (5,-1) node[right] {$\theta=-t\widehat a$};

             \draw (1,-0.2) node[right] {$A$};
            \draw[->] (0,-3)--(0,3);
            \draw[->] (-0.2,0)--(5.1,0);
            \draw (5.1,0) node[right] {$a$};
            \draw (0,3.1) node[above] {$\theta$};
            
        \end{tikzpicture}
        \caption{Representation of the domains $\mathcal{R}_j$}
        \label{figure_separation_domains}
\end{figure}
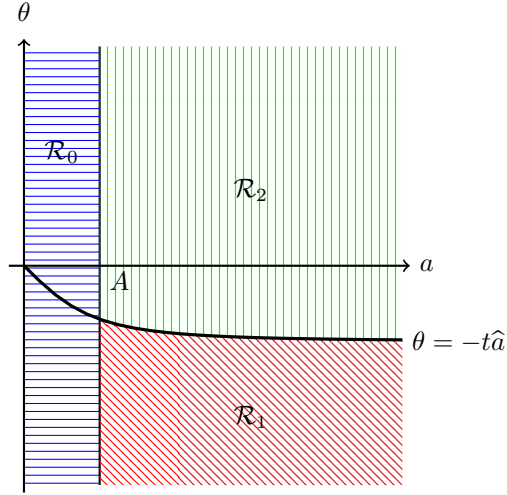

\begin{lemma}
    \label{lemma_splitting_different_R_j}
    Let $(t,r,\ell)\in [T_0,\infty)\times\R^*_+\times\R_+$ be fixed and 
    \begin{equation}
        \label{equation_expression_hbar_A}
        A:=F^{-1}(r-\hbar),\qquad \hbar:=\frac{C}{t^4}\frac{r^2}{(r+1+\sqrt{\ell})^2},
    \end{equation}
    where $F^{-1}$ denotes the inverse function of $F(a):=\frac{a^0+\sqrt{m^2+\ell a^2}}{a^2}=\kappa p +p$ and $C>0$ is a constant depending only on $m$.
    We consider the subspaces 
    \begin{equation*}
        \begin{array}{c}
             \displaystyle\mathcal{R}_0:=\{(\theta,a)\in \R\times\R^*_+\,|\, a\leq A\}, \vspace{5pt}\\
              \displaystyle\mathcal{R}_1:=\{(\theta,a)\in \R\times\R^*_+\,|\, a> A,\,\theta< -t\widehat a\},\quad \mathcal{R}_2:=\{(\theta,a)\in \R\times\R^*_+\,|\, a> A,\,\theta> -t\widehat a\}.
        \end{array}
    \end{equation*}
    In $\mathcal{R}_0$, for any fixed $\theta\in\R$, there exists at most one $a:=\aleph(\theta;t,r,\ell)$ solution to \eqref{equation_curve_R_r}. Moreover, it satisfies 
    \begin{equation}
        \label{equation_R_lower_bound_in_R0}
        |\partial_a\widetilde{R}(\theta,\aleph)|\geq \frac{3m}{4\aleph^3}.
    \end{equation}
    In $\mathcal{R}_1$ (resp. $\mathcal{R}_2$) for each $a>A$, there exists a unique $\theta:=\theta_1(a;t,r,\ell)$ (resp. $\theta:=\theta_2(a;t,r,\ell)$) solution to \eqref{equation_curve_R_r}. Moreover, we have the following estimates
    \begin{equation}
        \label{equation_lower_bound_in_R_j}
        \frac{1}{|\partial_\theta\widetilde{R}(\tau_j,a)|}\lesssim t^3 \frac{a^2}{a^0}+t^2\left(a|\theta_j|+\sqrt{\ell}(1+a)+a^{-1}+a\right),\qquad j=1,2.
    \end{equation}
\end{lemma}
\begin{proof}
    First, recall that $\widetilde{R}$ satisfies
    \begin{equation*}
\widetilde{R}(\theta,a)\geq \frac{a^0+\sqrt{m^2+\ell a^2}}{a^2}=\kappa p +p =F(a).
    \end{equation*}
    Hence, for \eqref{equation_curve_R_r} to be satisfied, we need to ensure $r\geq F(a)$. Moreover, $F$ is continuous and decreasing, so we can define its inverse $F^{-1}$, which is also decreasing. This gives us a necessary condition on $a$ for \eqref{equation_curve_R_r}, given by 
    \begin{equation}
        a\geq F^{-1}(r).
    \end{equation}
    \underline{\textbf{Study of $\mathcal{R}_0$}}\\

    Let us begin by rewriting \eqref{equation_curve_R_r}. We have
    \begin{equation*}
        \widetilde{R}(\theta,a)=r\Leftrightarrow |\theta+t\widehat a|=pG_a\left(\frac{r}{p}-\kappa\right).
    \end{equation*}
    Let $x:=\frac{|\theta+t\widehat a|}{p}$ and $h:=r-F(a)$ be such that $a=F^{-1}(r-h)$. Since $F(a)=\kappa p +p$, the above equation rewrites 
    \begin{align*}
        x&=G_a\left(1+\frac{h}{p}\right).
    \end{align*}
    Then, using the asymptotic expansion of $G_a$ given in Proposition \ref{proposition_asymptotic_expansion_G_H}, we obtain
    \begin{equation}
        \label{equation_rewritten_domain_R_0}
        x-\sqrt{2}\left(1+\frac{m^2}{a^0\sqrt{m^2+\ell a^2}}\right)\sqrt{\frac{h}{p}}+O_{h\rightarrow0}\left(\frac{h}{p}\right)^{3/2}=0.
    \end{equation}
    Hence, for $\frac{h}{p}$ small enough, say $\frac{h}{p}\leq c$ where $c>0$ is a small constant, we know that there exists at most one solution to \eqref{equation_rewritten_domain_R_0}. We will denote it by $\aleph(\theta;t,r,\ell)$ and consider $a=\aleph(\theta;t,r,\ell)$. In that case, $x=\frac{|\theta+t\widehat a|}{p}$ satisfies $0\leq x\leq 4\sqrt{\frac{h}{p}}$. Now we recall that 
    \begin{equation*}
        \partial_a\widetilde{R}=t\frac{m^2}{(a^0)^3}(\partial_\theta R)(\theta+t\widehat a,a)+(\partial_a R)(\theta+t\widehat a,a),
    \end{equation*}
    with
    \begin{equation*}
        (\partial_\theta R)(\theta+t\widehat a,a)=\sign(\theta+t\widehat a)H_a'\left(\frac{|\theta+t\widehat a|}{p}\right)=\sign(\theta+t\widehat a)H_a'(x).
    \end{equation*}
    Similarly, recall that 
    \begin{equation*}
        \frac{\widetilde{R}(\theta,a)}{p}-\kappa= H_a(x),\qquad \frac{|\theta+t\widehat{a}|}{p}=G_a(H_a(x))=\sqrt{H_a(x)^2-1}+\left(\kappa-\frac{1}{a^0 p}\right)\arcosh (H_a(x)).
    \end{equation*}
    Recalling the expression of $\partial_aR$ from Proposition \ref{proposition_writing_derivatives_R}, we use \eqref{equation_proof_equality_R_theta_H_a_prime} and Lemma \ref{lemma_derivatives_of_p_kappa} to derive
    \begin{align*}
        (\partial_aR)(\theta+t\widehat a,a)=&~\partial_a p \left(H_a(x)-\frac{|\theta+t\widehat{a}|}{p}H_a'(x)\right)-p\partial_a\left(\kappa-\frac{1}{a^0 p}\right)\arcosh(H_a(x))H_a'(x)+\partial_a(\kappa p)\\
        =&-\left(\partial_a p \sqrt{H_a(x)^2-1}+\partial_a \left(\kappa p-\frac{1}{a^0}\right)\arcosh(H_a(x))\right)H_a'(x)+\partial_a p H_a(x)+\partial_a(\kappa p)\\
        =&\left[\frac{\ell a^2+2m^2 }{a^3\sqrt{m^2+\ell a^2}}\sqrt{H_a^2\left(x\right)-1}+\frac{m^2(3a^2+2m^2)}{(a^0)^3a^3}\arcosh\left(H_a\left(x\right)\right)\right]H_a'\left(x\right)\\
        &-\frac{\ell a^2+2m^2}{a^3\sqrt{m^2+\ell a^2}}H_a\left(x\right)-\frac{a^2+2m^2}{a^0a^3}.
\end{align*}
    Hence,
    \begin{align}
        \partial_a \widetilde{R}=&\left[t\frac{m^2}{(a^0)^3}\sign(\theta+t\widehat a)+\frac{m^2(3a^2+2m^2)}{(a^0)^3a^3}\arcosh\left(H_a(x)\right)\right]H_a'(x)\\
        &+\frac{\ell a^2+2m^2}{a^3\sqrt{m^2+\ell a^2}}\left(\sqrt{H_a^2(x)-1}H_a'(x)-H_a(x)\right)-\frac{a^2+2m^2}{a^0a^3}.\nonumber
    \end{align}
    Moreover, since $H_a(x)\leq x+1$, we have
    \begin{equation*}
        H_a'(x)=\frac{H_a(x)}{H_a(x)+\kappa-\frac{1}{a^0p}}\sqrt{1-\frac{1}{H_a^2(x)}}\leq \sqrt{x}\frac{\sqrt{2+x}}{1+x}\leq \sqrt{2x},\qquad \arcosh(H_a(x))\leq \sqrt{H_a(x)^2-1}\leq \sqrt{x}\sqrt{2+x}.
    \end{equation*}
    Consequently, if we choose $x$ small enough, the positive terms in the above equations will be dominated by the negative terms, e.g. $-\frac{a^2+2m^2}{a^0a^3}$. More precisely, we find
    \begin{align*}
        \partial_a\widetilde{R}&\leq \left[\frac{tm^2}{a^3}+\frac{4m}{a^3}\arcosh(H_a(x))\right]H_a'(x)+\frac{\ell a^2+2m}{ma^3}\left(\sqrt{H_a^2(x)-1}H_a'(x)-H_a(x)\right)-\frac{m}{a^3}\\
        &\leq \frac{m}{a^3}\left(3\sqrt{x}\sqrt{2+x}+tm\right)\sqrt{2x}+\frac{\ell a^2 +2m}{ma^3}H_a(x)\left(\sqrt{2x}-1\right)-\frac{m}{a^3}.
    \end{align*}
    Hence, if $x$ satisfies 
    \begin{equation}
        \label{equation_x_needs_to_satisfy}
        (3\sqrt{x}\sqrt{2+x}+tm)\sqrt{2x}\leq \frac{1}{4},\qquad x\leq \frac{1}{2},
    \end{equation}
    we obtain the desired estimate \eqref{equation_R_lower_bound_in_R0}. Note that \eqref{equation_x_needs_to_satisfy} holds for 
    \begin{equation*}
        x\leq \frac{1}{32}\frac{1}{(2+tm)^2}.
    \end{equation*}
    Consequently, since $x\leq 4\sqrt{\frac{h}{p}}$, \eqref{equation_x_needs_to_satisfy} is satisfied for 
    \begin{equation*}
        \frac{h}{p}\leq \frac{1}{2^{14}}\frac{1}{(2+tm)^4}.
    \end{equation*}
    However, as $p\geq \frac{m}{a^2}$ and $a=F^{-1}(r-h)$,
    \begin{equation*}
        \frac{h}{p}\leq \frac{(F^{-1}(r-h))^2h}{m},
    \end{equation*}
    and, since $F(a)=\frac{a^0+\sqrt{m^2+\ell a^2}}{a^2}$, for all $0<\varepsilon\leq 1$,
    \begin{align*}
        \frac{(F^{-1}(r-h))^2h}{m}\leq \varepsilon & \Longleftrightarrow r\geq h +\frac{h}{\varepsilon m}\left(\sqrt{m^2+\frac{\varepsilon m}{h}}+\sqrt{m^2+\ell\frac{\varepsilon m}{h}}\right)\\
        &\Longleftrightarrow r\geq \frac{1}{m}\sqrt{\frac{hm}{\varepsilon}}\left(\varepsilon\sqrt{\frac{hm}{\varepsilon}}+\sqrt{1+\frac{hm}{\varepsilon}}+\sqrt{\ell+\frac{hm}{\varepsilon}}\right).
    \end{align*}
    Hence, if $\varepsilon\leq 1$ and $h$ satisfies 
    \begin{equation*}
        h\leq \frac{m\varepsilon r^2}{(3mr+1+\sqrt{\ell})^2},
    \end{equation*}
    we obtain
    \begin{equation*}
        \varepsilon \sqrt{\frac{h m}{\varepsilon}}\leq mr,\qquad \sqrt{1+\frac{hm}{\varepsilon}}\leq \sqrt{1+m^2r^2}\leq 1+ mr,\qquad \sqrt{\ell+\frac{hm}{\varepsilon}}\leq \sqrt{\ell+m^2r^2}\leq \sqrt{\ell} + mr.
    \end{equation*}
    Consequently,
    \begin{align*}
        \frac{1}{m}\sqrt{\frac{hm}{\varepsilon}}\left(\varepsilon\sqrt{\frac{hm}{\varepsilon}}+\sqrt{1+\frac{hm}{\varepsilon}}+\sqrt{\ell+\frac{hm}{\varepsilon}}\right)&\leq \frac{r}{3mr +1+\sqrt{\ell}}(3mr+1+\sqrt{\ell})\leq r,
    \end{align*}
    and then
    \begin{equation*}
        \frac{h}{p}\leq \varepsilon.
    \end{equation*}
    In conclusion, if $h$ satisfies
    \begin{equation*}
        h\leq \frac{C}{t^4}\frac{r^2}{(r+1+\sqrt{\ell})^2}=:\hbar,
    \end{equation*}
    we obtain \eqref{equation_R_lower_bound_in_R0}. Here, $C$ is a small constant depending only on $m$ and $c$ defined above.

    \underline{\textbf{Study of $\mathcal{R}_1$ and $\mathcal{R}_2$}}\\
    
    In the following, let $j=1,2$. We know that \eqref{equation_curve_R_r} is equivalent to
    \begin{equation*}
        |\theta+t\widehat a|=pG_a\left(\frac{r}{p}-\kappa\right).
    \end{equation*}
    By definition of $\mathcal{R}_1$ and $\mathcal{R}_2$, in each domain there is a unique solution $\theta_j$ to \eqref{equation_curve_R_r} given by
    \begin{equation*}
        \theta_1:=-t\widehat a -pG_a\left(\frac{r}{p}-\kappa\right),\qquad \theta_2:=-t\widehat a +pG_a\left(\frac{r}{p}-\kappa\right).
    \end{equation*}
    Then, since $r=\widetilde{R}(\theta_j,a)$ and $F(a)=\kappa p +p$, by Proposition \ref{proposition_asymptotic_expansion_G_H}, we have
    \begin{equation*}
        \partial_\theta \widetilde{R}(\theta_j,a)=\frac{\sqrt{\left(\frac{r}{p}-\kappa\right)^2-1}}{\frac{r}{p}-\frac{1}{a^0 p}}=\frac{\sqrt{r-F(a)}\sqrt{r-\kappa p +p}}{r- \frac{1}{a^0}}.
    \end{equation*}
    However, in $\mathcal{R}_j$, we know that $a>F^{-1}(r-\hbar)$, so that $r-F(a)\geq \hbar$. Recalling \eqref{equation_expression_hbar_A} along with the lower bound $\widetilde{R}(\theta_j,a)-\kappa p\geq p$, we derive 
    \begin{equation*}
        \frac{1}{|\partial_\theta\widetilde{R}(\theta_j,a)|}\lesssim \frac{r}{\sqrt{\hbar p}}\lesssim t^2a\big(r+1+\sqrt{\ell}\big).
    \end{equation*}
    However, since $H_a(x)\leq x+1$, we find
    \begin{equation*}
        r=\widetilde{R}(\theta_j,a)=pH_a\left(\frac{|\theta_j+t\widehat{a}|}{p}\right)+\kappa p \lesssim |\theta_j|+t\widehat a+a^{-2}+\sqrt{\ell}a^{-1}.
    \end{equation*}
    This gives us the result \eqref{equation_lower_bound_in_R_j}.
\end{proof}

\subsection{The potential and its derivatives}
\label{section_potential_derivatives}

Let us begin by computing the derivatives of the potential $\widetilde{\Psi}$. For any $\alpha,\beta\in\{\theta,a\}$, we have
\begin{equation}
    \label{equation_def_derivatives_psi}
    \partial_\alpha\widetilde{\Psi}=-\partial_\alpha\widetilde{R}\frac{M(t,\widetilde{R})}{\widetilde{R}^2},\qquad \partial_\alpha\partial_\beta\widetilde{\Psi}=-\frac{M(t,\widetilde{R})}{\widetilde{R}^2}\left(\partial_\alpha\partial_\beta \widetilde{R}-2\frac{\partial_\alpha\widetilde{R}\partial_\beta\widetilde{R}}{\widetilde{R}}\right)-\rho(t,\widetilde{R})\frac{\partial_\alpha\widetilde{R}\partial_\beta\widetilde{R}}{\widetilde{R}^2}.
\end{equation}
In order to give proper estimates of these derivatives, we first derive upper bounds for $M$ and $\rho$.

\subsubsection{Estimates for the enclosed mass and the charge density}
\begin{lemma}
    \label{lemma_estimate_total_mass}
    Let $M_I, M_O$ be defined as 
    \begin{equation*}
        M_I(t,r):=4\pi^2\iiint \1_{\mathcal{B}_t}\1_{\{\widetilde{R}(\theta,a)\leq r\}}\gamma(t,\theta,a,\ell)\mathrm{d}\theta\mathrm{d}a\mathrm{d}\ell,\quad M_O(t,r):=4\pi^2\iiint \1_{\mathcal{B}_t^c}\1_{\{\widetilde{R}(\theta,a)\leq r\}}\gamma(t,\theta,a,\ell)\mathrm{d}\theta\mathrm{d}a\mathrm{d}\ell.
    \end{equation*}
    For all $t\geq T_0, r>0$, and  $k,\sigma \geq 0$, the enclosed mass $M$ satisfies $M(t,r)=M_I(t,r)+M_O(t,r)$ and 
    \begin{equation}
    \notag
    \begin{array}{rl}
        \displaystyle \frac{M_I(t,r)}{r^k} \lesssim & t^{-k} \left(\|a^{-k}\gamma\|_{L^1} +\| \gamma\|_{L^1} \right),\vspace{3pt}\\
         \displaystyle \frac{M_O(t,r)}{r^k} \lesssim & t^{-k-\sigma}\left(\|a^{-2k-4\sigma}\gamma\|_{L^1}+\|a^{4k+3\sigma}\gamma\|_{L^1}+\|a^{k}\ell^{2k+2\sigma}\gamma\|_{L^1}+\||\theta|^{k+\sigma}(a^{k-\sigma}+a^k)\gamma\|_{L^1}\right).
    \end{array}
    \end{equation}
    In the case $r=\widetilde{R}(\theta,a)$, we can obtain a more specific estimate.
    \begin{equation}
    \notag
    \begin{array}{r@{}l}
        \displaystyle \frac{M(t,\widetilde{R})}{{\widetilde{R}}^k(\widetilde{R}-\frac{1}{a^0})} \lesssim & ~t^{-k} \left(\|a^{-k}\gamma\|_{L^1} +\| a^2\gamma\|_{L^1} \right)\vspace{3pt}\\
         \displaystyle & +~ t^{-k-\sigma}\left(\|(a^{-2k-4\sigma}+a^{4k+3\sigma+2})\gamma\|_{L^1}+\|\ell^{2k+2\sigma}(a^k+a^{k+2})\gamma\|_{L^1}+\||\theta|^{k+\sigma}(a^{k-\sigma}+a^{k+2})\gamma\|_{L^1}\right).
    \end{array}
    \end{equation}
\end{lemma}
\begin{proof}
    We first study $M_I$. We know that inside the bulk $\frac{t\widehat a}{2}\leq\widetilde{R}\leq 2t\widehat a$. Consequently, we obtain
    \begin{align*}
        M_I(t,r)r^{-k}&\lesssim \iiint \1_{\mathcal{B}_t}\widetilde{R}(\theta,a)^{-k}\gamma(t,\theta,a,\ell) \mathrm{d}\theta\mathrm{d}a\mathrm{d}\ell\lesssim t^{-k}\iiint a^{-k}(a^0)^k\gamma(t,\theta,a,\ell) \mathrm{d}\theta\mathrm{d}a\mathrm{d}\ell\\
        &\lesssim t^{-k}(\|a^{-k}\gamma\|_{L^1}+\|\gamma\|_{L^1}).
    \end{align*}
    To estimate $M_O$, we observe from \eqref{equation_lower_bound_R} that, for any $k, \sigma\geq 0$,
    \begin{align}
        \label{equation_R_tilde_in_complement_bulk_proof}
        \widetilde{R}^{-k}\1_{\mathcal{B}_t^c}&\lesssim a^{2k}(a^0)^{-k}\1_{\{t\leq (a^{-1}a^0)^{4}\}}+a^k\1_{\{t\leq a^{3}\}}+a^k\1_{\{t\leq\ell^2\}}+a^{2k}(a^0)^{-k}\1_{\{t\leq 2|\theta|a^{-1}a^0\}}\\
       \notag &\lesssim t^{-k-\sigma}\left(a^{-2k-4\sigma}(a^0)^{3k+4\sigma}+a^{4k+3\sigma}+a^k\ell^{2k+2\sigma}+a^{k-\sigma}(a^0)^\sigma|\theta|^{k+\sigma}\right)\\
       \notag &\lesssim t^{-k-\sigma}\left(a^{-2k-4\sigma} +a^{4k+3\sigma}+\ell^{2k+2\sigma}a^k+|\theta|^{k+\sigma}(a^{k-\sigma}+a^k)\right).
    \end{align}
    Hence, 
    \begin{align*}
        M_O(t,r)r^{-k}&\lesssim \iiint\1_{\{\widetilde{R}(\theta,a)\leq r\}}\widetilde{R}^{-k}\1_{\mathcal{B}_t^c}\gamma(t,\theta,a,\ell)\mathrm{d}\theta\mathrm{d}a\mathrm{d}\ell\\
        &\lesssim  t^{-k-\sigma}\iiint\left(a^{-2k-4\sigma} +a^{4k+3\sigma}+\ell^{2k+2\sigma}a^k+|\theta|^{k+\sigma}(a^{k-\sigma}+a^k)\right)\gamma(t,\theta,a,\ell)\mathrm{d}\theta\mathrm{d}a\mathrm{d}\ell\\
        &\lesssim  t^{-k-\sigma}\left(\|a^{-2k-4\sigma}\gamma\|_{L^1}+\|a^{4k+3\sigma}\gamma\|_{L^1}+\|a^{k}\ell^{2k+2\sigma}\gamma\|_{L^1}+\||\theta|^{k+\sigma}(a^{k-\sigma}+a^{k})\gamma\|_{L^1}\right).
    \end{align*}
    In order to derive the last estimate, we split the integral $M(t,\widetilde{R})$ into two parts by writing
    \begin{align*}
        \frac{M(t,\widetilde{R}(\theta,a))}{\widetilde{R}(\theta,a)-\frac{1}{a^0}}=&\frac{1}{\widetilde{R}(\theta,a)-\frac{1}{a^0}}\iiint\1_{\{\widetilde{R}(\theta,a)\geq\widetilde{R}(\vartheta,\alpha)\}}\1_{\{\alpha^0\leq\frac{a^0}{2}\}}\gamma(t,\vartheta,\alpha,\ell)\mathrm{d}\vartheta\mathrm{d}\alpha\mathrm{d}\ell\\
        &+\frac{1}{\widetilde{R}(\theta,a)-\frac{1}{a^0}}\iiint\1_{\{\widetilde{R}(\theta,a)\geq\widetilde{R}(\vartheta,\alpha)\}}\1_{\{\alpha^0\geq\frac{a^0}{2}\}}\gamma(t,\vartheta,\alpha,\ell)\mathrm{d}\vartheta\mathrm{d}\alpha\mathrm{d}\ell.
    \end{align*}
    In the first integral, since $\widetilde{R}(\vartheta,\alpha)\geq \frac{1}{\alpha^0}$, we have $\widetilde{R}(\theta,a)-\frac{1}{a^0}\geq  \frac{1}{2\alpha^0}> 0$. In the second integral, we use $\frac{1}{\widetilde{R}(\theta,a)-\frac{1}{a^0}}=\frac{a^0}{\widetilde{R}}\frac{1}{a^0-\frac{1}{\widetilde{R}(\theta,a)}}$ and \eqref{equation_lower_bound_R_a0_proof}. We hence have an upper bound given by
    \begin{equation*}
        \frac{M(t,\widetilde{R}(\theta,a))}{\widetilde{R}(\theta,a)-\frac{1}{a^0}}\lesssim (1+\widetilde{R}(\theta,a)^{-1})\iiint\alpha^0\1_{\{\widetilde{R}(\theta,a)\geq\widetilde{R}(\vartheta,\alpha)\}}\gamma(t,\vartheta,\alpha,\ell)\mathrm{d}\vartheta\mathrm{d}\alpha\mathrm{d}\ell.
    \end{equation*}
    Using the same arguments as in the estimate of $\frac{M(t,r)}{r^{-k}}$, we obtain the result.
\end{proof}
We can now give a similar Lemma for the charge density $\rho$.

\begin{lemma}
    \label{lemma_estimate_charge_density}
    There exist $\rho_I,\rho_O$ such that the charge density $\rho$ satisfies 
    \begin{equation*}
         \rho(t,r)=\rho_I(t,r)+\rho_O(t,r).
    \end{equation*}
    Moreover, consider $k,\sigma\geq 0$ such that
    \begin{equation}
        \label{equation_condition_k_sigma}
        k\leq 2,\qquad \sigma \leq 1,\qquad 1\leq k+\sigma \leq \frac{5}{2}.
    \end{equation}
    Then, for all $t\geq T_0, r>0$, the following estimates hold for $\rho_I$ and $\rho_O$
    \begin{equation}
    \label{equation_estimate_rho}
    \begin{array}{r@{}l}
        \displaystyle \frac{\rho_I(t,r)}{r^k} &\lesssim t^{-k-1}\big(\|(a^{-4}+a^6)\gamma\|_{L^1}+\|(a^{-2}+a^3)\partial_a\gamma\|_{L^1}\big),\vspace{3pt}\\
         \displaystyle \frac{\rho_O(t,r)}{r^k} &\lesssim t^{-k-\sigma-1}\mathcal{M},
    \end{array}
    \end{equation}
    where 
    \begin{equation}
        \label{equation_definition_mathcal_M}
        \mathcal{M}:=\|(a^{-24}+a^{24}+|\theta|^{8}+\ell^{16})(\gamma+|\gamma_a|+|\gamma_\theta|)\|_{L^1},
    \end{equation}
    and
    \begin{equation}
        \label{equation_writing_derivatives_gamma}
        \gamma_\theta :=\partial_\theta \gamma,\qquad \gamma_a:=\partial_a\gamma.
    \end{equation}
\end{lemma}
\begin{remark}
    For simplicity, we restrict ourselves to a subset of indices $k, \sigma$. However, in the proof of Lemma \ref{lemma_estimate_charge_density} we show that similar estimates hold for any $k,\sigma\geq 0$, with the right-hand side of \eqref{equation_estimate_rho} containing moments of $\gamma,\gamma_\theta,\gamma_a$ that depend on $k,\sigma$. Notice that if $k+\sigma < 1$, one obtains moments of $\gamma$ with negative exponents in $|\theta|$ and $\ell$. Therefore, in condition \eqref{equation_condition_k_sigma}, we enforce $k+\sigma \geq 1$ to prevent such issues.

\end{remark}
\begin{proof}
    Let $k,\sigma\geq 0$ be such that \eqref{equation_condition_k_sigma} holds. We begin by introducing $\chi$ a smooth non-negative function that vanishes on $\{x\leq \frac{1}{4}\}$ and is equal to $1$ on $\{x\geq \frac{1}{2}\}$. Let $\chi^1,\chi^2$ be defined as
    \begin{equation*}
        \chi^1(t,\theta,a,\ell):=\chi\left(\frac{(a^0)^3}{tm^2}\partial_a\widetilde{R}(\theta,a)\right),\qquad \chi^2(t,\theta,a,\ell):=1-\chi_1(t,\theta,a,\ell).
    \end{equation*}
    Then, for $j=1,2$ let
    \begin{equation*}
        M^j(t,r):=4\pi^2\iiint \1_{\{\widetilde{R}(\theta,a)\leq r\}}\chi^j(t,\theta,a,\ell)\gamma(t,\theta,a,\ell)\mathrm{d}\theta\mathrm{d}a\mathrm{d}\ell, 
    \end{equation*}
    and 
    \begin{equation*}
        \rho^j(t,r):=\partial_rM^j(t,r)=4\pi^2\iiint \delta(\widetilde{R}(\theta,a)- r)\chi^j(t,\theta,a,\ell)\gamma(t,\theta,a,\ell)\mathrm{d}\theta\mathrm{d}a\mathrm{d}\ell.
    \end{equation*}
    Finally, following the proof of Lemma \ref{lemma_estimate_total_mass}, we split the integral between the bulk and its complement. For instance, consider 
    \begin{equation*}
        \rho^2_I(t,r):=4\pi^2\iiint \delta(\widetilde{R}(\theta,a)- r)\1_{\mathcal{B}_t}\chi^2\gamma\mathrm{d}\theta\mathrm{d}a\mathrm{d}\ell,\qquad \rho_O^2(t,r):=4\pi^2\iiint \delta(\widetilde{R}(\theta,a)- r)\1_{\mathcal{B}_t^c}\chi^2\gamma\mathrm{d}\theta\mathrm{d}a\mathrm{d}\ell. 
    \end{equation*}
    We introduce this distinction to exploit the lower bound on $\partial_a\widetilde{R}$ given in \eqref{equation_lower_bound_R_tilde_bulk}. As a result, $\rho_I^2$ vanishes, leaving us with one less term to handle. Moreover, in the complement of the bulk, we can gain decay in $t$ at the expense of moments in $(\theta,a,\ell)$. For $\rho^1$, the decomposition requires an additional step. We first note that $\partial_a\left(\1_{\{\widetilde{R}(\theta,a)\leq r\}}\right) = -\partial_a\widetilde{R}(\theta,a)\delta(\widetilde{R}(\theta,a)- r)$, which yields
    \begin{equation*}
        \rho^1(t,r)=-4\pi^2\iiint \frac{1}{\partial_a\widetilde{R}(\theta,a)}\partial_a\left(\1_{\{\widetilde{R}(\theta,a)\leq r\}}\right)\chi^1\gamma\mathrm{d}\theta\mathrm{d}a\mathrm{d}\ell.
    \end{equation*}
    After integration by parts, we obtain
    \begin{align*}
         \rho^1(t,r)&=4\pi^2\iiint \partial_a\left(\frac{\chi^1\gamma}{\partial_a\widetilde{R}(\theta,a)}\right)\1_{\{\widetilde{R}(\theta,a)\leq r\}}\mathrm{d}\theta\mathrm{d}a\mathrm{d}\ell\\
         &=4\pi^2\iiint \partial_a\left(\frac{\chi^1\gamma}{\partial_a\widetilde{R}(\theta,a)}\right)\1_{\mathcal{B}_t}\1_{\{\widetilde{R}(\theta,a)\leq r\}}\mathrm{d}\theta\mathrm{d}a\mathrm{d}\ell+4\pi^2\iiint \partial_a\left(\frac{\chi^1\gamma}{\partial_a\widetilde{R}(\theta,a)}\right)\1_{\mathcal{B}_t^c}\1_{\{\widetilde{R}(\theta,a)\leq r\}}\mathrm{d}\theta\mathrm{d}a\mathrm{d}\ell\\
         &=:\rho^1_I(t,r)+\rho^1_O(t,r).
    \end{align*}
    We now detail the study of each term $\rho^j_I,\rho^j_O$.

    \textbf{\underline{Study of $\rho^1_I$}}\\
    
    We have
    \begin{align*}
        \rho^1_I(t,r)=&~4\pi^2\iiint \frac{\chi^1}{\partial_a \widetilde{R}}\partial_a\gamma \1_{\{\widetilde{R}(\theta,a)\leq r\}}\1_{\mathcal{B}_t}\mathrm{d}\theta\mathrm{d}a\mathrm{d}\ell+4\pi^2 \iiint \frac{3a a^0}{tm^2}\chi'\1_{\{\widetilde{R}(\theta,a)\leq r\}}\1_{\mathcal{B}_t}\gamma\mathrm{d}\theta\mathrm{d}a\mathrm{d}\ell\\
        &+4\pi^2\iiint   \left(\frac{(a^0)^3\chi'}{tm^2\partial_a \widetilde{R}} - \frac{\chi^1}{(\partial_a\widetilde{R})^2}\right) \partial_a^2\widetilde{R} \gamma \1_{\{\widetilde{R}(\theta,a)\leq r\}}\1_{\mathcal{B}_t}\mathrm{d}\theta\mathrm{d}a\mathrm{d}\ell\\
        =&:\rho^1_{I,1}(t,r)+\rho^1_{I,2}(t,r)+\rho^1_{I,3}(t,r).
    \end{align*}
    For the first term, by \eqref{equation_lower_bound_R_tilde_bulk}, we derive
    \begin{align*}
        r^{-k}\left|\rho^1_{I,1}(t,r)\right|&\lesssim \iiint t^{-1-k}(a^0)^3a^{-k}(a^0)^k|\partial_a\gamma| \mathrm{d}\theta\mathrm{d}a\mathrm{d}\ell\\
        &\lesssim t^{-1-k}\left(\|a^3\partial_a\gamma\|_{L^1}+\|a^{-k}\partial_a\gamma\|_{L^1}\right).
    \end{align*}
    For the second term, since $\widetilde{R}\gtrsim t\widehat{a}$ in $\mathcal{B}_t$, we obtain
    \begin{align*}
        r^{-k}|\rho^1_{I,2}(t,r)|\lesssim t^{-1}\iiint aa^0\widetilde{R}^{-k}\1_{\mathcal{B}_t}\mathrm{d}\theta \mathrm{d}a\mathrm{d}\ell\lesssim t^{-1-k}(\|a^{1-k}\gamma\|_{L^1}+\|a^{2}\gamma\|_{L^1}). 
    \end{align*}
    
    Finally, for the last term, the definition of $\chi^1$ yields
    \begin{equation*}
        \left|\frac{(a^0)^3\chi'}{tm^2\partial_a \widetilde{R}} - \frac{\chi^1}{(\partial_a\widetilde{R})^2}\right|\lesssim \frac{(a^0)^6}{t^2}.
    \end{equation*}
    Hence, using \eqref{equation_lower_bound_R} and the upper bound for $\partial_a^2\widetilde{R}$ from \eqref{equations_estimates_R_tilde_derivatives}, we find
    \begin{align}
        \label{equation_upperbound_chi_in_proof}
        \left|\frac{(a^0)^3\chi'}{tm^2\partial_a \widetilde{R}} - \frac{\chi^1}{(\partial_a\widetilde{R})^2}\right||\partial_a^2\widetilde{R}|\lesssim&~ \frac{(a^0)^6}{t^2}\left(t^2\frac{a}{\widetilde{R}(a^0)^4}+t\left[\frac{a}{(a^0)^5}+\frac{1}{(a^0)^3}\left(\frac{a^0}{\widetilde{R}a^3}+\frac{\sqrt{\ell}}{\widetilde{R}a^2}\right)\right]\right)\\
        \notag &+\frac{(a^0)^6}{t^2}\left(\frac{1}{a^2}+\frac{\sqrt{\ell}}{a^3}+\frac{1}{a^4}\right)\log\left\langle\frac{\widetilde{R}}{p}\right\rangle\\
        \notag \lesssim&~ \frac{a(a^0)^2}{\widetilde{R}}+\frac{1}{t}\left(aa^0+\frac{(a^0)^4}{\widetilde{R}a^3}+\frac{(a^0)^3\sqrt{\ell}}{\widetilde{R}a^2}\right)+\frac{(a^0)^6}{t^2}\left(\frac{1}{a^2}+\frac{\sqrt{\ell}}{a^3}+\frac{1}{a^4}\right)\log\left\langle\frac{\widetilde{R}}{p}\right\rangle\\
        \notag \lesssim&~ \frac{a(a^0)^2}{\widetilde{R}}+\frac{1}{t}\left(aa^0+\frac{(a^0)^3}{a}\right)+\frac{(a^0)^6}{t^2}\left(\frac{1}{a^2}+\frac{p}{a^2}\right)\log\left\langle\frac{\widetilde{R}}{p}\right\rangle,
    \end{align}
    where we used the lower bound \eqref{equation_lower_bound_R} to obtain the last inequality. Furthermore, in the bulk $\frac{1}{2}t\widehat a\leq\widetilde{R}\leq 2t\widehat a$. This gives a more precise estimate
    \begin{equation*}
        \left|\frac{(a^0)^3\chi'}{tm^2\partial_a \widetilde{R}} - \frac{\chi^1}{(\partial_a\widetilde{R})^2}\right||\partial_a^2\widetilde{R}|\lesssim \frac{(a^0)^5}{t}\left(\frac{1}{a}+a\right).
    \end{equation*}
    This implies 
    \begin{align*}
        r^{-k}\left|\rho^1_{I,3}(t,r)\right|&\lesssim t^{-1-k}\iiint\left(a^{-1}(a^0)^5+a(a^0)^5\right)(a^0)^ka^{-k}\gamma\mathrm{d}\theta\mathrm{d}a\mathrm{d}\ell\\
        &\lesssim t^{-1-k}\left(\|a^{-k-1}\gamma\|_{L^1}+\|a^6\gamma\|_{L^1}\right).
    \end{align*}
    
    Hence, since $k,\sigma$ satisfy \eqref{equation_condition_k_sigma}, we have
    \begin{equation*}
        |r^{-k}\rho^1_{I}(t,r)|\lesssim t^{-1-k}\left(\|(a^{-k-2}+a^{6})\gamma\|_{L^1}+\|(a^{-k}+a^3)\partial_a\gamma\|_{L^1}\right)\lesssim t^{-k-1}\|(a^{-4}+a^6)\gamma\|_{L^1}+\|(a^{-2}+a^3)\partial_a\gamma\|_{L^1}.
    \end{equation*}

    \textbf{\underline{Study of $\rho^1_O$}}\\

    Similarly, for $\rho^1_O$, we have
    \begin{align*}
        \rho^1_O(t,r)=&~4\pi^2\iiint \frac{\chi^1}{\partial_a \widetilde{R}}\partial_a\gamma \1_{\{\widetilde{R}(\theta,a)\leq r\}}\1_{\mathcal{B}_t^c}\mathrm{d}\theta\mathrm{d}a\mathrm{d}\ell+4\pi^2 \iiint \frac{3 aa^0\chi'}{tm^2}\1_{\{\widetilde{R}(\theta,a)\leq r\}}\1_{\mathcal{B}^c_t}\gamma\mathrm{d}\theta\mathrm{d}a\mathrm{d}\ell\\
        &+4\pi^2\iiint   \left(\frac{(a^0)^3\chi'}{tm^2\partial_a \widetilde{R}} - \frac{\chi^1}{(\partial_a\widetilde{R})^2}\right) \partial_a^2\widetilde{R} \gamma \1_{\{\widetilde{R}(\theta,a)\leq r\}}\1_{\mathcal{B}_t^c}\mathrm{d}\theta\mathrm{d}a\mathrm{d}\ell\\
        =&:\rho^1_{O,1}(t,r)+\rho^1_{O,2}(t,r)+\rho^1_{O,3}(t,r).
    \end{align*}
    By the definition of $\chi^1$, the integrands are supported where $\partial_a\widetilde{R}\gtrsim t(a^0)^{-3}$. Hence, using \eqref{equation_R_tilde_in_complement_bulk_proof}, we have
    \begin{align*}
        r^{-k}\left|\rho^1_{O,1}(t,r)\right|\lesssim &~ t^{-1}\iiint (a^0)^3\1_{\mathcal{B}_t^c}(\widetilde{R})^{-k}|\partial_a\gamma|\mathrm{d}\theta\mathrm{d}a\mathrm{d}\ell\\
        \lesssim&~ t^{-1-k-\sigma}\iiint \left(a^{-2k-4\sigma}+a^{4k+3\sigma+3}+\ell^{2k+2\sigma}(a^k+a^{k+3})+|\theta|^{k+\sigma}(a^{k-\sigma}+a^{k+3})\right)|\partial_a\gamma|\mathrm{d}\theta\mathrm{d}a\mathrm{d}\ell\\
        \lesssim &~t^{-1-k-\sigma}\left(\|(a^{-2k-4\sigma}+a^{4k+3\sigma+3})\partial_a\gamma\|_{L^1}+\|\ell^{2k+2\sigma}(a^k+a^{k+3})\partial_a\gamma\|_{L^1}\right)\\
        & + t^{-1-k-\sigma}\|(a^{k+3}+a^{k-\sigma})|\theta|^{k+\sigma}\partial_a\gamma\|_{L^1}.
    \end{align*}
    Consequently, using the condition \eqref{equation_condition_k_sigma} on $k,\sigma$ and Young's inequality,

    \begin{align*}
        r^{-k}\left|\rho^1_{O,1}(t,r)\right|\lesssim t^{-1-k-\sigma}\|(a^{-7}+a^{14}+\ell^6 (a^2+a^5)+|\theta|^3(a^{-1}+ a^5))\partial_a\gamma\|_{L^1}\lesssim t^{-1-k-\sigma}\mathcal{M}.
    \end{align*}

    Similarly,
    \begin{align*}
        r^{-k}|\rho^1_{O,2}(t,r)|\lesssim &~t^{-1}\iiint aa^0\1_{\mathcal{B}_t^c}(\widetilde{R})^{-k}\gamma\mathrm{d}\theta\mathrm{d}a\mathrm{d}\ell\\
        \lesssim &~t^{-1-k-\sigma}\left(\|(a^{-2k-4\sigma+1}+a^{4k+3\sigma+2})\gamma\|_{L^1}+\|(\ell^{2k+2\sigma}(a^{k+1}+a^{k+2})+(a^{k+2}+a^{k-\sigma+1})|\theta|^{k+\sigma})\gamma\|_{L^1}\right)\\
    \lesssim &~ t^{-1-k-\sigma}\|(a^{-7}+a^{13}+\ell^6 (a^3+a^4)+|\theta|^3(a^2+ a^4))\gamma\|_{L^1}\\
        \lesssim &~ t^{-1-k-\sigma}\mathcal{M}.
    \end{align*}
    For $\rho^1_{O,3}$, we begin by recalling that by the definition of $\widetilde{R}$ and \eqref{equation_estimates_G_a_H_a}, $\widetilde{R}\lesssim p +|\theta|+t\widehat a+\frac{a^0}{a^2}$. In particular, the following holds
    
    \begin{equation*}
        \left(\frac{1}{a^2}+\frac{p}{a^2}\right)\log\left\langle \frac{\widetilde{R}}{p} \right\rangle\lesssim \widetilde{R}+\frac{\widetilde{R}}{a^2}\lesssim t (1+a^{-2})+(a^{-1}+a^{-4})+ (1+a^{-2})|\theta|+(a^{-1}+a^{-3})\sqrt{\ell}. 
    \end{equation*}
    
    Applying \eqref{equation_lower_bound_R} and the above estimate to \eqref{equation_upperbound_chi_in_proof}, we find that
    \begin{align*}
        \left|\frac{(a^0)^3\chi'}{tm^2\partial_a \widetilde{R}} - \frac{\chi^1}{(\partial_a\widetilde{R})^2}\right||\partial_a^2\widetilde{R}|\lesssim &~ \frac{a(a^0)^2}{\widetilde{R}}+\frac{1}{t}\left(aa^0+\frac{(a^0)^3}{a}\right)+\frac{(a^0)^6}{t^2}\left(\frac{1}{a^2}+\frac{p}{a^2}\right)\log\left\langle\frac{\widetilde{R}}{p}\right\rangle\\
        \lesssim &~a^3 a^0+\frac{1}{t}\left(aa^0+\frac{(a^0)^3}{a}\right)\\
        &+\frac{(a^0)^6}{t^2}\left(t (1+a^{-2})+(a^{-1}+a^{-4})+ (1+a^{-2})|\theta|+(a^{-1}+a^{-3})\sqrt{\ell}\right)\\
        \lesssim &~a^3+a^4 +\frac{1}{t}(a^{-2}+a^6)+\frac{1}{t^2}\left(a^{-4}+a^5+|\theta|(a^{-2}+a^6)+\sqrt{\ell}(a^{-3}+a^5)\right).
    \end{align*}
    Then, by replacing $\sigma$ with $\sigma +1$ in \eqref{equation_R_tilde_in_complement_bulk_proof}, we derive
    \begin{equation*}
        (a^3+a^4)\widetilde{R}^{-k}\1_{\mathcal{B}_t^c}\lesssim \left(a^{-2k-4\sigma-1}+a^{4k+3\sigma+7}+\ell^{2k+2\sigma+2}(a^{k+3}+a^{k+4})+|\theta|^{k+\sigma+1}(a^{k-\sigma+2}+a^{k+4})\right).
    \end{equation*}
    Then, 
    \begin{equation*}
        \frac{1}{t}(a^{-2}+a^6)\widetilde{R}^{-k}\1_{\mathcal{B}_t^c}\lesssim t^{-k-\sigma-1}\left(a^{-2k-4\sigma-2}+a^{4k+3\sigma+6}+\ell^{2k+2\sigma}(a^{k-2}+a^{k+6})+|\theta|^{k+\sigma}(a^{k-\sigma-2}+a^{k+6)}\right).
    \end{equation*}
    Similarly, by replacing $\sigma$ with $\sigma-1$ in \eqref{equation_R_tilde_in_complement_bulk_proof},
    \begin{align*}
        t^{-2}\left(a^{-4}+a^5\right)\widetilde{R}^{-k}\1_{\mathcal{B}_t^c}\lesssim&~ t^{-k-\sigma-1}\left(a^{-2k-4\sigma}+a^{4k+3\sigma+2}+\ell^{2k+2\sigma-2}(a^{k-4}+a^{k+5})+|\theta|^{k+\sigma-1}(a^{k-\sigma-3}+a^{k+5})\right),\\[10pt]
        t^{-2}|\theta|(a^{-2}+a^6)\widetilde{R}^{-k}\1_{\mathcal{B}_t^c}\lesssim &~ t^{-k-\sigma-1}\big(|\theta|(a^{-2k-4\sigma+2}+a^{4k+3\sigma+3})+|\theta|^{k+\sigma}(a^{k-\sigma-1}+a^{k+6})\big)\\
        &+t^{-k-\sigma-1}\big(|\theta|\ell^{2k+2\sigma-2}(a^{k-2}+a^{k+6})\big),\\[10pt]
        t^{-2}\sqrt{\ell}(a^{-3}+a^5)\widetilde{R}^{-k}\1_{\mathcal{B}_t^c}\lesssim&~ t^{-k-\sigma-1}\big(\sqrt{\ell}(a^{-2k-4\sigma+1}+a^{4k+3\sigma+2})+\sqrt{\ell}|\theta|^{k+\sigma-1}(a^{k-\sigma-2}+a^{k+5})\big)\\
        & +t^{-k-\sigma-1}\big(\ell^{2k+2\sigma-\frac{3}{2}}(a^{k-3}+a^{k+5})\big).
    \end{align*}
    The last five inequalities yield
    \begin{align*}
        \left|\frac{(a^0)^3\chi'}{tm^2\partial_a \widetilde{R}} - \frac{\chi^1}{(\partial_a\widetilde{R})^2}\right|\frac{|\partial_a^2\widetilde{R}|}{\widetilde{R}^{k}}\1_{\mathcal{B}_t^c}t^{k+\sigma+1}\lesssim &~a^{-2k-4\sigma -2}+ a^{4k+3\sigma+7} + |\theta|(a^{-2k-4\sigma+2}+a^{4k+3\sigma+3})\\
        &+\sqrt{\ell}(a^{-2k-4\sigma +1}+a^{4k+3\sigma+2})+(|\theta|^{k+\sigma +1}+|\theta|^{k+\sigma-1})(a^{k-\sigma-3}+a^{k+6})\\
        & +\sqrt{\ell}|\theta|^{k+\sigma-1}(a^{k-\sigma-2}+a^{k+5})+ (\ell^{2k+2\sigma+2}+\ell^{2k+2\sigma-2})(a^{k-4}+a^{k+6})\\
        & +|\theta|\ell^{2k+2\sigma-2}(a^{k-2}+a^{k+6})
    \end{align*}
    Finally, we derive
    \begin{align*}
        |\rho^1_{O,3}(t,r)|r^{-k}t^{k+\sigma+1}\lesssim&~\|(a^{-2k-4\sigma -2}+ a^{4k+3\sigma+7})\gamma\|_{L^1}+\||\theta|(a^{-2k-4\sigma+2}+a^{4k+3\sigma+3})\gamma\|_{L^1}\\
        &+\|\sqrt{\ell}(a^{-2k-4\sigma +1}+a^{4k+3\sigma+2})\gamma\|_{L^1}+\|(|\theta|^{k+\sigma +1}+|\theta|^{k+\sigma-1})(a^{k-\sigma-3}+a^{k+6})\gamma\|_{L^1}\\
        &+\|\sqrt{\ell}|\theta|^{k+\sigma-1}(a^{k-\sigma-2}+a^{k+5})\gamma\|_{L^1}+ \|(\ell^{2k+2\sigma+2}+\ell^{2k+2\sigma-2})(a^{k-4}+a^{k+6})\gamma\|_{L^1} \\
        &+\||\theta|\ell^{2k+2\sigma-2}(a^{k-2}+a^{k+6})\gamma\|_{L^1}
    \end{align*}
    Consequently, since \eqref{equation_condition_k_sigma} holds, Young's inequality yields
    \begin{align*}
        |\rho^1_{O,3}(t,r)|r^{-k}t^{k+\sigma+1}\lesssim&~\|(a^{-10}+ a^{18})\gamma\|_{L^1}+\||\theta|(a^{-6}+a^{14})\gamma\|_{L^1}+\|\sqrt{\ell}(a^{-7}+a^{13})\gamma\|_{L^1}+\||\theta|^{4}(a^{-4}+a^{8})\gamma\|_{L^1}\\
        &+\|\sqrt{\ell}|\theta|^{2}(a^{-3}+a^{7})\gamma\|_{L^1}+ \|\ell^{8}(a^{-4}+a^{8})\gamma\|_{L^1} +\||\theta|\ell^{4}(a^{-2}+a^{6})\gamma\|_{L^1}\\
        \lesssim &~ t^{-k-\sigma-1}\mathcal{M}.
    \end{align*}

    \textbf{\underline{Study of $\rho^2_O$}}\\

    First, we split $\rho^2_{O}$ into three terms using Lemma \ref{lemma_splitting_different_R_j}. Let
    \begin{equation*}
        \rho^2_{O,j}(t,r)=4\pi^2\iiint \delta(\widetilde{R}(\theta,a)-r)\chi^2(t,\theta,a,\ell)\gamma(t,\theta,a,\ell)\1_{\mathcal{B}_t^c}\1_{\mathcal{R}_j}\mathrm{d}\theta\mathrm{d}a\mathrm{d}\ell,\qquad j=0,1,2.
    \end{equation*}
    For $\rho^2_{O,0}$, by Lemma \ref{lemma_splitting_different_R_j} and \eqref{equation_R_tilde_in_complement_bulk_proof}, we have
    \begin{align*}
        |r^{-k}\rho^2_{O,0}(t,r)|&=\left|\iint \chi^2\1_{\mathcal{R}_0}\1_{\mathcal{B}_t^c}\gamma(t,\theta,\aleph,\ell)\widetilde{R}^{-k}(\theta,\aleph)\frac{\mathrm{d}\theta\mathrm{d}\ell}{\partial_a\widetilde{R}(\theta,\aleph)}\right|\\
        &\lesssim t^{-k-\sigma-1}\iint \left(\aleph^{-2k-4\sigma-1}+\aleph^{4k+3\sigma+6}+\ell^{2k+2\sigma+2}\aleph^{k+3}+|\theta|^{k+\sigma+1}(\aleph^{k-\sigma+3}+\aleph^{k+3})\right)\gamma(\aleph)\mathrm{d}\theta\mathrm{d}\ell.
    \end{align*}
    Then, for $q>0$,\,$\aleph^q\gamma(t,\theta,\aleph,\ell)=\int^\aleph_0\partial_a(a^q\gamma(t,\theta,a,\ell))\mathrm{d}a$. Similarly, for $q\leq 0$, we also have $\aleph^q\gamma(t,\theta,\aleph,\ell)=-\int_\aleph^{+\infty}\partial_a(a^q\gamma(t,\theta,a,\ell))\mathrm{d}a$. Hence, for all $q\in\R$,
    \begin{equation*}
        |\aleph^q\gamma(t,\theta,\aleph,\ell)|\lesssim \int_{\R^*_+}|\partial_a(a^q\gamma(t,\theta,a,\ell))|\mathrm{d}a.
    \end{equation*}
    This implies
    \begin{align*}
        |r^{-k}\rho^2_{O,0}(t,r)|\lesssim&~ t^{-k-\sigma-1}\iiint \left(a^{-2k-4\sigma-1}+a^{4k+3\sigma+6}+\ell^{2k+2\sigma+2}a^{k+3}+|\theta|^{k+\sigma+1}(a^{k-\sigma+3}+a^{k+3})\right)|\partial_a\gamma|\mathrm{d}\theta\mathrm{d}a\mathrm{d}\ell\\
        &+t^{-k-\sigma-1}\iiint \left(a^{-2k-4\sigma-2}+a^{4k+3\sigma+5}+\ell^{2k+2\sigma+2}a^{k+2}+|\theta|^{k+\sigma+1}(a^{k-\sigma+2}+a^{k+2})\right)\gamma\mathrm{d}\theta\mathrm{d}a\mathrm{d}\ell\\
        \lesssim&~t^{-k-\sigma-1}\left(\|a^{-2k-4\sigma-2}\gamma\|_{L^1}+\|a^{4k+3\sigma+5}\gamma\|_{L^1}+\||\theta|^{k+\sigma+1}(a^{k-\sigma+3}+a^{k+3})\gamma\|_{L^1}\right)\\
        &+t^{-k-\sigma-1}\left(\|\ell^{2k+2\sigma+2}a^{k+2}\gamma\|_{L^1}\right)+t^{-k-\sigma-1}\left(\|a^{-2k-4\sigma-1}\partial_a\gamma\|_{L^1}+\|a^{4k+3\sigma+6}\partial_a\gamma\|_{L^1}\right)\\
        &+t^{-k-\sigma-1}\left(\|\ell^{2k+2\sigma+2}a^{k+3}\partial_a\gamma\|_{L^1}+\||\theta|^{k+\sigma+1}(a^{k-\sigma+3}+a^{k+3})\partial_a\gamma\|_{L^1}\right).
    \end{align*}
    We then use \eqref{equation_condition_k_sigma} and Young's inequality to find
    \begin{align*}
        |r^{-k}\rho^2_{O,0}(t,r)|\lesssim&~t^{-k-\sigma-1}\left(\|a^{-10}\gamma\|_{L^1}+\|a^{16}\gamma\|_{L^1}+\||\theta|^{4}a^5\gamma\|_{L^1}+\|\ell^{8}a^{4}\gamma\|_{L^1}\right)\\
        &+t^{-k-\sigma-1}\left(\|a^{-9}\partial_a\gamma\|_{L^1}+\|a^{17}\partial_a\gamma\|_{L^1}+\|\ell^{8}a^{5}\partial_a\gamma\|_{L^1}+\||\theta|^{4}a^5\partial_a\gamma\|_{L^1}\right)\\
        \lesssim &~t^{-k-\sigma-1}\mathcal{M}.
    \end{align*}
    Similarly, for $j=1,2$, by Lemma \ref{lemma_splitting_different_R_j} and replacing $\sigma$ in \eqref{equation_R_tilde_in_complement_bulk_proof} with $\sigma+4$ and $\sigma +3$ respectively, we have
    \begin{align*}
        |r^{-k}\rho^2_{O,j}(t,r)|=&\left|\iint \chi^2\1_{\mathcal{R}_j}\1_{\mathcal{B}_t^c}\gamma(t,\theta_j,a,\ell)\widetilde{R}^{-k}(\theta_j,a)\frac{\mathrm{d}a\mathrm{d}\ell}{\partial_\theta\widetilde{R}(\theta_j,a)}\right|\\
        \lesssim &~t^3\iint \frac{a^2}{a^0}\1_{\mathcal{R}_j}\1_{\mathcal{B}_t^c}\gamma(\theta_j)\widetilde{R}^{-k}\mathrm{d}a\mathrm{d}\ell +t^2 \iint\left(a|\theta_j|+\sqrt{\ell}(1+a)+a^{-1}+a\right)\1_{\mathcal{R}_j}\1_{\mathcal{B}_t^c}\gamma(\theta_j)\widetilde{R}^{-k}\mathrm{d}a\mathrm{d}\ell\\
        \lesssim  &~ t^{-k-\sigma-1}\iint \left(a^{-2k-4\sigma-15}+a^{4k+3\sigma+13}+\ell^{2k+2\sigma+8}a^{k+1}+|\theta_j|^{k+\sigma+4}(a^{k-\sigma-3}+a^{k+1})\right)\gamma(\theta_j)\mathrm{d}a\mathrm{d}\ell\\
        &+t^{-k-\sigma-1}\iint \left(a^{-2k-4\sigma-13}+a^{4k+3\sigma+10}+\ell^{2k+2\sigma+\frac{13}{2}}(a^{k-1} +a^{k+1})+\ell^{2k+2\sigma+6}|\theta_j|a^{k+1}\right)\gamma(\theta_j)\mathrm{d}a\mathrm{d}\ell\\
        &+t^{-k-\sigma-1}\iint \left(|\theta_j|^{k+\sigma+4}(a^{k-\sigma-4}+a^{k+1})+\sqrt{\ell}|\theta_j|^{k+\sigma+3}(a^{k-\sigma-3}+a^{k+1})\right)\gamma(\theta_j)\mathrm{d}a\mathrm{d}\ell\\
        &+t^{-k-\sigma-1}\iint \left(\sqrt{\ell}(a^{-2k-4\sigma-12}+a^{4k+3\sigma+10})+|\theta_j|(a^{-2k-4\sigma-11}+a^{4k+3\sigma+10})\right)\gamma(\theta_j)\mathrm{d}a\mathrm{d}\ell.
    \end{align*}
    Using Young's inequality to eliminate some terms, we find
    \begin{align*}
        |r^{-k}\rho^2_{O,j}(t,r)|\lesssim&~t^{-k-\sigma-1}\iint \left(a^{-2k-4\sigma-15}+a^{4k+3\sigma+14}+\ell^{2k+2\sigma+8}(a^{k-\sigma-3}+1+a^{k+1})\right)\gamma(\theta_j)\mathrm{d}a\mathrm{d}\ell\\
        &+t^{-k-\sigma-1}\iint |\theta_j|^{k+\sigma+4}(a^{k-\sigma-4}+1+a^{k+1})\gamma(\theta_j)\mathrm{d}a\mathrm{d}\ell.
    \end{align*}
    By similar arguments as in the estimate of $\rho_{O,0}^2$, for all $n\geq 0$, we have
    \begin{equation*}
        |\theta_j|\gamma(t,\theta_j,a,\ell)\lesssim \int_{\R} |\partial_\theta(|\theta|^n\gamma(t,\theta,a,\ell))|\mathrm{d}\theta.
    \end{equation*}
    We then derive
    \begin{align*}
        |r^{-k}\rho^2_{O,j}(t,r)|\lesssim&~t^{-k-\sigma-1}\iiint \left(a^{-2k-4\sigma-15}+a^{4k+3\sigma+14}+\ell^{2k+2\sigma+8}(a^{k-\sigma-3}+1+a^{k+1})\right)|\partial_\theta\gamma|\mathrm{d}a\mathrm{d}\ell\\
        &+t^{-k-\sigma-1}\iiint \left(|\theta|^{k+\sigma+4}(a^{k-\sigma-4}+1+a^{k+1})|\partial_\theta\gamma|+|\theta|^{k+\sigma+3}(a^{k-\sigma-4}+1+a^{k+1})\gamma\right)\mathrm{d}a\mathrm{d}\ell\\
        \lesssim &~t^{-k-\sigma-1}\left(\|a^{-2k-4\sigma-15}\partial_\theta\gamma\|_{L^1}+\|a^{4k+3\sigma+14} \partial_\theta\gamma\|_{L^1}+\|\ell^{2k+2\sigma+8}(a^{k-\sigma-4}+1+a^{k+1})\partial_\theta\gamma\|_{L^1}\right)\\
        &+t^{-k-\sigma-1}\left(\||\theta|^{k+\sigma+4}(a^{k-\sigma-4}+1+a^{k+1})\partial_\theta\gamma\|_{L^1}+\||\theta|^{k+\sigma+3}(a^{k-\sigma-4}+1+a^{k+1})\gamma\|_{L^1}\right).
    \end{align*}
    By similar arguments as before, this implies
    \begin{align*}
        |r^{-k}\rho^2_{O,j}(t,r)|\lesssim &~t^{-k-\sigma-1}\left(\|\big(a^{-23}+a^{25}+\ell^{14}(a^{-5}+a^{6})+|\theta|^{7}(a^{-5}+a^{3})\big)\partial_\theta\gamma\|_{L^1}+\||\theta|^{6}(a^{-5}+a^{3})\gamma\|_{L^1}\right)\\
        &+t^{-k-\sigma-1}\||\theta|^{k+\sigma+4}(a^{k-\sigma-4}+1+a^{k+1})\partial_\theta\gamma\|_{L^1}.
    \end{align*}
    To avoid having higher order moments, we study the last term of the right-hand more carefully. First, if $k-\sigma\geq 0$ then
    \begin{equation*}
        \||\theta|^{k+\sigma+4}(a^{k-\sigma-4}+1+a^{k+1})\partial_\theta\gamma\|_{L^1}\leq \||\theta|^{\frac{13}{2}}(a^{-4}+a^{3})\partial_\theta\gamma\|_{L^1}.
    \end{equation*}
    Conversely, if $k-\sigma\leq0$, then $k+\sigma\leq 2$ and we find
    \begin{equation*}
        \||\theta|^{k+\sigma+4}(a^{k-\sigma-4}+1+a^{k+1})\partial_\theta\gamma\|_{L^1}\leq \||\theta|^{6}(a^{-5}+a^{3})\partial_\theta\gamma\|_{L^1}.
    \end{equation*}
    In both cases, we can apply Young's inequality to finally derive
    \begin{equation*}
        |r^{-k}\rho^2_{O,j}(t,r)|\lesssim t^{-k-\sigma-1}\mathcal{M}.
    \end{equation*}
    We obtain the final estimate by writing $\rho_I:=\rho^1_I$, $\rho_O:=\rho^1_O+\rho^2_O$ and combining the previous estimates.
\end{proof}

\begin{remark}
    Lemmas \ref{lemma_estimate_total_mass}--\ref{lemma_estimate_charge_density} suggest that the estimates for the nonlinear solution to \eqref{equation_non_linear_aa_idea}  are consistent with the usual estimates obtained for the Vlasov-Maxwell system. Indeed, if $(f,E,B)$ is a small data solution to the Vlasov-Maxwell system, one finds that, inside the light cone,  the fields and the charge density $\int_{\R^3_v} f\mathrm{d}v$ satisfy
    \begin{equation*}
        |(E,B)|(t,x)\lesssim t^{-2},\qquad \int_{\R^3_v}f(t,x,v)\mathrm{d}v\lesssim t^{-3}.
    \end{equation*}
    This is comparable to
    \begin{equation*}
        |E(t,r)|=\frac{|M(t,r)|}{r^2}\lesssim t^{-2},\qquad \frac{|\rho(t,r)|}{r^2}\lesssim t^{-3},
    \end{equation*}
    where we omitted the moments of $\gamma$.
\end{remark}

\subsubsection{Estimates on the potential}

\begin{proposition}
    \label{proposition_estimates_psi_first_derivative}
    Let $t\geq T_0$. The first derivatives of $\widetilde{\Psi}$ satisfy
    \begin{equation}
        \label{equation_estimate_Psi_first_derivative}
        \begin{array}{r@{}l}
             \displaystyle t^\frac{3}{2}\|a^{-1}\partial_\theta\widetilde{\Psi}\|_{L^\infty(\R\times\R^*_+)} \lesssim& ~\|a^{-\frac{3}{2}}\gamma\|_{L^1}+\|\gamma\|_{L^1} +t^{-\frac{1}{4}}\left(\|a^{-4}\gamma\|_{L^1}+\|a^{\frac{27}{4}}\gamma\|_{L^1}+\|\ell^\frac{9}{2}\gamma\|_{L^1}+\||\theta|^\frac{9}{4}\gamma\|_{L^1}\right),\\[6pt]
             \displaystyle t|\partial_a\widetilde{\Psi}(\theta,a)|\lesssim& ~\|a^{-2}\gamma\|_{L^1}+\|\gamma\|_{L^1}(a^{-3}+a^3+\ell^2+|\theta|)\\[6pt]
             \displaystyle&~+~t^{-\frac{1}{4}}\left(\|a^{-5}\gamma\|_{L^1}+\|a^9\gamma\|_{L^1}+\||\theta|^3\gamma\|_{L^1}+\|\ell^6\gamma\|_{L^1}\right).
        \end{array}
    \end{equation}
\end{proposition}
\begin{proof}
    First, since $\widetilde{R}a^2\gtrsim 1$ and $|\partial_\theta\widetilde{R}|\leq 1$, 
    \begin{align*}
        |a^{-1}\partial_\theta\widetilde{\Psi}|=\frac{1}{(\widetilde{R}a^2)^\frac{1}{2}}\frac{M(t,\widetilde{R})}{\widetilde{R}^\frac{3}{2}}|\partial_\theta\widetilde{R}|\lesssim\frac{M(t,\widetilde{R})}{\widetilde{R}^\frac{3}{2}}.
    \end{align*}
    Using Lemma \ref{lemma_estimate_total_mass} with $k=\frac{3}{2},\sigma=\frac{1}{4}$ and Young's inequality, we find
    \begin{align*}
        t^\frac{3}{2}|a^{-1}\partial_\theta\widetilde{\Psi}|&\lesssim \|a^{-\frac{3}{2}}\gamma\|_{L^1}+\|\gamma\|_{L^1} +t^{-\frac{1}{4}}\left(\|a^{-4}\gamma\|_{L^1}+\|a^{\frac{27}{4}}\gamma\|_{L^1}+\|a^\frac{3}{2}\ell^\frac{7}{2}\gamma\|_{L^1}+\||\theta|^\frac{7}{4}(a^\frac{3}{2}+a^\frac{5}{4})\gamma\|_{L^1}\right) \\
        &\lesssim\|a^{-\frac{3}{2}}\gamma\|_{L^1}+\|\gamma\|_{L^1} +t^{-\frac{1}{4}}\left(\|a^{-4}\gamma\|_{L^1}+\|a^{\frac{27}{4}}\gamma\|_{L^1}+\|\ell^\frac{9}{2}\gamma\|_{L^1}+\||\theta|^\frac{9}{4}\gamma\|_{L^1}\right).
    \end{align*}
    Then, using \eqref{equations_estimates_R_tilde_derivatives}, we have
    \begin{align*}
        |\partial_a\widetilde{\Psi}|=|\partial_a\widetilde{R}|\frac{M(t,\widetilde{R})}{\widetilde{R}^2}\lesssim \frac{M(t,\widetilde{R})}{\widetilde{R}^2}\left(t+\left(\frac{1}{a}+\frac{p}{a}\right)\log\left\langle\frac{\widetilde{R}}{p}\right\rangle\right)\lesssim t\frac{M(t,\widetilde{R})}{\widetilde{R}^2}+\frac{M(t,\widetilde{R})}{\widetilde{R}^\frac{3}{2}}+\frac{M(t,\widetilde{R})}{\widetilde{R}a}\frac{p}{\widetilde{R}}\log\left\langle\frac{\widetilde{R}}{p}\right\rangle.
    \end{align*}
    Using Lemma \ref{lemma_estimate_total_mass} with $k=2,\sigma=\frac{1}{4}$, we obtain
    \begin{align*}
        t\frac{M(t,\widetilde{R})}{\widetilde{R}^2}&\lesssim t^{-1}\left(\|a^{-2}\gamma\|_{L^1}+\|\gamma\|_{L^1}\right)+t^{-1-\frac{1}{4}}\left(\|a^{-5}\gamma\|_{L^1}+\|a^\frac{35}{4}\gamma\|_{L^1}+\|a^{2}\ell^{\frac{9}{2}}\gamma\|_{L^1}+\|a^{\frac{7}{4}}|\theta|^{\frac{9}{4}}\gamma\|_{L^1}+\|a^2|\theta|^{\frac{9}{4}}\gamma\|_{L^1}\right)\\
        & \lesssim t^{-1}\left(\|a^{-2}\gamma\|_{L^1}+\|\gamma\|_{L^1}\right)+t^{-1-\frac{1}{4}}\left(\|a^{-5}\gamma\|_{L^1}+\|a^\frac{35}{4}\gamma\|_{L^1}+\|\ell^{\frac{35}{6}}\gamma\|_{L^1}+\||\theta|^{\frac{35}{12}}\gamma\|_{L^1}\right).
    \end{align*}
    Since $\frac{1}{2}t\widehat a\leq\widetilde{R}\leq2t\widehat a$ in the bulk $\mathcal{B}_t$, we also derive
    \begin{align*}
        \frac{M(t,\widetilde{R})}{\widetilde{R}a}\frac{p}{\widetilde{R}}\log\left\langle\frac{\widetilde{R}}{p}\right\rangle\lesssim& ~\frac{M(t,\widetilde{R})}{\widetilde{R}a}\1_{\mathcal{B}_t}+\frac{M(t,\widetilde{R})}{\widetilde{R}a}\1_{\mathcal{B}_t^c}\\
        \lesssim&~t^{-1}M(t,\widetilde{R})(a^{-1}+a^{-2})+\frac{M(t,\widetilde{R})}{\widetilde{R}a}\left(\1_{\{t\leq (a^{-1}a^0)^{4}\}}+\1_{\{t\leq a^{3}\}}+\1_{\{t\leq\ell^2\}}+\1_{\{t\leq 2|\theta|a^{-1}a^0\}}\right)\\
        \lesssim&~t^{-1}M(t,\widetilde{R})(a^{-1}+a^{-2})+t^{-1}\frac{M(t,\widetilde{R})}{\widetilde{R}a}(a^{-4}+a^3+\ell^2+|\theta|(1+a^{-1})).
    \end{align*}
    Since $M(t,r)\leq 4\pi^2\|\gamma\|_{L^1}$ and $\widetilde{R}a^2\geq m,\, \widetilde{R}a\geq 1$, this implies
    \begin{equation*}
        \frac{M(t,\widetilde{R})}{\widetilde{R}a}\frac{p}{\widetilde{R}}\log\left\langle\frac{\widetilde{R}}{p}\right\rangle\lesssim  t^{-1}\|\gamma\|_{L^1} (a^{-3}+a^3+\ell^2+|\theta|).
    \end{equation*}
    Finally, by Lemma \ref{lemma_estimate_total_mass} with $k=\frac{3}{2},\sigma=0$, 
    \begin{equation*}
        \frac{M(t,\widetilde{R})}{\widetilde{R}^\frac{3}{2}}\lesssim t^{-\frac{3}{2}}\left(\|a^{-3}\gamma\|_{L^1}+\|a^6\gamma\|_{L^1}+\|\ell^4\gamma\|_{L^1}+\||\theta|^2\gamma\|_{L^1}\right).
    \end{equation*}
\end{proof}
We can then prove similar estimates for the second order derivatives.

\begin{proposition}
\label{proposition_estimates_psi_second_derivatives}
    Let $t\geq T_0$ and $\mathcal{M}$ be defined as in \eqref{equation_definition_mathcal_M}. The second order derivatives of $\widetilde{\Psi}$ satisfy
    \begin{equation}
        t^{\frac{3}{2}}\|\partial_\theta\partial_a\widetilde{\Psi}\|_\infty+ t^2\|(1+a^{-2})\partial_\theta^2 \widetilde{\Psi}\|_\infty\lesssim \mathcal{M},
    \end{equation}
    \begin{equation}
        t\left\|\frac{a^2}{1+a^2}\partial_a^2\widetilde{\Psi}\right\|_\infty\lesssim \|(a^{-6}+a^6)\gamma\|_{L^1}+\|(a^{-3}+a^3)\gamma_a\|_{L^1}+t^{-\frac{1}{2}}\mathcal{M}.
    \end{equation}
\end{proposition}
\begin{proof}
     We begin with the study of $\partial^2_\theta\widetilde{\Psi}$. By \eqref{equations_estimates_R_tilde_derivatives} and \eqref{equation_def_derivatives_psi}, since $\widetilde{R}a^2\geq a^0\gtrsim 1$, we have
     \begin{align*}
        |\partial_\theta^2\widetilde{\Psi}(\theta,a)|&\lesssim \frac{M(t,\widetilde{R})}{\widetilde{R}^3}+ \frac{M(t,\widetilde{R})}{\widetilde{R}^2}\frac{1}{\widetilde{R}-\frac{1}{a^0}}+ \frac{\rho(t,\widetilde{R})}{\widetilde{R}^2},\\
        |a^{-2}\partial_\theta^2\widetilde{\Psi}(\theta,a)|&\lesssim \frac{M(t,\widetilde{R})}{\widetilde{R}^3 a^2}+\frac{M(t,\widetilde{R})a^0}{\widetilde{R}^3a^2}+\frac{\rho(t,\widetilde{R})}{\widetilde{R}^2a^2}\lesssim \frac{M(t,\widetilde{R})}{\widetilde{R}^2} +\frac{\rho(t,\widetilde{R})}{\widetilde{R}}. 
    \end{align*}
    Then, we apply Lemma \ref{lemma_estimate_total_mass} with $k=2,\sigma=0$ and $k=3,\sigma=0$. We find
    \begin{equation*}
        \frac{M(t,\widetilde{R})}{\widetilde{R}^2}\lesssim t^{-2}\|(a^{-4}+a^{8}+|\theta|^4+\ell^8)\gamma\|_{L^1} \qquad \frac{M(t,\widetilde{R})}{\widetilde{R}^3}\lesssim t^{-3}\|(a^{-6}+a^{12}+|\theta|^4+\ell^8)\gamma\|_{L^1}.
    \end{equation*}
    Using the second estimate of $M$ in Lemma \ref{lemma_estimate_total_mass} with $k=2,\sigma=0$, we have
    \begin{equation*}
        \frac{M(t,\widetilde{R})}{\widetilde{R}^2}\frac{1}{\widetilde{R}-\frac{1}{a^0}}\lesssim t^{-2}\|(a^{-4}+a^{10}+\ell^8+|\theta|^4)\gamma\|_{L^1}.
    \end{equation*}
    For the two terms that depend on $\rho$, we apply Lemma \ref{lemma_estimate_charge_density} with $\kappa=1,\sigma=0$ and $\kappa=2,\sigma=0$.  We obtain
    \begin{equation*}
         \frac{\rho(t,\widetilde{R})}{\widetilde{R}}\lesssim t^{-2}\mathcal{M},\qquad\frac{\rho(t,\widetilde{R})}{\widetilde{R}^2}\lesssim t^{-3}\mathcal{M}.
    \end{equation*}
    The five previous estimates grant the upper bound of $|(1+a^{-2})\partial^2_\theta\widetilde{\Psi}|$. We then prove the upper bound on $\partial_a^2\widetilde{\Psi}$. By \eqref{equation_lower_bound_R} and \eqref{equations_estimates_R_tilde_derivatives}, we have
    \begin{align*}
        \frac{(\partial_a\widetilde{R})^2}{\widetilde{R}^2}\lesssim \frac{t^2}{\widetilde{R}^2}+\frac{1}{\widetilde{R}^2a^2}\log\langle\widetilde{R}a^2\rangle+\frac{p^2}{\widetilde{R}^2a^2}\log\left\langle \frac{\widetilde{R}}{p}\right\rangle\lesssim \frac{t^2}{\widetilde{R}^2}+\frac{1}{\widetilde{R}}+\frac{1}{a^2},
    \end{align*}
    as well as
    \begin{align*}
        \frac{|\partial_a^2\widetilde{R}|}{\widetilde{R}}\lesssim \frac{t}{a\widetilde{R}}+\frac{t^2}{\widetilde{R}^2}+\frac{1}{a^2\widetilde{R}}\log\langle \widetilde{R}a^2\rangle +\frac{p}{a^2\widetilde{R}}\log\left\langle \frac{\widetilde{R}}{p}\right\rangle \lesssim \frac{t}{a\widetilde{R}}+\frac{t^2}{\widetilde{R}^2}+1+\frac{1}{a^2}.
    \end{align*}
    With \eqref{equation_def_derivatives_psi}, this implies 
    \begin{equation*}
        |\partial_a^2\widetilde{\Psi}|\lesssim (1+a^{-2})\left[\frac{M(t,\widetilde{R})}{\widetilde{R}}\left(1+\frac{t}{\widetilde{R}}+\frac{t^2}{\widetilde{R}^2}\right)+\rho(t,\widetilde{R})\left(1+\frac{t}{\widetilde{R}}+\frac{t^2}{\widetilde{R}^2}\right)\right].
    \end{equation*}
    For the terms containing $M$, we apply Lemma \ref{lemma_estimate_total_mass} with $k=1,\sigma=\frac{1}{2}$,\,$k=2,\sigma=\frac{1}{2}$, and $k=3,\sigma=\frac{1}{2}$ respectively. We derive
    \begin{equation}
        \label{equation_proof_estimate_M_R3}
        \frac{M(t,\widetilde{R})}{\widetilde{R}}+t\frac{M(t,\widetilde{R})}{\widetilde{R}^2}+t^2\frac{M(t,\widetilde{R})}{\widetilde{R}^3}\lesssim t^{-1}\|(a^{-3}+1)\gamma\|_{L^1}+t^{-\frac{3}{2}}\|(a^{-8}+a^{24}+|\theta|^4+\ell^8)\gamma\|_{L^1}.
    \end{equation}
    Then, applying Lemma \ref{lemma_estimate_charge_density} with $k=0,\sigma=1$, $k=1,\sigma=\frac{1}{2}$, and $k=2,\sigma=\frac{1}{2}$, we obtain
    \begin{align*}
        |\rho(t,\widetilde{R})|+t\frac{|\rho(t,\widetilde{R})|}{\widetilde{R}}+t^2\frac{|\rho(t,\widetilde{R})|}{\widetilde{R}}\lesssim&~t^{-1}\left(\|(a^{-4}+a^6)\gamma\|_{L^1}+\|(a^{-2}+a^{3})\gamma_a\|_{L^1}\right)+t^{-\frac{3}{2}}\mathcal{M}.
    \end{align*}
    Together with the previous estimate, this yields the upper bound of $|\partial_a^2\widetilde{\Psi}|$. Finally, for $\partial_\theta\partial_a\widetilde{\Psi}$, by \eqref{equation_lower_bound_R} and \eqref{equations_estimates_R_tilde_derivatives}, we have
    \begin{equation*}
        \frac{|\partial_a\widetilde{R}|}{\widetilde{R}^2}\lesssim \frac{t}{\widetilde{R}^2}+\frac{1}{\widetilde{R}^\frac{3}{2}}+\frac{1}{\widetilde{R}^\frac{1}{2}},\qquad  \frac{|\partial_a\partial_\theta\widetilde{R}|}{\widetilde{R}}\lesssim \frac{t}{\widetilde{R}^2}+\frac{1}{\widetilde{R}^\frac{1}{2}}.
    \end{equation*}
    Hence, using \eqref{equation_def_derivatives_psi}, we find
    \begin{equation*}
        |\partial_\theta\partial_a\widetilde{\Psi}|\lesssim \frac{M(t,\widetilde{R})}{\widetilde{R}^\frac{3}{2}}+\frac{M(t,\widetilde{R})}{\widetilde{R}^\frac{5}{2}}+t\frac{M(t,\widetilde{R})}{\widetilde{R}^3}+\frac{\rho(t,\widetilde{R})}{\widetilde{R}^\frac{1}{2}}+\frac{\rho(t,\widetilde{R})}{\widetilde{R}^\frac{3}{2}}+t\frac{\rho(t,\widetilde{R})}{\widetilde{R}^2}.
    \end{equation*}
    Therefore, by Lemma \ref{lemma_estimate_total_mass} with $k=\frac{3}{2},\sigma=0$, $k=\frac{5}{2},\sigma=0$ and $k=3,\sigma=0$, we have
    \begin{equation*}
        \frac{M(t,\widetilde{R})}{\widetilde{R}^\frac{3}{2}}+\frac{M(t,\widetilde{R})}{\widetilde{R}^\frac{5}{2}}+t\frac{M(t,\widetilde{R})}{\widetilde{R}^3}\lesssim t^{-\frac{3}{2}}\|(a^{-8}+a^{24}+|\theta|^4+\ell^8)\gamma\|_{L^1}.
    \end{equation*}
    Similarly, by Lemma \ref{lemma_estimate_charge_density} with $k=\frac{1}{2},\sigma=0$, $k=\frac{3}{2},\sigma=0$ and $k=2,\sigma=0$, we have
    \begin{equation*}
        \frac{\rho(t,\widetilde{R})}{\widetilde{R}^\frac{1}{2}}+\frac{\rho(t,\widetilde{R})}{\widetilde{R}^\frac{3}{2}}+t\frac{\rho(t,\widetilde{R})}{\widetilde{R}^2}\lesssim t^{-\frac{3}{2}}\mathcal{M}.
    \end{equation*}
    This concludes the proof of the proposition.
\end{proof}


\section{Study of the nonlinear problem}
\label{section_non_linear_problem}
We recall that if $f$ is a solution to \eqref{equation_VM_pc_radial_case}, $\gamma(t,\theta,a,\ell)=f(R(\theta+t\widehat a,a),U(\theta+t\widehat a,a))$ is a solution to 
\begin{equation}
    \label{equation_non_linear_gamma}
    \partial_t\gamma=\{\widetilde{\Psi},\gamma\}_{\theta,a},
\end{equation}
where $\{g,h\}_{\theta,a}:=\partial_\theta g\partial_a h-\partial_ag\partial_\theta h$.
Similarly, finding a solution to \eqref{equation_non_linear_gamma} grants a solution to \eqref{equation_VM_pc_radial_case} given by $f(t,r,u,\ell)= \gamma(t,\Theta(r,u)-t\widehat{\mathcal{A}}(r,u),\mathcal{A}(r,u),\ell)$. Hence, we can focus on studying the equation satisfied by $\gamma$. To prove global existence, we will propagate moments on $\gamma$ and its derivatives. To this end, we begin by noting that the moments of $\gamma$ satisfy the nonlinear equations given by 

\begin{equation}
    \label{equation_non_linear_a_q_gamma}
    \partial_t(a^q\gamma)=\{\widetilde{\Psi},a^q\gamma\}_{\theta,a}-qa^{q-1}\gamma\partial_\theta\widetilde{\Psi},
\end{equation}
\begin{equation}
    \label{equation_non_linear_theta_q_gamma}
    \partial_t (|\theta|^q\gamma)=\{\widetilde{\Psi},|\theta|^q\gamma\}_{\theta,a}+q\frac{\theta}{|\theta|}|\theta|^{q-1}\partial_a\widetilde{\Psi}.
\end{equation}

\subsection{Propagation of moments}

We begin by propagating moments on $\gamma$ a solution to \eqref{equation_non_linear_gamma}.

\begin{proposition}
    \label{proposition_bootstrap_gamma}
    Let $q\geq \frac{5}{2},~T>T_0$ and assume 
    \begin{equation}
        \label{equation_moments_assumption_boostrap_initial}
        \|(a^{-3q}+a^{3q}+|\theta|^q+\ell^{2q})\gamma(T_0)\|_{L^1}\leq \varepsilon_0,
    \end{equation}
    and, for all $t\in[T_0,T]$,
    \begin{equation}
        \label{equation_moments_assumption_boostrap_t}
        \|(a^{-3q}+a^{3q}+|\theta|^q+\ell^{2q})\gamma(t)\|_{L^1}\leq \varepsilon_1\langle t\rangle^\delta,
    \end{equation}
    with $\delta<\frac{1}{4}$. Then, the following stronger bounds hold. For all $t\in[T_0,T]$,
    \begin{equation}
        \label{equation_upper_bound_moments_gamma_final}
        \|(a^{-3q}+a^{3q}+\ell^{2q})\gamma(t)\|_{L^1}\lesssim \varepsilon_0,\qquad \| |\theta|^q \gamma\|_{L^1}\lesssim (\varepsilon_0+\varepsilon_1)t^{C\varepsilon_0},
    \end{equation}
    where $C$ is a constant that depends only on $m$ and $T_0$.
\end{proposition}
\begin{proof}
    Let $q\geq \frac{5}{2}$.  The moments in $\ell$ are conserved in time, in particular
    \begin{equation*}
        \|\ell^{2q}\gamma(t)\|_{L^1}=\|\ell^{2q}\gamma(T_0)\|_{L^1}\leq \varepsilon_0.
    \end{equation*}
    Then, by \eqref{equation_non_linear_a_q_gamma}, we have
    \begin{align*}
        \frac{\mathrm{d}}{\mathrm{d}t}\|a^{3q}\gamma\|_{L^1}\leq \left| \iiint \{\widetilde{\Psi},a^{3q}\gamma\}_{\theta,a}\mathrm{d}\theta\mathrm{d}a\mathrm{d}\ell\right|+3q\left| \iiint a^{3q-1}\gamma\partial_\theta \widetilde{\Psi}\mathrm{d}\theta\mathrm{d}a\mathrm{d}\ell\right|.
    \end{align*}
    Since the integral of the Poisson bracket vanishes, using \eqref{equation_estimate_Psi_first_derivative}, we are left with
    \begin{align*}
        \frac{\mathrm{d}}{\mathrm{d}t}\|a^{3q}\gamma\|_{L^1}&\leq 3q\iiint a^{3q-1}\gamma|\partial_\theta\widetilde{\Psi}|\mathrm{d}\theta\mathrm{d}a\mathrm{d}\ell\\
        &\lesssim t^{-\frac{3}{2}}\left[\|a^{-\frac{3}{2}}\gamma\|_{L^1}+\|\gamma\|_{L^1} +t^{-\frac{1}{4}}\left(\|a^{-4}\gamma\|_{L^1}+\|a^{\frac{27}{4}}\gamma\|_{L^1}+\|\ell^\frac{9}{2}\gamma\|_{L^1}+\||\theta|^\frac{9}{4}\gamma\|_{L^1}\right)\right]\|a^{3q}\gamma\|_{L^1}.
    \end{align*}
    Then, using the assumptions \eqref{equation_moments_assumption_boostrap_initial}--\eqref{equation_moments_assumption_boostrap_t}, we obtain
    \begin{equation*}
        \frac{\mathrm{d}}{\mathrm{d}t}\|a^{3q}\gamma\|_{L^1}\lesssim  t^{-\frac{3}{2}+\delta}\|a^{3q}\gamma\|_{L^1},
    \end{equation*}
    which, by Gronwall's lemma, leads to 
    \begin{equation*}
        \|a^{3q}\gamma\|_{L^1}\lesssim  \varepsilon_0.
    \end{equation*}
    The same holds for $\|a^{-3q}\gamma\|_{L^1}$.  Finally, by \eqref{equation_non_linear_theta_q_gamma} and the upper bound on $\partial_a\widetilde{\Psi}$ in \eqref{equation_estimate_Psi_first_derivative}, we obtain
    \begin{align*}
        \frac{\mathrm{d}}{\mathrm{d}t}\||\theta|^q\gamma\|_{L^1}\lesssim&~t^{-1}\|a^{-2}\gamma\|_{L^1}\iiint|\theta|^{q-1}\gamma\mathrm{d}\theta\mathrm{d}a\mathrm{d}\ell+t^{-1}\|\gamma\|_{L^1}\iiint(a^{-3}+a^3+\ell^2+|\theta|)|\theta|^{q-1}\gamma\mathrm{d}\theta\mathrm{d}a\mathrm{d}\ell\\
             \displaystyle&+t^{-\frac{5}{4}}\left(\|a^{-5}\gamma\|_{L^1}+\|a^9\gamma\|_{L^1}+\||\theta|^3\gamma\|_{L^1}+\|\ell^6\gamma\|_{L^1}\right)\iiint |\theta|^{q-1}\gamma\mathrm{d}\theta\mathrm{d}a\mathrm{d}\ell.
    \end{align*}
    Using Young's inequality, assumptions \eqref{equation_moments_assumption_boostrap_initial}--\eqref{equation_moments_assumption_boostrap_t}, as well as the first upper bound in \eqref{equation_upper_bound_moments_gamma_final}, we derive 
    \begin{equation*}
        \frac{\mathrm{d}}{\mathrm{d}t}\||\theta|^q\gamma\|_{L^1}\lesssim \varepsilon_0^2 t^{-1}+\varepsilon_0\varepsilon_1 t^{-\frac{5}{4}+\delta}+\left(\varepsilon_0t^{-1}+\varepsilon_1t^{-\frac{5}{4}+\delta}\right)\||\theta|^q\gamma\|_{L^1},
    \end{equation*}
    which, by Gronwall's Lemma, proves the last estimate in \eqref{equation_upper_bound_moments_gamma_final}.
\end{proof}

\subsection{The bootstrap argument for the derivatives}

To prove the propagation of moments, one could try to directly consider, for instance, $\partial_t\|a^q\gamma_a\|_{L^1}$. However, doing so yields an upper bound containing an integral of $a^q\gamma_\theta\partial_a^2\widetilde{\Psi}$. Since $t\partial^2_a\widetilde{\Psi}$ is bounded by $1+a^{-2}$ multiplied by moments of $(\gamma,\gamma_\theta,\gamma_a)$, it will induce higher-order moments, preventing us from closing the bootstrap argument. \\
To circumvent this issue, we can follow the approach of \cite{Pausader_Widmayer_2021}. The general idea is to consider specific weight functions $\omega$ and define $\omega^{(1)}:=P_1(a)\omega,\,\omega^{(2)}:=P_2(a)\omega$. Here, $P_1$ and $P_2$ are defined to bypass the difficulty coming from the estimate of $\partial_a^2\widetilde{\Psi}$. More precisely, we require $P_1,P_2$ to satisfy
\begin{equation*}
    (1+a^{-2})P_2(a)\leq P_1(a),\qquad (1+a^{-2})^{-1}P_1(a)\leq P_2(a).
\end{equation*}
We can thus consider
\begin{equation*}
    P_1(a):=a+a^{-1},\qquad P_2(a):=a,
\end{equation*}
and the weights 
\begin{equation*}
    \omega_{d,n,q}:=|\theta|^d\ell^na^q,\quad \omega^{(1)}_{d,n,q}:=(a+a^{-1})\omega_{d,n,q},\quad \omega^{(2)}_{d,n,q}:=a\omega_{d,n,q}.
\end{equation*}
with $d,n\in\N$ and $q\in\R$. Let $\omega=\omega_{d,n,q}$. In that case, we find
\begin{equation}
    \label{equation_estimates_omega_a}(1+a^{-2})^{-1}\omega^{(1)}\lesssim \omega^{(2)},\quad (1+a^{-2})\omega^{(2)}\lesssim \omega^{(1)},\quad  |\partial_a \omega^{(1)}|\lesssim a^{-1} \omega^{(1)},\quad |\partial_a \omega^{(2)}|\lesssim a^{-1} \omega^{(2)},
\end{equation}
as well as the following nonlinear equations 
\begin{align}
    \label{equation_non_linear_omega_1}\partial_t(\omega^{(1)}|\gamma_\theta|)&=\{\widetilde{\Psi},\omega^{(1)}|\gamma_\theta|\}_{\theta,a}-|\gamma_\theta|\partial_a\omega^{(1)}\partial_\theta\widetilde{\Psi}+|\gamma_\theta|\partial_\theta\omega^{(1)}\partial_a\widetilde{\Psi}+\omega^{(1)}\frac{\gamma_\theta}{|\gamma_\theta|}\gamma_a\partial_\theta^2\widetilde{\Psi}-\omega^{(1)}\frac{\gamma_\theta}{|\gamma_\theta|}\gamma_\theta\partial_a\partial_\theta\widetilde{\Psi},\\
     \label{equation_non_linear_omega_2} \partial_t(\omega^{(2)}|\gamma_a|)&=\{\widetilde{\Psi},\omega^{(2)}|\gamma_a|\}_{\theta,a}-|\gamma_a|\partial_a\omega^{(2)}\partial_\theta\widetilde{\Psi}+|\gamma_a|\partial_\theta\omega^{(2)}\partial_a\widetilde{\Psi}+\omega^{(2)}\frac{\gamma_a}{|\gamma_a|}\gamma_a\partial_a\partial_\theta\widetilde{\Psi}-\omega^{(2)}\frac{\gamma_a}{|\gamma_a|}\gamma_\theta\partial_a^2\widetilde{\Psi}.
\end{align}
Here, we can see that $\partial_\theta\omega^{(1)}=d\frac{\theta}{|\theta|}\omega^{(1)}_{d-1,n,q}$. Moreover, recalling the estimate of $\partial_a\widetilde{\Psi}$ given in \eqref{equation_estimate_Psi_first_derivative}, we derive 
\begin{align*}
    |\gamma_\theta|\partial_\theta\omega^{(1)}|\partial_a\widetilde{\Psi}|\lesssim&~ t^{-1}\big(\|a^{-2}\gamma\|_{L^1}+t^{-\frac{1}{4}}\|(a^{-5}+a^9+|\theta|^3+\ell^6)\gamma\|_{L^1}\big)|\gamma_\theta|\omega^{(1)}_{d-1,n,q}\\
    &+t^{-1}\|\gamma\|_{L^1}(\omega^{(1)}_{d-1,n,q-3}+\omega^{(1)}_{d-1,n,q+3}+\omega^{(1)}_{d-1,n+2,q}+\omega^{(1)})|\gamma_\theta|.
\end{align*}
Therefore, to estimate $\gamma_\theta$ with the weight $\omega^{(1)}_{d,n,q}$, we need to already have estimates with the weights $\omega^{(1)}_{d-1,n,q-3}$, $\omega^{(1)}_{d-1,n,q+3}$, $\omega^{(1)}_{d-1,n+2,q}$ and $\omega^{(1)}_{d-1,n,q}$. For this purpose, we consider a set of weights $\mathcal{I}=\{\omega_{d,n,q}\}_{d,n,q}$ that need to satisfy 
\begin{equation*}
    \omega_{d,n,q}\in \mathcal{I}\Rightarrow \omega_{d-1,n,q-3},\,\omega_{d-1,n,q+3},\,\omega_{d-1,n+2,q},\,\omega_{d-1,n,q}\in \mathcal{I}, 
\end{equation*}
with $\omega_{-1,n,q}\equiv 0$. Then, we assume, as bootstrap assumption, that we have estimates on $\|\omega^{(1)}\gamma_\theta(t)\|_{L^1}$ and $\|\omega^{(2)}\gamma_a(t)\|_{L^1}$ for any $\omega\in\mathcal{I}$. Here, we will consider
\begin{equation}
    \mathcal{I}:=\{|\theta|^{9-k}a^{-3i_1+3i_2}\ell^{i_3}\,|\, i_1+i_2+i_3=k,\,\, i_1, i_2, i_3, k\in\llbracket 0,9\rrbracket\}.
\end{equation}
In particular, we ensure that 
\begin{equation}
    \mathcal{M}\leq \|(a^{-24}+a^{24}+|\theta|^{8}+\ell^{16})\gamma\|_{L^1}+\sum_{\omega\in \mathcal{I}}\left(\|\omega^{(1)}\gamma_\theta\|_{L^1}+\|\omega^{(2)}\gamma_a\|_{L^1}\right),
\end{equation}
with
\begin{equation}
    \label{equation_upperbound_moments_with_omega}
    \|(a^{-24}+a^{24}+|\theta|^{8}+\ell^{16})\gamma\|_{L^1}+\sum_{\omega\in \mathcal{I}}\left(\|\omega^{(1)}\gamma_\theta\|_{L^1}+\|\omega^{(2)}\gamma_a\|_{L^1}\right)\lesssim  \|(a^{-30}+a^{30}+|\theta|^{10}+\ell^{20})(\gamma+|\gamma_\theta|+|\gamma_a|)\|_{L^1}.
\end{equation}
\begin{proposition}
    \label{proposition_bootstrap_gamma_derivatives}
    Let $\delta \leq \frac{1}{44}$. Assume that for all $\omega\in \mathcal{I}$
    \begin{equation*}
        \|(a^{-24}+a^{24}+|\theta|^{8}+\ell^{16})\gamma(T_0)\|_{L^1}+\|\omega^{(1)}\gamma_\theta(T_0)\|_{L^1}+\|\omega^{(2)}\gamma_a(T_0)\|_{L^1}\leq \varepsilon_0,
    \end{equation*}
    and for all $t\in[T_0,T]$,
    \begin{equation}
        \label{equation_bootstrap_assumption_t}
        \|(a^{-24}+a^{24}+|\theta|^{8}+\ell^{16})\gamma(t)\|_{L^1}\leq \varepsilon_1t^\delta,\qquad \|\omega_{d,n,q}^{(1)}\gamma_\theta(t)\|_{L^1}+\|\omega_{d,n,q}^{(2)}\gamma_a(t)\|_{L^1}\leq \varepsilon_1 t^{(d+1)\delta}
\end{equation}
    If $C\varepsilon_0\leq \delta$, where $C$ is a constant that depends only on $m$ and $T_0$, then the stronger bounds hold 
    \begin{equation*}
        \|(a^{-24}+a^{24}+\ell^{16})\gamma(t)\|_{L^1}\lesssim \varepsilon_0,\qquad \| |\theta|^{8}\gamma(t)\|_{L^1}\lesssim (\varepsilon_0+\varepsilon_1)t^{C\varepsilon_0}, 
    \end{equation*}
    \begin{equation}
        \label{equation_estimate_derivative_bootstrap}
        \|\omega^{(1)}_{d,n,q}\gamma_\theta(t)\|\lesssim \varepsilon_0+\frac{\varepsilon_1(\varepsilon_0+\varepsilon_1)}{\delta} t^{d\delta},\quad \|\omega^{(2)}_{d,n,q}\gamma_a(t)\|\lesssim \varepsilon_0+\frac{\varepsilon_1(\varepsilon_0+\varepsilon_1)}{\delta^2}t^{(d+1)\delta}.
    \end{equation}
\end{proposition}

\begin{proof}
    The estimates of $\gamma$ follow directly from Proposition \ref{proposition_bootstrap_gamma}. For the derivatives, let us begin by considering $\omega=\omega_{d,n,q}\in \mathcal{I}$. By \eqref{equation_non_linear_omega_1} and \eqref{equation_estimates_omega_a}, we have
    \begin{equation}
        \label{equation_proof_derivative_omega_gamma_theta}
        \begin{array}{rl}
         \displaystyle\frac{\mathrm{d}}{\mathrm{d}t}\|\omega^{(1)}\gamma_\theta\|_{L^1}\lesssim& \displaystyle(\|a^{-1}\partial_\theta\widetilde{\Psi}\|_{L^\infty}+\|\partial_\theta\partial_a\widetilde{\Psi}\|_{L^\infty})\|\omega^{(1)}\gamma_\theta\|_{L^1} + \|(1+a^{-2})\partial_\theta^2\widetilde{\Psi}\|_{L^\infty}\|\omega^{(2)}\gamma_a\|_{L^1}\\
        \displaystyle &\displaystyle+\iiint |\gamma_\theta\omega^{(1)}_{d-1,n,q}\partial_a\widetilde{\Psi}|\mathrm{d}a\mathrm{d}\theta\mathrm{d}\ell.
        \end{array}
    \end{equation}
    The first two terms can be readily estimated using Propositions \ref{proposition_estimates_psi_first_derivative}--\ref{proposition_estimates_psi_second_derivatives} and the bootstrap assumption. Indeed, we find
    \begin{equation*}
        \|a^{-1}\partial_\theta\widetilde{\Psi}\|_{L^\infty}\|\omega^{(1)}\gamma_\theta\|_{L^1}\lesssim \varepsilon_0\varepsilon_1 t^{(d+2)\delta-\frac{3}{2}},\quad \|\partial_\theta\partial_a\widetilde{\Psi}\|_{L^\infty}\|\omega^{(1)}\gamma_\theta\|_{L^1}+\|(1+a^{-2})\partial_\theta^2\widetilde{\Psi}\|_{L^\infty}\|\omega^{(2)}\gamma_a\|_{L^1} \lesssim \varepsilon_1^2t^{(d+11)\delta-\frac{3}{2}}.
    \end{equation*}
    Hence, since $\delta \leq \frac{1}{44}$, these terms can be bounded by $C\varepsilon_1(\varepsilon_0+\varepsilon_1)t^{-\frac{5}{4}+d\delta}$. When $d=0$, the last term of the inequality \eqref{equation_proof_derivative_omega_gamma_theta} vanishes, and we directly obtain the result by integrating. When $d\neq0$, by \eqref{equation_estimate_Psi_first_derivative}, we have
    \begin{align*}
        \iiint |\gamma_\theta\omega^{(1)}_{d-1,n,q}\partial_a\widetilde{\Psi}|\mathrm{d}a\mathrm{d}\theta\mathrm{d}\ell\lesssim &~t^{-1}\|\gamma\|_{L^1}\|(\omega^{(1)}_{d-1,n,q-3}+\omega^{(1)}_{d-1,n,q+3}+\omega^{(1)}_{d-1,n+2,q})\gamma_\theta\|_{L^1}+t^{-1}\|\gamma\|_{L^1}\|\omega^{(1)}_{d,n,q}\gamma_\theta\|_{L^1}\\
        &+t^{-1}\big(\|a^{-2}\gamma\|_{L^1}+t^{-\frac{1}{4}}\|(a^{-5}+a^9+|\theta|^3+\ell^6)\gamma\|_{L^1}\big)\|\omega^{(1)}_{d-1,n,q}\gamma_\theta\|_{L^1}.
    \end{align*}
    With the bootstrap assumption and Proposition \ref{proposition_bootstrap_gamma}, this implies 
    \begin{equation*}
        \frac{\mathrm{d}}{\mathrm{d}t}\|\omega^{(1)}\gamma_\theta\|_{L^1}\lesssim (\varepsilon_0+\varepsilon_1)\varepsilon_1 t^{-1+d\delta}+\varepsilon_0t^{-1}\|\omega^{(1)}\gamma_\theta\|_{L^1}.
    \end{equation*}
    Hence, by Gronwall's inequality
    \begin{equation*}
        \|\omega^{(1)}\gamma_\theta\|_{L^1}\lesssim \varepsilon_0+\frac{\varepsilon_1(\varepsilon_0+\varepsilon_1)}{\delta}(t^{d\delta}+t^{C\varepsilon_0}),
    \end{equation*}
    where $C$ depends only on $m$ and $T_0$. Provided $C\varepsilon_0\leq \delta$, we obtain the estimate for  $ \|\omega^{(1)}\gamma_\theta\|_{L^1}$.\\
    Then, by \eqref{equation_non_linear_omega_2}, we have
    \begin{align*}
        \frac{\mathrm{d}}{\mathrm{d}t}\|\omega^{(2)}\gamma_a\|_{L^1}\lesssim&~\big(\|a^{-1}\partial_\theta\widetilde{\Psi}\|_{L^1}+\|\partial_\theta\partial_a\widetilde{\Psi}\|_{L^1}\big)\|\omega^{(2)}\gamma_a\|_{L^1}+\|(1+a^{-2})^{-1}\partial_a^2\widetilde{\Psi}\|_{L^\infty}\|\omega^{(1)}\gamma_\theta\|_{L^1}\\
        &+\iiint|\gamma_a\omega_{d-1,n,q}^{(2)}\partial_a\widetilde{\Psi}|\mathrm{d}\theta\mathrm{d}a\mathrm{d}\ell.
    \end{align*}
    Again, by Propositions \ref{proposition_estimates_psi_first_derivative}--\ref{proposition_estimates_psi_second_derivatives} and the bootstrap assumption, the first two terms can be estimated by the following
    \begin{equation*}
        \big(\|a^{-1}\partial_\theta\widetilde{\Psi}\|_{L^1}+\|\partial_\theta\partial_a\widetilde{\Psi}\|_{L^1}\big)\|\omega^{(2)}\gamma_a\|_{L^1}\lesssim t^{-\frac{3}{2}}\mathcal{M}\|\omega^{(2)}\gamma_a\|_{L^1}\lesssim t^{-\frac{5}{4}+d\delta}.
    \end{equation*}
    Then, using the estimate we obtained on $\|\omega^{(1)}\gamma_\theta\|_{L^1}$, we find
    \begin{equation*}
        \|(1+a^{-2})^{-1}\partial_a^2\widetilde{\Psi}\|_{L^\infty}\|\omega^{(1)}\gamma_\theta\|_{L^1}\lesssim t^{-1+(d+1)\delta}\frac{1}{\delta}\varepsilon_1(\varepsilon_0+\varepsilon_1)+\frac{1}{\delta}t^{-\frac{3}{2}+(d+10)\delta}\varepsilon_1(\varepsilon_0+\varepsilon_1)\lesssim t^{-1+(d+1)\delta}\varepsilon_1(\varepsilon_0+\varepsilon_1).
    \end{equation*}
    For the last term, we have
    \begin{align*}
        \iiint |\gamma_a\omega^{(2)}_{d-1,n,q}\partial_a\widetilde{\Psi}|\mathrm{d}a\mathrm{d}\theta\mathrm{d}\ell\lesssim &~t^{-1}\|\gamma\|_{L^1}\|(\omega^{(2)}_{d-1,n,q-3}+\omega^{(2)}_{d-1,n,q+3}+\omega^{(2)}_{d-1,n+2,q})\gamma_a\|_{L^1}+t^{-1}\|\gamma\|_{L^1}\|\omega^{(2)}_{d,n,q}\gamma_a\|_{L^1}\\
        &+t^{-1}\big(\|a^{-2}\gamma\|_{L^1}+t^{-\frac{1}{4}}\|(a^{-5}+a^9+|\theta|^3+\ell^6)\gamma\|_{L^1}\big)\|\omega^{(2)}_{d-1,n,q}\gamma_a\|_{L^1}.
    \end{align*}
    Using the bootstrap assumption, we find
    \begin{equation*}
        \iiint |\gamma_a\omega^{(2)}_{d-1,n,q}\partial_a\widetilde{\Psi}|\mathrm{d}a\mathrm{d}\theta\mathrm{d}\ell\lesssim \frac{1}{\delta}\varepsilon_1(\varepsilon_0+\varepsilon_1)t^{-1+(d+1)\delta}.
    \end{equation*}
    By combining the last estimates and integrating, we derive
    \begin{equation*}
        \|\omega^{(2)}\gamma_a\|_{L^1}\lesssim \varepsilon_0+\frac{\varepsilon_1(\varepsilon_0+\varepsilon_1)}{\delta^2}t^{(d+1)\delta}.
    \end{equation*}
\end{proof}
We can then use Propositions \ref{proposition_bootstrap_gamma}--\ref{proposition_bootstrap_gamma_derivatives}, together with a bootstrap argument, to deduce the existence of a global solution provided $\varepsilon_0$ and $\varepsilon_1$ are small enough. 
\begin{corollary}
    Let $\varepsilon>0$ small enough, and $\gamma$ be a local solution to \eqref{equation_non_linear_gamma} with initial data $\gamma_0$. Assume that $\gamma_0$ satisfies, for all $\omega\in \mathcal{I}$,
    \begin{equation*}
        \|(a^{-30}+a^{30}+|\theta|^{10}+\ell^{20})\gamma_0\|_{L^1}+\|\omega^{(1)}\partial_\theta\gamma_0\|_{L^1}+\|\omega^{(2)}\partial_a\gamma_0\|_{L^1}\leq \varepsilon.
    \end{equation*}
    Then $\gamma$ is a global solution. Moreover, it satisfies the estimates given in Propositions \ref{proposition_bootstrap_gamma}--\ref{proposition_bootstrap_gamma_derivatives}.
\end{corollary}
\begin{proof}
    By a standard breakdown criterion, this follows from Propositions \ref{proposition_bootstrap_gamma}--\ref{proposition_bootstrap_gamma_derivatives} and \eqref{equation_upperbound_moments_with_omega}.
\end{proof}

\section{Asymptotics}
\label{section_asymptotics}

When we defined $\gamma$, we already composed with the linear flow. Consequently, to establish a linear scattering result, we may directly try to prove the convergence of $\gamma$ towards a scattering state. To do so, it is enough to show that $\partial_t\gamma$ is integrable in time. However, using Propositions \ref{proposition_estimates_psi_first_derivative} and \ref{proposition_bootstrap_gamma_derivatives}, we find
\begin{equation*}
    \|\partial_t\gamma\|_{L^1}\leq \|\partial_a\gamma\partial_\theta\widetilde{\Psi}\|_{L^1} +\|\partial_\theta\gamma\partial_a\widetilde{\Psi}\|_{L^1}\lesssim (\varepsilon_0+\varepsilon_1)t^{-\frac{5}{4}}+t^{-1}(\varepsilon_0+\varepsilon_1)t^{\delta}\lesssim (\varepsilon_0+\varepsilon_1)t^{-1+\delta}.
\end{equation*}
Here, the lack of integrability arises from the electric field $\partial_a\widetilde{\Psi}$, which decays as $t^{-1}$. We focus on finding the asymptotic limit of $t\partial_a\widetilde{\Psi}$ and consider a modification to the linear characteristics to absorb the leading-order term of \eqref{equation_non_linear_gamma}. In order to give an idea of the expected asymptotic behavior of the field, let us first recall \eqref{equation_def_derivatives_psi} and Proposition \ref{proposition_derivatives_R_tilde}, so that
\begin{equation*}
    \partial_a\widetilde{\Psi}=-\partial_a\widetilde{R}\frac{M(t,\widetilde{R})}{\widetilde{R}^2}=-\bigg[t\frac{m^2}{(a^0)^3}\partial_\theta\widetilde{R}+(\partial_a R)(\theta+t\widehat a,a)\bigg]\frac{M(t,\widetilde{R})}{\widetilde{R}^2}.
\end{equation*}
Moreover, we have $(\theta,a)\in\mathcal{B}_t$ provided that $t$ is sufficiently large. Then, recall that, in the bulk, $\widetilde{R}(\theta,a)\sim t\widehat{a}$ as well as 
\begin{equation*}
   \partial_\theta\widetilde{R}\xrightarrow[t\rightarrow +\infty]{} 1,\qquad (\partial_a R)(\theta+t\widehat a,a)\frac{M(t,\widetilde{R})}{\widetilde{R}^2}\lesssim t^{-\frac{6}{5}}.
\end{equation*}
Hence, heuristically, we find that for large times
\begin{equation*}
    \partial_a\widetilde{\Psi}\sim -t\frac{m^2}{(a^0)^3}\frac{M(t,t\widehat a)}{(t\widehat a)^2}.
\end{equation*}
Furthermore, using similar arguments, 

\begin{equation*}
    M(t,t\widehat a)=4\pi^2\iiint \1_{\{\widetilde{R}(\theta,\alpha)\leq t\widehat a\}}\gamma(t,\theta,\alpha,\ell)\mathrm{d}\theta\mathrm{d}\alpha\mathrm{d}\ell\sim4\pi^2\iiint \1_{\{t\widehat\alpha\leq t\widehat a\}}\gamma(t,\theta,\alpha,\ell)\mathrm{d}\theta\mathrm{d}\alpha\mathrm{d}\ell=:\mathcal{E}(t,a).
\end{equation*}
We will then prove that $\mathcal{E}(t,a)$ converges to $\mathcal{E}_\infty(a)$ as $t$ goes to infinity. This will allow us to derive
\begin{equation*}
    \partial_a\widetilde{\Psi}\sim -\frac{1}{t}\frac{m^2}{a^0 a^2}\mathcal{E}_\infty(a).
\end{equation*}
In the following, we aim at proving this rigorously.

\subsection{Asymptotic behavior}
     We begin by defining a smaller version of the bulk. Let
    \begin{equation}
        \mathcal{B}_t^*:=\{(\theta,a,\ell)\,|\,|\theta|\leq t^{\frac{1}{4}},\, a\leq t^{\frac{1}{12}},\, \widehat a\geq t^{-\frac{1}{4}},\, \sqrt{\ell}\leq t^{\frac{1}{4}}\}\subset \mathcal{B}_t.
    \end{equation}
    In this subset, we know that, for large times, $\widetilde{R}$ does not deviate too far from $t\widehat a$.
    \begin{lemma}
        \label{lemma_difference_R_tilde_ta}
        Let $t\geq T_0$ be sufficiently large. Then
        \begin{equation}
            \1_{\mathcal{B}_t^*}|\widetilde{R}(\theta,a)-t\widehat a|\lesssim t^{\frac{1}{2}}\log(t).
        \end{equation}
    \end{lemma}
    \begin{proof}
        Let $\vartheta:=\theta+t\widehat{a}$. Recall
        \begin{equation*}
            \widetilde{R}(\theta,a)=pH_a\left(\frac{|\vartheta|}{p}\right)+\kappa p.
        \end{equation*}
        In $\mathcal{B}_t^*$, we have 
        \begin{equation*}
             mt^{-\frac{1}{6}}\leq p\leq \frac{2}{m}t^\frac{1}{2},\qquad \frac{1}{2}t^\frac{3}{4}\leq|\vartheta|\leq 2t.
        \end{equation*}
        Hence, $\frac{|\vartheta|}{p}\geq \frac{m}{4}t^{\frac{1}{4}}$, and for $t$ sufficiently large, we may apply \eqref{equation_DL_G_H_infinity} and find, for $(\theta,a,\ell)\in \mathcal{B}_t^*$,
        \begin{equation*}
            \widetilde{R}(\theta,a)-t\widehat a=|\vartheta|-t\widehat{a}-p\left(\kappa-\frac{1}{a^0p}\right)\log\left(2\frac{|\vartheta|}{p}\right)+p^2\left(\kappa-\frac{1}{a^0p}\right)^2\frac{\log\left(2\frac{|\vartheta|}{p}\right)}{2|\vartheta|}+\kappa p+O_{t\rightarrow +\infty}\left(\frac{p}{|\vartheta|}\right).
        \end{equation*}
        Consequently, 
        \begin{equation*}
            \1_{\mathcal{B}_t^*}|\widetilde{R}(\theta,a)-t\widehat{a}|\lesssim \1_{\mathcal{B}_t^*}||\vartheta|-t\widehat{a}|+t^{\frac{1}{2}}\log(t)\lesssim \1_{\mathcal{B}_t^*}|\theta|+t^{\frac{1}{2}}\log(t)\lesssim t^{\frac{1}{2}}\log(t) 
        \end{equation*}
    \end{proof}
    
    We can also derive the asymptotics of $\partial_a\widetilde{\Psi}$ in $\mathcal{B}_t^*$.
\begin{proposition}
    \label{proposition_asymp_E_psi}
   Let 
   \begin{equation*}
       \mathcal{E}(t,a):=4\pi^2\iiint \1_{\{\alpha\leq a\}}\gamma(t,\theta,\alpha,\ell)\mathrm{d}\theta\mathrm{d}\alpha\mathrm{d}\ell.
   \end{equation*}
   Then, there exists $\mathcal{E}_\infty\in L^\infty((0,+\infty))$ such that 
   \begin{equation}
        \label{equation_asymptotics_mathcal_E}
       |\mathcal{E}(t,a)-\mathcal{E}_\infty(a)|\lesssim \varepsilon_1^2t^{-\frac{1}{4}},\qquad |\mathcal{E}_\infty(a)|\lesssim \varepsilon_0. 
   \end{equation}
   Moreover, in $\mathcal{B}_t^*$, we find
\begin{equation}
    \label{equation_asymptotics_psi}
    \1_{\mathcal{B}_t^*}\left|\partial_a\widetilde{\Psi}+\frac{1}{t}\frac{m^2}{a^0a^2}\mathcal{E}_\infty(a)\right|\lesssim \varepsilon_1 t^{-\frac{6}{5}}.
\end{equation}
\end{proposition}
\begin{proof}
    We begin by considering $\partial_t \mathcal{E}(t,a)$. Using an integration by parts in $\theta$, we find
    \begin{equation*}
        \partial_t \mathcal{E}(t,a)=4\pi^2\iiint\1_{\{\alpha\leq a\}}\{\gamma,\widetilde{\Psi}\}\mathrm{d}\theta\mathrm{d}\alpha\mathrm{d}\ell=4\pi^2\iiint\1_{\{\alpha\leq a\}}\gamma_a\partial_\theta\widetilde{\Psi}\mathrm{d}\theta\mathrm{d}\alpha\mathrm{d}\ell+4\pi^2\iiint\1_{\{\alpha\leq a\}}\gamma\partial_\theta\partial_a\widetilde{\Psi}\mathrm{d}\theta\mathrm{d}\alpha\mathrm{d}\ell.
    \end{equation*}
    Using the estimates given in Propositions \ref{proposition_estimates_psi_first_derivative}--\ref{proposition_estimates_psi_second_derivatives}, and Propositions \ref{proposition_bootstrap_gamma}--\ref{proposition_bootstrap_gamma_derivatives}, this implies
    \begin{equation*}
        |\partial_t \mathcal{E}(t,a)|\lesssim \|a\gamma_a\|_{L^1}\|a^{-1}\partial_\theta\widetilde{\Psi}\|_{\infty}+\|\gamma\|_{L^1}\|\partial_\theta\partial_a\widetilde{\Psi}\|_\infty\lesssim \varepsilon_1^2 t^{-\frac{5}{4}}.
    \end{equation*}
    This implies the first estimate of \eqref{equation_asymptotics_mathcal_E}. Moreover, by Proposition \ref{proposition_bootstrap_gamma}, we directly derive $|\mathcal{E}(t,a)|\lesssim \varepsilon_0$ and then $|\mathcal{E}_\infty(a)|\lesssim \varepsilon_0$. We now prove \eqref{equation_asymptotics_psi}. First, write
    \begin{equation*}
        \partial_a\widetilde{\Psi}=-t\frac{m^2}{(a^0)^3}\frac{M(t,\widetilde{R})}{\widetilde{R}^2}+\mathcal{N}_1,
    \end{equation*}
    where 
    \begin{equation*}
        \mathcal{N}_1:=\frac{M(t,\widetilde{R})}{\widetilde{R}^2}\left[t\frac{m^2}{(a^0)^3}(1-\partial_\theta\widetilde{R})-(\partial_aR)(\theta+t\widehat a,a)\right].
    \end{equation*}
    Let us prove that $\1_{\mathcal{B}_t}|\mathcal{N}_1|\lesssim \varepsilon_1 t^{-\frac{6}{5}}$. First, since $\theta+t\widehat{a}\geq 0$ in $\mathcal{B}_t$,  by Proposition \ref{proposition_writing_derivatives_R}, we have
    \begin{equation*}
        \1_{\mathcal{B}_t}(1-\partial_\theta\widetilde{R})= \frac{\1_{\mathcal{B}_t}}{a^2(\widetilde{R}-\frac{1}{a^0})\left(\widetilde{R}-\frac{1}{a^0}+\sqrt{(\widetilde{R}-p\kappa)^2-p^2}\right)}\left(2\frac{m^2\widetilde{R}}{a^0}-\frac{m^2}{(a^0)^2}+\ell\right).
    \end{equation*}
    
    Hence, by \eqref{equation_lower_bound_R_a0_proof}
    \begin{equation*}
        \1_{\mathcal{B}_t}|1-\partial_\theta\widetilde{R}|\lesssim  \1_{\mathcal{B}_t}\frac{a^0}{\widetilde{R}a^2 \left(a^0-\frac{1}{\widetilde{R}}\right)}\left(\frac{1}{\left(a^0-\frac{1}{\widetilde{R}}\right)}+\frac{\ell}{\left(\widetilde{R}-\frac{1}{a^0}\right)}\right) \lesssim \1_{\mathcal{B}_t}\frac{1}{\widetilde{R}a^2}(a^0+\sqrt{\ell}a).
    \end{equation*}
    Thus, since $\widetilde{R}\geq\frac{t\widehat{a}}{4}$ holds in the bulk,
    \begin{equation*}
        \1_{\mathcal{B}_t}|1-\partial_\theta\widetilde{R}|\lesssim t^{-1}\left(\frac{1}{a}+\frac{1}{a^3}+\sqrt{\ell}+\frac{\sqrt{\ell}}{a^2}\right).
    \end{equation*}
    Consequently, by \eqref{equation_estimate_first_derivatives_R}, since $\sqrt{\ell}\leq t^\frac{1}{4},\,a^{-1}\leq \frac{1}{m}t^\frac{1}{4}$ in the bulk, 
    \begin{align*}
        |\1_{\mathcal{B}_t}\mathcal{N}_1|&\lesssim \1_{\mathcal{B}_t}\frac{M(t,\widetilde{R})}{\widetilde{R}^2}\left(\frac{1}{a}+\frac{1}{a^3}+\sqrt{\ell}+\frac{\sqrt{\ell}}{a^2}\right)\log\left\langle\frac{\widetilde{R}}{p}\right\rangle\lesssim \1_{\mathcal{B}_t}t^\frac{3}{4}\frac{M(t,\widetilde{R})}{\widetilde{R}^2}\log\langle\widetilde{R}a^2\rangle\lesssim t^\frac{1}{60}t^\frac{3}{4}\frac{M(t,\widetilde{R})}{\widetilde{R}^{\frac{59}{30}}}.
    \end{align*}
    Since $\mathcal{B}_t^*\subset \mathcal{B}_t$, the same upper bound holds for $\1_{\mathcal{B}_t^*}\mathcal{N}_1$. Applying Lemma \ref{lemma_estimate_total_mass} with $k=\frac{59}{30},\sigma=0$, and Proposition \ref{proposition_bootstrap_gamma}, we derive the estimate for $\1_{\mathcal{B}_t^*}\mathcal{N}_1$. Next, we write
    \begin{equation*}
        \partial_a\widetilde{\Psi}=-t^{-1}\frac{m^2}{(a^0)^3}\frac{M(t,t\widehat a)}{(t\widehat{a})^2}+\mathcal{N}_2+\mathcal{N}_1,
    \end{equation*}
    where 
    \begin{equation*}
        \mathcal{N}_2:=t\frac{m^2}{(a^0)^3}\left[\frac{M(t,\widetilde{R})}{\widetilde{R}^2}-\frac{M(t,t\widehat a)}{(t\widehat{a})^2}\right].
    \end{equation*}
    By the mean value theorem and Lemma \ref{lemma_difference_R_tilde_ta}, we find
    \begin{align*}
        \1_{\mathcal{B}_t^*}\left|\frac{M(t,\widetilde{R})}{\widetilde{R}^2}-\frac{M(t,t\widehat a)}{(t\widehat{a})^2}\right|\leq \sup_{r>0}\left(\frac{|\rho(t,r)|}{r^2}+2\frac{M(t,r)}{r^3}\right)\1_{\mathcal{B}_t^*}|\widetilde{R}(\theta,a)-t\widehat{a}|\lesssim \sup_{r>0}\left(\frac{|\rho(t,r)|}{r^2}+2\frac{M(t,r)}{r^3}\right)t^{\frac{1}{2}}\log(t).
    \end{align*}
    Then, using Lemmas \ref{lemma_estimate_total_mass}--\ref{lemma_estimate_charge_density} and Propositions \ref{proposition_bootstrap_gamma}--\ref{proposition_bootstrap_gamma_derivatives}, we derive
    \begin{equation*}
        \1_{\mathcal{B}_t^*}|\mathcal{N}_2|\lesssim \varepsilon_1t^{-\frac{3}{2}}t^{11\delta}\log(t)\lesssim\varepsilon_1t^{-\frac{6}{5}}.
    \end{equation*}
    Finally, 
    \begin{equation*}
        \partial_a\widetilde{\Psi}=-t\frac{m^2}{(a^0)^3}\frac{1}{\widehat{a}^2}\mathcal{E}(t,a)+\mathcal{N}_3+\mathcal{N}_2+\mathcal{N}_1,
    \end{equation*}
    with 
    \begin{equation*}
        \mathcal{N}_3:=t^{-1}\frac{4\pi^2m^2}{(a^0)^3}\frac{1}{(\widehat{a})^2}\iiint(\1_{\{t\widehat{\alpha}\leq t\widehat{a}\}}-\1_{\{\widetilde{R}(\vartheta,\alpha)\leq t\widehat{a}\}})\gamma(\vartheta,\alpha,\ell)\mathrm{d}\vartheta\mathrm{d}\alpha\mathrm{d}\ell.
    \end{equation*}
    Now, we split the integral over the restricted bulk $\mathcal{B}_t^*$ and its complement. First, consider
    \begin{equation*}
        \mathcal{N}_3^O:=\frac{1}{(\widehat{a})^2}\iiint(\1_{\{t\widehat{\alpha}\leq t\widehat{a}\}}-\1_{\{\widetilde{R}(\vartheta,\alpha)\leq t\widehat{a}\}})\1_{(\mathcal{B}_t^*)^c}\gamma\mathrm{d}\vartheta\mathrm{d}\alpha\mathrm{d}\ell.
    \end{equation*}
    Moreover, in the bulk $(\widehat{a})^{-2}\leq t^\frac{1}{2}$. Similarly, for any $k\geq 0$
    \begin{equation*}
        \1_{(\mathcal{B}_t^*)^c}\lesssim t^{-k}(a^{-4k}+a^{12k}+|\theta|^{4k}+\ell^{2k}).
    \end{equation*}
    Hence, for $k=\frac{3}{4}$, using Proposition \ref{proposition_bootstrap_gamma}, we find
    \begin{equation*}
       \1_{\mathcal{B}_t^*} |\mathcal{N}_3^O|\lesssim t^{-\frac{1}{4}}\|(a^{-3}+a^{9}+|\theta|^{3}+\ell^{\frac{3}{2}})\gamma\|_{L^1}\lesssim \varepsilon_1t^{-\frac{1}{5}}.
    \end{equation*}
    Then, we consider 
     \begin{align*}
        \mathcal{N}_3^I:=~&\frac{1}{(\widehat{a})^2}\iiint(\1_{\{t\widehat{\alpha}\leq t\widehat{a}\}}-\1_{\{\widetilde{R}(\vartheta,\alpha)\leq t\widehat{a}\}})\1_{\mathcal{B}_t^*}\gamma\mathrm{d}\vartheta\mathrm{d}\alpha\mathrm{d}\ell\\
        =~&\frac{1}{(\widehat{a})^2}\iiint(\1_{\{\widetilde{R}(\vartheta,\alpha)>t\widehat{a},\,\,\alpha\leq a\}}-\1_{\{\widetilde{R}(\vartheta,\alpha)\leq t\widehat{a},\,\, \alpha>a\}})\1_{\mathcal{B}_t^*}\gamma\mathrm{d}\vartheta\mathrm{d}\alpha\mathrm{d}\ell.
    \end{align*}
    Here, by Lemma \ref{lemma_difference_R_tilde_ta}, we have
    \begin{equation*}
        \mathcal{B}_t^*\cap\left(\{\widetilde{R}(\vartheta,\alpha)>t\widehat{a},\,\,\alpha\leq a\}\cup\{\widetilde{R}(\vartheta,\alpha)\leq t\widehat{a},\,\, \alpha>a\}\right)\subset \{|\widehat{a}-\widehat{\alpha}|\lesssim t^{-\frac{1}{2}}\log(t)\}.
    \end{equation*}
    Hence, since $(\widehat{a})^{-2}\leq t^{\frac{1}{2}}$ in the bulk,
    \begin{equation*}
        \1_{\mathcal{B}_t^*} |\mathcal{N}_3^I|\lesssim \1_{\mathcal{B}_t^*}\frac{1}{\widehat{a}^2}\iiint \1_{\{|\widehat{a}-\widehat{\alpha}|\lesssim t^{-\frac{1}{2}}\log(t)\}}\1_{\mathcal{B}_t^*}\gamma\mathrm{d}\vartheta\mathrm{d}\alpha\mathrm{d}\ell\lesssim  \1_{\mathcal{B}_t^*}\iiint \1_{\{|\widehat{a}-\widehat{\alpha}|\lesssim t^{-\frac{1}{2}}\log(t)\}}\1_{\mathcal{B}_t^*}\frac{\log(t)}{\widehat{\alpha}^2}\gamma\mathrm{d}\vartheta\mathrm{d}\alpha\mathrm{d}\ell.
    \end{equation*}
    Moreover, since $a,\alpha\leq t^{\frac{1}{12}}$ in the restricted bulk, using the mean value theorem, we obtain
    \begin{equation*}
        \{|\widehat{a}-\widehat{\alpha}|\lesssim t^{-\frac{1}{2}}\log(t)\}\subset \{|a-\alpha|\lesssim t^{-\frac{1}{4}}\log(t)\}.
    \end{equation*}
    We also have, using the fundamental theorem of calculus, 
    \begin{equation*}
        \iint(1+\alpha^{-2})\gamma(t,\vartheta,\alpha,\ell)\mathrm{d}\theta\mathrm{d}\ell \lesssim \|(1+a^{-2})\gamma_a\|_{L^1}+\|a^{-3}\gamma\|_{L^1}.
    \end{equation*}
    Combining these properties and Propositions \ref{proposition_bootstrap_gamma}--\ref{proposition_bootstrap_gamma_derivatives}, we derive
    \begin{equation*}
        \1_{\mathcal{B}_t^*} |\mathcal{N}_3^I|\lesssim \varepsilon_1t^\delta\log(t)\int_0^{+\infty} \1_{\{|a-\alpha|\lesssim t^{-\frac{1}{4}}\log(t)\}}\mathrm{d}\alpha\lesssim \varepsilon_1t^{-\frac{1}{5}}.
    \end{equation*}
    This implies
    \begin{equation*}
        \1_{\mathcal{B}_t^*}|\mathcal{N}_3|\lesssim \varepsilon_1t^{-\frac{6}{5}}.
    \end{equation*}
    The final estimate \eqref{equation_asymptotics_psi} follows from \eqref{equation_asymptotics_mathcal_E}.
\end{proof}


\subsection{Modified scattering}

Now that we have expressed the higher-order term in the asymptotic expansion of $\partial_a\widetilde{\Psi}$, we can define the modification of the action characteristic and establish the convergence to a scattering state.

\begin{theorem}
    Let $\mu$ be defined as 
    \begin{equation*}
        \mu(t,\theta,a,\ell):= \gamma\left(t,\theta-\log(t)\frac{m^2}{a^0a^2}\mathcal{E}_\infty(a),a,\ell\right).
    \end{equation*}
    Then, there exists $\gamma_\infty\in L^1$ such that
    \begin{equation*}
        \|\mu(t)-\gamma_\infty\|_{L^1}\lesssim \varepsilon_1^2 t^{-\frac{1}{6}}.
    \end{equation*}
\end{theorem}
\begin{proof}
    First, note that 
    \begin{equation*}
        \partial_t \mu(t,\theta,a,\ell)=\partial_a\gamma\partial_\theta\widetilde{\Psi}-\partial_\theta \gamma\left(\partial_a\widetilde{\Psi}+\frac{1}{t}\frac{m^2}{a^0a^2}\mathcal{E}_\infty\right),
    \end{equation*}
    where the terms on the right-hand side are evaluated at $\left(t,\theta-\log(t)\frac{m^2}{a^0a^2}\mathcal{E}_\infty(a),a,\ell\right)$. Here, by Proposition \ref{proposition_estimates_psi_first_derivative} and \ref{proposition_bootstrap_gamma_derivatives}, we already have
    \begin{equation*}
        \|\partial_a\gamma\partial_\theta\widetilde{\Psi}\|_{L^1}\lesssim \varepsilon_1 t^{-\frac{4}{3}}\|a\partial_a\gamma\|_{L^1} \lesssim \varepsilon_1^2t^{-\frac{6}{5}}.
    \end{equation*}
    Using \eqref{equation_asymptotics_psi}, we also know that, in the restricted bulk, 
    \begin{equation*}
        \left\|\1_{\mathcal{B}_t^*}\partial_\theta \gamma\left(\partial_a\widetilde{\Psi}+\frac{1}{t}\frac{m^2}{a^0a^2}\mathcal{E}_\infty\right)\right\|_{L^1}\lesssim \varepsilon_1t^{-\frac{6}{5}}\|\partial_\theta \gamma\|_{L^1}\lesssim \varepsilon_1^2 t^{-\frac{7}{6}}.
    \end{equation*}
    It remains to study the integral in the complement set. In that case, we already know that, for any $k\geq 0$
    \begin{equation*}
        \1_{(\mathcal{B}_t^*)^c}\lesssim t^{-k}(a^{-4k}+a^{12k}+|\theta|^{4k}+\ell^{2k}).
    \end{equation*}
    Consequently, by Propositions \ref{proposition_estimates_psi_first_derivative} and \ref{proposition_bootstrap_gamma_derivatives}, we obtain
    \begin{equation*}
        \|\1_{(\mathcal{B}_t^*)^c}\partial_\theta \gamma\partial_a\widetilde{\Psi}\|_{L^1}\lesssim \varepsilon_1^2 t^{-\frac{7}{6}}.
    \end{equation*}
    Similarly,
    \begin{equation*}
        \left\|\1_{(\mathcal{B}_t^*)^c}\partial_\theta \gamma\frac{1}{t}\frac{m^2}{a^0a^2}\mathcal{E}_\infty\right\|_{L^1}\lesssim \varepsilon_1^2t^{-\frac{7}{6}}.
    \end{equation*}
    Hence, 
    \begin{equation*}
        \|\partial_t \mu\|_{L^1}\lesssim \varepsilon_1^2 t^{-\frac{7}{6}}, 
    \end{equation*}
    and we derive the result by integrating.
\end{proof}

\subsection{Going from action-angle to real coordinates}
\label{section_going_back}

Now that we proved the existence of a global solution for the non-linear problem in action-angle variables \eqref{equation_non_linear_gamma}, we would like to obtain a similar result for the Vlasov-Maxwell system with a point charge \eqref{equation_VM_pc_radial_case}. We already know that if $\gamma$ is a solution to \eqref{equation_non_linear_gamma}, then
\begin{equation}
    \label{equation_definition_f_from_gamma}
    f(t,r,u,\ell):=\gamma(t,\Theta(r,u)-t\widehat{\mathcal{A}}(r,u),\mathcal{A}(r,u),\ell)
\end{equation}
solves \eqref{equation_VM_pc_radial_case}. Conversely, given a solution $f$ to \eqref{equation_VM_pc_radial_case}, we recover a solution to \eqref{equation_non_linear_gamma} by considering $\gamma$ defined in  \eqref{equation_def_gamma}. 
\begin{proof}[Proof of Theorem \ref{main_theorem}]
    Consider
    \begin{equation*}
        \gamma_0(\theta,a,\ell):=f_0(R(\theta,a),U(\theta,a),\ell).
    \end{equation*}
    If we prove that $\gamma_0$ satisfies the hypothesis of Theorem \ref{main_theorem_action_angle}, we recover a global solution $\gamma$ to \eqref{equation_non_linear_gamma}. We then obtain a solution to \eqref{equation_VM_pc_radial_case} with initial data $f_0$ by considering $f$ as in \eqref{equation_definition_f_from_gamma}. Let us then prove that $\gamma_0$ satisfies the assumption of Theorem \ref{main_theorem_action_angle}. First, note that
    \begin{equation*}
        \partial_a\gamma_0=\partial_aR(\partial_rf_0)(R,U,\ell)+\partial_aU(\partial_uf_0)(R,U,\ell),\qquad \partial_\theta\gamma_0=\partial_\theta R(\partial_rf_0)(R,U,\ell)+\partial_\theta U(\partial_uf_0)(R,U,\ell).
    \end{equation*}
    We then use action-angle variables to recover estimates in the original variables. For instance, we have
    \begin{equation*}
        \|(a^{30}+a^{-30}+|\theta|^{10}+\ell^{20})\partial_aR(\partial_rf_0)(R,U,\ell)\|_{L^1_{\theta,a,\ell}}=\|(\mathcal{A}^{30}+\mathcal{A}^{-30}+|\Theta|^{10}+\ell^{20})(\partial_aR)(\Theta,\mathcal{A})\partial_rf_0\|_{L^1_{r,u,\ell}}.
    \end{equation*}
   It remains to estimate the different quantities, namely $\mathcal{A},\Theta$, $\partial_\theta R(\Theta,\mathcal{A}),\partial_aR(\Theta,\mathcal{A})$, $\partial_\theta U(\Theta,\mathcal{A}),\partial_aU(\Theta,\mathcal{A})$, in terms of $r,u,\ell$. From the definitions of $\mathcal{A}$ and $\Theta$ in \eqref{equation_definition_action_angle_A_Theta}, we have
    \begin{equation}
        \label{equation_bounds_A_Theta_proof}
        \frac{1}{r}\leq \mathcal{A}(r,u)\lesssim 1+|u|+\frac{\sqrt{\ell}}{r}+\frac{1}{r},\qquad |\Theta|(r,u)\leq 2r.
    \end{equation}
    Hence, 
    \begin{equation*}
        \mathcal{A}^{30}+\mathcal{A}^{-30}+|\Theta|^{10}+\ell^{20}\lesssim |u|^{30}+\ell^{40}+r^{-60}+r^{30}.
    \end{equation*}
    Moreover, since $\mathcal{A}\geq \frac{1}{r}$, using the expression of $\partial_\theta R$ in Proposition \ref{proposition_writing_derivatives_R} and \eqref{equation_estimate_first_derivatives_R}, we find
    \begin{equation}
        \label{equation_proof_end_estimate_derivatives_R}
        |\partial_\theta R|(\Theta,\mathcal{A})=\frac{\mathcal{A}^0}{\mathcal{A}}\frac{U(\Theta,\mathcal{A})}{\mathcal{A}^0-\frac{1}{r}}\leq 1,\qquad |\partial_aR|(\Theta,\mathcal{A})\lesssim \frac{1}{\mathcal{A}}\log\langle r\mathcal{A}^2\rangle+ \frac{p(\mathcal{A},\ell)}{\mathcal{A}}\log\left\langle \frac{r}{p(\mathcal{A,\ell})}\right\rangle\lesssim \sqrt{r}+r^2.
    \end{equation}
    Then, recalling the expression of $U$ from \eqref{equation_definition_action_angle_R_U}, we find
    \begin{align*}
        \partial_\theta U(\theta,a)&=\frac{\theta}{|\theta|}\frac{\frac{\partial_\theta R}{R^2}\left(a^0-\frac{1}{R}\right)+\frac{\ell}{R^3}\partial_\theta R}{\sqrt{\left(a^0-\frac{1}{R}\right)^2-m^2-\frac{\ell}{R^2}}}.
    \end{align*}
    Then note that, by \eqref{equation_bounds_A_Theta_proof}, $\mathcal{A}^0-\frac{1}{r}\lesssim 1+|u|+\frac{\sqrt{\ell}}{r}+\frac{1}{r}$. Consequently, since $r\geq \frac{\mathcal{A}^0}{\mathcal{A}}$, $r\geq \frac{\sqrt{\ell}}{a}$, using \eqref{equation_lower_bound_R_a0_proof} and \eqref{equation_proof_end_estimate_derivatives_R}, we obtain
    \begin{align*}
         |\partial_\theta U|(\Theta,\mathcal{A})\leq \frac{1}{r^2}\left(\frac{\mathcal{A}^0}{\mathcal{A}}+\frac{\ell}{\mathcal{A}}\frac{1}{r-\frac{1}{\mathcal{A}^0}}\right)\leq 2\frac{\mathcal{A}}{r}\lesssim \frac{1}{r}+\frac{|u|}{r}+\frac{\sqrt{\ell}}{r^2}+\frac{1}{r^2}.
    \end{align*}
    For the derivative with respect to $a$, using the expression of $\partial_\theta R$ in Proposition \ref{proposition_writing_derivatives_R}, we note that
    \begin{equation}
        \label{equation_new_def_U_proof}
        U(\theta,a)=\frac{\theta}{|\theta|}\frac{pa}{R}\sqrt{\left(\frac{R}{p}-\kappa\right)^2-1}=\frac{\theta}{|\theta|}\frac{pa}{R}\sqrt{H_a^2\left(\frac{|\theta|}{p}\right)-1}.
    \end{equation}
    Consequently, since $R=pH_a\left(\frac{|\theta|}{p}\right)+\kappa p$, using \eqref{equation_derivatives_G_a_H_a}, we compute
    \begin{align}
        \notag \partial_a U(\theta,a)&=\frac{\theta}{|\theta|}\partial_a\left(\frac{p a}{R}\right)\sqrt{H_a^2\left(\frac{|\theta|}{p}\right)-1}-\frac{\theta}{|\theta|}\frac{a}{R}\frac{\partial_a p}{p}\frac{|\theta|H_a'\left(\frac{|\theta|}{p}\right)H_a\left(\frac{|\theta|}{p}\right)}{\sqrt{H_a^2\left(\frac{|\theta|}{p}\right)-1}}+\frac{\theta}{|\theta|}\frac{p a}{R}\frac{(\partial_a H_a)\left(\frac{|\theta|}{p}\right)H_a\left(\frac{|\theta|}{p}\right)}{\sqrt{H_a^2\left(\frac{|\theta|}{p}\right)-1}}\\
       \notag &=\left(\frac{1}{a}-\frac{\partial_a p}{p}-\frac{\partial_a R}{R}\right)U-\frac{a\partial_a p}{R}\frac{\theta H_a\left(\frac{|\theta|}{p}\right)}{R-\frac{1}{a^0}}-\frac{\theta}{|\theta|}\left(\kappa p-\frac{1}{a^0}\right)\frac{a}{R}\arcosh\left(H_a\left(\frac{|\theta|}{p}\right)\right)\frac{p}{R-\frac{1}{a^0}}.
    \end{align}
    For the first term, by \eqref{equation_estimates_derivatives_kappa_p} and Proposition \ref{proposition_writing_derivatives_R}, we have
    \begin{equation*}
        \left(\frac{1}{a}+\frac{|\partial_a p|}{p}+\frac{|\partial_a R|}{R}\right)|U|\lesssim \left(\frac{1}{a}+\frac{1}{R a}\log\langle Ra^2\rangle +\frac{p}{R a}\log \left\langle \frac{R}{p}\right\rangle\right)|U|\lesssim |U|\left(\frac{1}{a}+\frac{1}{\sqrt{R}}\right)\lesssim |U|(R+R^{-\frac{1}{2}}).
    \end{equation*}
    Similarly, since $|\theta|\leq 2R$ and $H_a\left(\frac{|\theta|}{p}\right)\leq \frac{R}{p}$, we find, using \eqref{equation_lower_bound_R_a0_proof},
    \begin{equation*}
        \frac{a|\partial_a p|}{R}\frac{|\theta| H_a\left(\frac{|\theta|}{p}\right)}{R-\frac{1}{a^0}}=\frac{aa^0|\partial_a p|}{R^2}\frac{|\theta| H_a\left(\frac{|\theta|}{p}\right)}{a^0-\frac{1}{R}}\leq 2aa^0\frac{|\partial_a p|}{p}\frac{1}{a^0-\frac{1}{R}}\lesssim a^0.
    \end{equation*}
    Finally, for the last term, since $\kappa p-\frac{1}{a^0}=\frac{m^2}{a^2a^0}$ and $\arcosh(x)\leq \sqrt{x^2-1}$ for $x\geq 1$, we derive
    \begin{equation*}
        \left(\kappa p-\frac{1}{a^0}\right)\frac{a}{R}\arcosh\left(H_a\left(\frac{|\theta|}{p}\right)\right)\frac{p}{R-\frac{1}{a^0}}\leq \frac{m^2}{R a a^0}p \frac{\sqrt{H_a^2\left(\frac{|\theta|}{p}\right)-1}}{R-\frac{1}{a^0}}\leq \frac{m}{Ra}|\partial_\theta R|\leq m.
    \end{equation*}
    The last three estimates imply
    \begin{equation*}
        |\partial_a U|(\Theta,\mathcal{A})\lesssim 1+(r^{-\frac{1}{2}}+r)|u|+\mathcal{A}^0\lesssim 1+(r^{-\frac{1}{2}}+r)|u|+\frac{\sqrt{\ell}}{r}+\frac{1}{r}.
    \end{equation*}
    Consequently,
    \begin{align*}
        \|(a^{30}+a^{-30}+|\theta|^{10}+\ell^{20})(\gamma+|\partial_a\gamma|+|\partial_\theta\gamma|)\|_{L^1}\lesssim&~ \|(r^{-65}+r^{35}+|u|^{35}+\ell^{45})(f_0+|\partial_r f_0|+|\partial_u f_0|)\|_{L^1}.
    \end{align*}
    In consequence, if the term on the right-hand side is sufficiently small, we obtain the smallness assumption \eqref{equation_smallness_condition_idea} and derive a global solution $\gamma$ to \eqref{equation_non_linear_gamma}. It also satisfies the scattering property \eqref{equation_scattering_gamma_idea}. By composing with the action-angle variables, we exhibit the modified scattering behavior of $f$
    \begin{equation}
        \label{equation_exact_modified_scattering_f}
        f\left(t,R\left(\Theta+t\widehat{\mathcal{A}}-\frac{m^2}{\mathcal{A}^2\mathcal{A}^0}\mathcal{E}_\infty(\mathcal{A})\log(t),\mathcal{A}\right), U\left(\Theta+t\widehat{\mathcal{A}}-\frac{m^2}{\mathcal{A}^2\mathcal{A}^0}\mathcal{E}_\infty(\mathcal{A})\log(t),\mathcal{A}\right), \ell\right)\xrightarrow[t\rightarrow+\infty]{L^1}\gamma_\infty(\Theta,\mathcal{A},\ell).
    \end{equation}
    Finally, we recover a solution to \eqref{equation_VM_pc_radial_case} defined for $t\geq 0$ by considering
    \begin{equation*}
        \widetilde{f}(t,r,u,\ell):=f(t+T_0,r,u,\ell).
    \end{equation*}
\end{proof}


\printbibliography

@Article{Pausader_Widmayer_2021,
author={Pausader, Benoit
and Widmayer, Klaus},
title={Stability of a Point Charge for the Vlasov--Poisson System: The Radial Case},
journal={Communications in Mathematical Physics},
year={2021},
month={Aug},
day={01},
volume={385},
number={3},
pages={1741-1769}
}

@Article{Horst_1990,
author={Horst, E.},
title={Symmetric plasmas and their decay},
journal={Communications in Mathematical Physics},
year={1990},
month={Jan},
day={01},
volume={126},
number={3},
pages={613-633}
}

@misc{kepka_widmayer_2025,
      title={Modified scattering dynamics in the Vlasov-Poisson equation near an attractive point mass}, 
      author={Bernhard Kepka and Klaus Widmayer},
      year={2025},
      eprint={2511.04363},
      archivePrefix={arXiv},
      primaryClass={math.AP},
      url={https://arxiv.org/abs/2511.04363}, 
}

@book{Griffiths_intro_EM,
    author = {Griffiths, David J},
    title = {Introduction to electrodynamics},
    publisher ={Pearson} ,
    year ={1981} 
}

@article{Glassey_Strauss_1987,
  title = {Absence of Shocks in an Initially Dilute Collisionless Plasma},
  author = {Glassey, Robert T. and Strauss, Walter A.},
  year = 1987,
  month = jun,
  journal = {Communications in Mathematical Physics},
  volume = {113},
  number = {2},
  pages = {191--208}
}

@misc{Breton_2025_absence_linear,
      title={A note on the non-$L^1$ asymptotic completeness of the Vlasov-Maxwell system}, 
      author={Emile Breton},
      year={2025},
      eprint={2509.04025},
      archivePrefix={arXiv},
      primaryClass={math.AP},
      url={https://arxiv.org/abs/2509.04025}, 
}

@article{elskens_Kiessling_2020,
  title = {Microscopic {{Foundations}} of {{Kinetic Plasma Theory}}: {{The Relativistic Vlasov}}--{{Maxwell Equations}} and {{Their Radiation-Reaction-Corrected Generalization}}},
  shorttitle = {Microscopic {{Foundations}} of {{Kinetic Plasma Theory}}},
  author = {Elskens, Y. and Kiessling, M. K.-H.},
  year = 2020,
  month = sep,
  journal = {Journal of Statistical Physics},
  volume = {180},
  number = {1},
  pages = {749--772}
}

@incollection{kiessling_2012,
  title = {On the {{Motion}} of {{Point Defects}} in {{Relativistic Fields}}},
  booktitle = {Quantum {{Field Theory}} and {{Gravity}}: {{Conceptual}} and {{Mathematical Advances}} in the {{Search}} for a {{Unified Framework}}},
  author = {Kiessling, Michael K.-H.},
  year = 2012,
  pages = {299--335},
  publisher = {Springer}
}

@book{klimontovich_1967,
  title={The Statistical Theory of Non-equilibrium Processes in a Plasma},
  author={Klimontovich, I.U.L.},
  lccn={67007161},
  series={International series in natural philosophy},
  year={1967},
  publisher={M.I.T. Press}
}

@article{caprino_marchioro_2010,
  title = {On the Plasma-Charge Model},
  author = {Caprino, Silvia and Marchioro, Carlo},
  year = {2010},
  journal = {Kinetic and Related Models},
  volume = {3},
  number = {2},
  pages = {241--254}
}

@article{Wang_2022,
  title = {Propagation of {{Regularity}} and {{Long Time Behavior}} of the {{3D Massive Relativistic Transport Equation II}}: {{Vlasov}}--{{Maxwell System}}},
  shorttitle = {Propagation of {{Regularity}} and {{Long Time Behavior}} of the {{3D Massive Relativistic Transport Equation II}}},
  author = {Wang, Xuecheng},
  year = {2022},
  month = {01},
  journal = {Communications in Mathematical Physics},
  volume = {389},
  number = {2},
  pages = {715--812}
}

@article{Bigorgne_sharp_2020,
  title = {Sharp {{Asymptotic Behavior}} of {{Solutions}} of the 3d {{Vlasov}}--{{Maxwell System}} with {{Small Data}}},
  author = {Bigorgne, L{\'e}o},
  year = {2020},
  month = {06},
  journal = {Communications in Mathematical Physics},
  volume = {376},
  number = {2},
  pages = {893--992}
}

@article{wei_yang_2021,
  title = {On the {{3D Relativistic Vlasov-Maxwell System}} with {{Large Maxwell Field}}},
  author = {Wei, Dongyi and Yang, Shiwu},
  year = {2021},
  month = {05},
  journal = {Communications in Mathematical Physics},
  volume = {383},
  number = {3},
  pages = {2275--2307}
}

@article{pankavich_ben-artzi_2025,
  title = {Modified Scattering of Solutions to the Relativistic {{Vlasov}}--{{Maxwell}} System inside the Light Cone},
  author = {Pankavich, Stephen and {Ben-Artzi}, Jonathan},
  year = {2025},
  journal = {Journal of the London Mathematical Society},
  volume = {112},
  number = {5}
}

@article{breton_modified_2026,
  title = {Modified {{Scattering}} for {{Small Data Solutions}} to the {{Vlasov}}--{{Maxwell System}}: {{A Short Proof}}},
  shorttitle = {Modified {{Scattering}} for {{Small Data Solutions}} to the {{Vlasov}}--{{Maxwell System}}},
  author = {Breton, Emile},
  year = {2026},
  month = {06},
  journal = {Asymptotic Analysis},
  volume = {148},
  number = {2},
  pages = {707--725}
}

@article{bigorgne_modified_2025,
  title = {Global Existence and Modified Scattering for the Solutions to the {{Vlasov}}--{{Maxwell}} System with a Small Distribution Function},
  author = {Bigorgne, L{\'e}o},
  year = {2025},
  month = {03},
  journal = {Analysis \& PDE},
  volume = {18},
  number = {3},
  pages = {629--714}
}

@misc{bigorgne_ScatteringMap_2023,
  title = {Scattering Map for the {{Vlasov-Maxwell}} System around Source-Free Electromagnetic Fields},
  author = {Bigorgne, L{\'e}o},
  year = {2023},
  month = {12},
  number = {arXiv:2312.12214},
  eprint = {2312.12214},
  primaryclass = {math.AP},
  publisher = {arXiv},
  doi = {10.48550/arXiv.2312.12214}
}

@article{pankavich_2021,
  title = {Exact {{Large Time Behavior}} of {{Spherically Symmetric Plasmas}}},
  author = {Pankavich, Stephen},
  year = {2021},
  month = {01},
  journal = {SIAM Journal on Mathematical Analysis},
  volume = {53},
  number = {4},
  pages = {4474--4512}
}

@article{pankavichAsymptoticDynamicsDispersive2022,
  title = {Asymptotic {{Dynamics}} of {{Dispersive}}, {{Collisionless Plasmas}}},
  author = {Pankavich, Stephen},
  year = {2022},
  month = {12},
  journal = {Communications in Mathematical Physics},
  volume = {391},
  number = {2},
  pages = {455--493}
}

@online{bigorgneHomeomorphicModifiedWave2026,
  title = {Homeomorphic Modified Wave Operators for the {{Vlasov-Poisson}} System},
  author = {Bigorgne, Léo},
  date = {2026-06-04},
  eprint = {2606.06488},
  eprinttype = {arXiv},
  eprintclass = {math.AP}
}

@article{choiModifiedScatteringVlasov2016,
  title = {Modified Scattering for the {{Vlasov}}–{{Poisson}} System},
  author = {Choi, Sun-Ho and Kwon, Soonsik},
  date = {2016-08},
  journaltitle = {Nonlinearity},
  shortjournal = {Nonlinearity},
  volume = {29},
  number = {9},
  pages = {2755},
  publisher = {IOP Publishing}
}

@article{flynnScatteringMapVlasov2023,
  title = {Scattering {{Map}} for the {{Vlasov}}–{{Poisson System}}},
  author = {Flynn, Patrick and Ouyang, Zhimeng and Pausader, Benoit and Widmayer, Klaus},
  date = {2023-09-01},
  journaltitle = {Peking Mathematical Journal},
  shortjournal = {Peking Math J},
  volume = {6},
  number = {2},
  pages = {365--392},
  issn = {2524-7182}
}

@article{ionescuAsymptoticBehaviorSolutions2022,
  title = {On the {{Asymptotic Behavior}} of {{Solutions}} to the {{Vlasov}}–{{Poisson System}}},
  author = {Ionescu, Alexandru D and Pausader, Benoit and Wang, Xuecheng and Widmayer, Klaus},
  date = {2022-05-17},
  journaltitle = {International Mathematics Research Notices},
  shortjournal = {Int Math Res Notices},
  volume = {2022},
  number = {12},
  pages = {8865--8889}
}

@article{choiAsymptoticBehaviorNonlinear2011,
  title = {Asymptotic {{Behavior}} of the {{Nonlinear Vlasov Equation}} with a {{Self-Consistent Force}}},
  author = {Choi, Sun-Ho and Ha, Seung-Yeal},
  date = {2011-01},
  journaltitle = {SIAM Journal on Mathematical Analysis},
  shortjournal = {SIAM J. Math. Anal.},
  volume = {43},
  number = {5},
  pages = {2050--2077},
  publisher = {{Society for Industrial and Applied Mathematics}}
}

@article{caprinoAttractivePlasmaChargeSystem2012,
  title = {On the {{Attractive Plasma-Charge System}} in 2-d},
  author = {Caprino, S. and Marchioro, C. and Miot, E. and Pulvirenti, M.},
  date = {2012-07-01},
  journaltitle = {Communications in Partial Differential Equations},
  volume = {37},
  number = {7},
  pages = {1237--1272}
}

@article{marchioroCauchyProblem3D2011,
  title = {The {{Cauchy Problem}} for the 3-{{D Vlasov}}–{{Poisson System}} with {{Point Charges}}},
  author = {Marchioro, Carlo and Miot, Evelyne and Pulvirenti, Mario},
  date = {2011-07-01},
  journaltitle = {Archive for Rational Mechanics and Analysis},
  shortjournal = {Arch Rational Mech Anal},
  volume = {201},
  number = {1},
  pages = {1--26}
}

@article{desvillettesPolynomialPropagationMoments2015,
  title = {Polynomial Propagation of Moments and Global Existence for a {{Vlasov}}–{{Poisson}} System with a Point Charge},
  author = {Desvillettes, Laurent and Miot, Evelyne and Saffirio, Chiara},
  date = {2015-04-01},
  journaltitle = {Annales de l'Institut Henri Poincaré C},
  volume = {32},
  number = {2},
  pages = {373--400}
}

@article{wuPolynomialPropagationMoments2021,
  title = {Polynomial Propagation of Moments for a {{Plasma-Charge}} Model with Large Data},
  author = {Wu, Jingpeng and Zhang, Xianwen},
  date = {2021-04-01},
  journaltitle = {Applied Mathematics Letters},
  shortjournal = {Applied Mathematics Letters},
  volume = {114},
  pages = {106890}
}

@article{wuPlasmachargeModelConvex2024,
  title = {The Plasma-Charge Model in a Convex Domain},
  author = {Wu, Jingpeng},
  date = {2024-03},
  journaltitle = {Nonlinearity},
  shortjournal = {Nonlinearity},
  volume = {37},
  number = {5},
  pages = {055003},
  publisher = {IOP Publishing}
}

@online{wuPlasmaChargeModelBoundary2026,
  title = {The {{Plasma-Charge Model}}: {{Boundary Effects}} and {{Global Well-posedness}}},
  shorttitle = {The {{Plasma-Charge Model}}},
  author = {Wu, Jingpeng},
  date = {2026-04-13},
  eprint = {2604.11371},
  eprinttype = {arXiv},
  eprintclass = {math.AP}
}

@article{pausaderStabilityPointCharge2024,
  title = {Stability of a Point Charge for the Repulsive {{Vlasov}}–{{Poisson}} System},
  author = {Pausader, Benoît and Widmayer, Klaus and Yang, Jiaqi},
  date = {2024-08-30},
  journaltitle = {Journal of the European Mathematical Society},
  volume = {28},
  number = {7},
  pages = {2751--2848}
}

@online{chaturvediLinearNonlinearPhase2026,
  title = {Linear and Nonlinear Phase Mixing for the Gravitational {{Vlasov-Poisson}} System under an External {{Kepler}} Potential},
  author = {Chaturvedi, Sanchit and Luk, Jonathan},
  date = {2026-05-04},
  eprint = {2409.14626},
  eprinttype = {arXiv},
  eprintclass = {math.AP}
}

@article{onemSolutionsClassicalRelativistic1998,
  title = {The {{Solutions}} of the {{Classical Relativistic Two-Body Equation}}},
  author = {ÖNEM, Coşkun},
  date = {1998-01-01},
  journaltitle = {Turkish Journal of Physics},
  volume = {22},
  number = {2},
  pages = {107--114}
}

@Article{Luk_Strain_14,
author={Luk, Jonathan
and Strain, Robert M.},
title={Strichartz Estimates and Moment Bounds for the Relativistic Vlasov--Maxwell System},
journal={Archive for Rational Mechanics and Analysis},
year={2016},
month={01},
day={01},
volume={219},
number={1},
pages={445-552}
}

\end{document}